\documentclass[times]{speauth}

\usepackage{moreverb}

\usepackage{amsfonts,amssymb,amsthm,newlfont,graphicx,subfigure,algorithmic,algorithm,fixltx2e,booktabs,bm}
\usepackage[numbers]{natbib}
\usepackage{epstopdf}
\usepackage{empheq} 
\usepackage[colorlinks,bookmarksopen,bookmarksnumbered,citecolor=red,urlcolor=red]{hyperref}
\numberwithin{algorithm}{section}

\newcommand\BibTeX{{\rmfamily B\kern-.05em \textsc{i\kern-.025em b}\kern-.08em
T\kern-.1667em\lower.7ex\hbox{E}\kern-.125emX}}

\def\volumeyear{2015}
\newtheorem{thm}{Theorem}[section]
\newtheorem{lem}{Lemma}[section]
\newtheorem{definition}{Definition}[section]
\newtheorem{cor}{Corollary}[section]
\newtheorem{rem}{Remark}[section]

\numberwithin{equation}{section}
\renewcommand{\theequation}{\thesection.\arabic{equation}}
\newcommand\numberthis{\addtocounter{equation}{1}\tag{\theequation}}
\def\simgt{\,\hbox{\lower0.6ex\hbox{$>$}\llap{\raise0.3ex\hbox{$\sim$}}}\,}
\def\simlt{\,\hbox{\lower0.6ex\hbox{$<$}\llap{\raise0.3ex\hbox{$\sim$}}}\,}
\def\simgteq{\,\hbox{\lower0.6ex\hbox{$\ge$}\llap{\raise0.6ex\hbox{$\sim$}}}\,}
\def\simlteq{\,\hbox{\lower0.6ex\hbox{$\le$}\llap{\raise0.6ex\hbox{$\sim$}}}\,}
\makeatletter
\def\user@resume{resume}
\def\user@intermezzo{intermezzo}
\newcounter{previousequation}
\newcounter{lastsubequation}
\newcounter{savedparentequation}
\renewenvironment{subequations}[1][]{%
      \def\user@decides{#1}%
      \setcounter{previousequation}{\value{equation}}%
      \ifx\user@decides\user@resume
           \setcounter{equation}{\value{savedparentequation}}%
      \else
      \ifx\user@decides\user@intermezzo
           \refstepcounter{equation}%
      \else
           \setcounter{lastsubequation}{0}%
           \refstepcounter{equation}%
      \fi\fi
      \protected@edef\theHparentequation{%
          \@ifundefined {theHequation}\theequation \theHequation}%
      \protected@edef\theparentequation{\theequation}%
      \setcounter{parentequation}{\value{equation}}%
      \ifx\user@decides\user@resume
           \setcounter{equation}{\value{lastsubequation}}%
         \else
           \setcounter{equation}{0}%
      \fi
      \def\theequation  {\theparentequation  \alph{equation}}%
      \def\theHequation {\theHparentequation \alph{equation}}%
      \ignorespaces
}{%
  \ifx\user@decides\user@resume
       \setcounter{lastsubequation}{\value{equation}}%
       \setcounter{equation}{\value{previousequation}}%
  \else
  \ifx\user@decides\user@intermezzo
       \setcounter{equation}{\value{parentequation}}%
  \else
       \setcounter{lastsubequation}{\value{equation}}%
       \setcounter{savedparentequation}{\value{parentequation}}%
       \setcounter{equation}{\value{parentequation}}%
  \fi\fi
  \ignorespacesafterend
}
\makeatother

\begin{document}

\runningheads{Kareem T. Elgindy}{High-Order Numerical Solution of the Telegraph Equation}

\title{High-Order Numerical Solution of Second-Order One-Dimensional Hyperbolic Telegraph Equation Using a Shifted Gegenbauer Pseudospectral Method}

\author{Kareem T. Elgindy\corrauth}

\address{Mathematics Department, Faculty of Science, Assiut University, Assiut 71516, Egypt}

\corraddr{Mathematics Department, Faculty of Science, Assiut University, Assiut 71516, Egypt}

\begin{abstract}
We present a high-order shifted Gegenbauer pseudospectral method (SGPM) to solve numerically the second-order one-dimensional hyperbolic telegraph equation provided with some initial and Dirichlet boundary conditions. The framework of the numerical scheme involves the recast of the problem into its integral formulation followed by its discretization into a system of well-conditioned linear algebraic equations. The integral operators are numerically approximated using some novel shifted Gegenbauer operational matrices of integration. We derive the error formula of the associated numerical quadratures. We also present a method to optimize the constructed operational matrix of integration by minimizing the associated quadrature error in some optimality sense. We study the error bounds and convergence of the optimal shifted Gegenbauer operational matrix of integration. Moreover, we construct the relation between the operational matrices of integration of the shifted Gegenbauer polynomials and standard Gegenbauer polynomials. We derive the global collocation matrix of the SGPM, and construct an efficient computational algorithm for the solution of the collocation equations. We present a study on the computational cost of the developed computational algorithm, and a rigorous convergence and error analysis of the introduced method. Four numerical test examples have been carried out in order to verify the effectiveness, the accuracy, and the exponential convergence of the method. The SGPM is a robust technique, which can be extended to solve a wide range of problems arising in numerous applications.
\end{abstract}


\maketitle

\vspace{-6pt}

\section{Introduction}
\label{int}
\vspace{-2pt}
Second-order hyperbolic partial differential equations (PDEs) have been studied for many decades, as they frequently arise in many applications like seismology, acoustics, general relativity, oceanography, electromagnetics, electrodynamics, thermoelasticity, thermodynamics of thermal waves, fluid dynamics, reaction-diffusion processes, materials science, geophysics, biological systems, ecology, and a host of other important areas; cf. \cite{Guddati1999,Kreiss2002,Ramos2007,Mattsson2009,Ashyralyev2010}. The range and significance of their applications manifest the demand for achieving higher-order numerical approximations using robust and efficient numerical schemes. In the present work, we establish a high-order numerical approximation to the solution of the following second-order one-dimensional hyperbolic telegraph equation:
\begin{subequations}
\begin{equation}\label{int:eq:teleg1}
{u_{tt}} + {\beta _1}\,{u_t} + {\beta _2}\,u = {u_{xx}} + f(x,t),\quad 0 < x < l,\quad 0 < t < \tau ,
\end{equation}
\text{provided with the initial conditions given by}
\begin{align}
	u(x,0) &= {g_1}(x),\quad 0 < x < l,\label{int:eq:teleg12}\\
	{u_t}(x,0) &= {g_2}(x),\quad 0 < x < l\label{int:eq:teleg122},
\end{align}
\text{and the following Dirichlet boundary conditions}
\begin{align}
	u(0,t) &= {h_1}(t),\quad 0 < t \le \tau,\label{int:dbc1}\\
	u(l,t) &= {h_2}(t),\quad 0 < t \le \tau,\label{int:dbc2}
\end{align}
\end{subequations}
where $u$ is the unknown solution function, $f$ is a given integrable function; $g_1, g_2, h_1$, and $h_2$ are some given functions; $\beta_1$ and $\beta_2$ are some known constant coefficients. We shall refer to Equation \eqref{int:eq:teleg1} provided with Conditions \eqref{int:eq:teleg12} - \eqref{int:dbc2} by Problem $\mathcal{P}$. In fact, the telegraph equation \eqref{int:eq:teleg1} models an infinitesimal piece of a telegraph wire as an electrical circuit, and it describes the voltage and current in a double conductor with distance $x$ and time $t$ \citep{Pandit2015}. The telegraph equation is in particular important as it is commonly used in the study and modeling of signal analysis for transmission and propagation of electrical signals in a cable transmission line \cite{Aloy2007,Sari2014}, and in reaction diffusion occurring in many branches of sciences \cite{Abdusalam2004,Ahmed2001}.

The numerical solution of second order hyperbolic PDEs has been studied extensively by a variety of techniques such as the finite element methods \cite{Bales1984,Bales1986}, finite-difference schemes \cite{Day1966,Jovanovic1987,Mohanty2005,Ramos2007}, combined finite difference scheme and Haar wavelets \cite{Pandit2015}, discrete eigenfunctions method \cite{Aloy2007}, Legendre multiwavelet approximations \cite{Yousefi2010}, the singular dynamic method \cite{Renard2010}, interpolating scaling functions \cite{Lakestani2010}, cubic and quartic B-spline collocation methods \cite{Dosti2012,Mittal2013}, non-polynomial spline methods \cite{Ding2012}, the reduced differential transform method \cite{Srivastava2013}, etc. In the present work, we present a shifted Gegenbauer pseudospectral method (SGPM) for the solution of Problem $\mathcal{P}$. The numerical scheme exploits the stability and the well-conditioning of the numerical integral operators, and collocates the integral formulation of Problem $\mathcal{P}$ in the physical (nodal) space using some novel operational matrices of integration (also called integration matrices) based on shifted Gegenbauer polynomials. The proposed method leads to well-conditioned linear system of algebraic equations, which can be solved efficiently using standard linear system solvers. The rapid convergence, economy in calculations, memory minimization, and the simplicity in programming and application are some of the features enjoyed by the present method. The current work is an extension to the works of \citet{Elgindy2013} and \citet{Elgindy2013a} to second-order hyperbolic PDEs using shifted Gegenbauer polynomials.

The rest of the article is organized as follows: In Section \ref{sec:pre}, we give some basic preliminaries relevant to Gegenbauer and shifted Gegenbauer polynomials. In Section \ref{sec:ort}, we derive the Lagrange form of the shifted Gegenbauer interpolation at the shifted Gegenbauer-Gauss (SGG) nodes. In Section \ref{sec:theshi}, we derive the shifted Gegenbauer integration matrix and its associated quadrature error formula in Section \ref{subsec:err}. In Section \ref{subsec:opt}, we construct an optimal shifted Gegenbauer integration matrix in some optimality measure, and analyze its associated quadrature error in Section \ref{subsec:erranaotosq}. Section \ref{subsec:errbounds} gives the error bounds of the optimal shifted Gegenbauer quadrature. Section \ref{subsec:rel} presents the relation between the integration matrices of the shifted Gegenbauer polynomials and standard Gegenbauer polynomials. Section \ref{sec:theshi2} introduces the SGPM for the efficient numerical solution of Problem $\mathcal{P}$. Section \ref{subsec:mrhs11} is devoted for the constructions of the global collocation matrix and the right hand side of the collocation equations.
Section \ref{subsec:GAS1} establishes the global approximate interpolant over the whole solution domain. Section \ref{sec:conerr1} is devoted for the study of the convergence and error analysis of the proposed method. Four numerical test examples are studied in Section \ref{sec:numerical} to assess the efficiency and accuracy of the numerical scheme. We provide some concluding remarks in Section \ref{sec:conc} followed by possible future work in Section \ref{sec:fw}. Finally,  Appendix \ref{appendix:comalg11} establishes an efficient computational algorithm for the constructions of the global collocation matrix and the right hand side of the collocation equations.

\vspace{-6pt}
\section{Preliminaries}
\label{sec:pre}
\vspace{-2pt}
In this section, we present some preliminary properties of the Gegenbauer polynomials and the shifted Gegenbauer polynomials defined on one and two dimensions. Moreover, we present the discrete inner product of any two functions for the shifted Gegenbauer approximations.

The Gegenbauer polynomial $C_n^{(\alpha )}(x)$, of degree $n \in \mathbb{Z}^+$, and associated with the parameter $\alpha > -1/2$, is a real-valued function, which appears as an eigensolution to a singular Sturm-Liouville problem in the finite domain $[-1, 1]$ \cite{Szego1975}. It is a Jacobi polynomial, $P_n^{(\alpha, \beta)}$, with $\alpha = \beta$, and can be standardized so that:
\begin{equation}\label{sec:pre:eq:normaliz1}
C_n^{(\alpha )}(x) = \frac{{n! \Gamma (\alpha  + \tfrac{1}{2})}}{{\Gamma (n + \alpha  + \tfrac{1}{2})}}P_n^{(\alpha  - 1/2, \alpha  - 1/2)}(x),\quad n = 0,1,2, \ldots.
\end{equation}
Therefore, we recover the $n$th-degree Chebyshev polynomial of the first kind, $T_n(x)$, and the $n$th-degree Legendre polynomial, $L_n(x)$, for $\alpha = 0$ and $0.5$, respectively. The Gegenbauer polynomials can be generated by the following three-term recurrence equation:
\begin{subequations}
\begin{equation}
(n + 2  \alpha )  C_{n + 1}^{(\alpha )}(x) = 2  (n + \alpha )  x C_n^{(\alpha )}(x) - n  C_{n - 1}^{(\alpha )}(x), \quad n = 1, 2, 3, \ldots,
\end{equation}
starting with the following two equations:
\begin{equation}
C_0^{(\alpha )}(x) = 1,
\end{equation}
\begin{equation}
C_1^{(\alpha )}(x) = x,
\end{equation}
\end{subequations}
or in terms of the hypergeometric functions,
\begin{equation}\label{sec:pre:eq:hyp1}
	C_n^{(\alpha )}(x) = {_2}{F_1}\left( { - n,2\alpha  + n;\alpha  + \frac{1}{2};\frac{{1 - x}}{2}} \right), n = 0, 1, 2, \ldots,
\end{equation}
where $_2{F_1}\left( {a,b;c;x} \right),$ is the first hypergeometric function (Gauss's hypergeometric function), which converges if $c \notin \mathbb{Z^ - } \cup \{ 0\} ,$ for all $\left| x \right| < 1$, or at the endpoints $x = \pm 1$, if ${\text{Re}}[c - a - b] > 0$. The leading coefficients of the Gegenbauer polynomials $C_n^{(\alpha)}(x)$, are denoted by $K_n^{(\alpha)}$, and are given by the following relation:
\begin{equation}\label{sec:pre:eq:lead1}
K_n^{(\alpha )} = {2^{n-1}}\frac{{\Gamma (n + \alpha )\;\Gamma (2\alpha  + 1)}}{{\Gamma (n + 2\alpha )\;\Gamma (\alpha  + 1)}},\quad n = 0, 1, 2, \ldots.
\end{equation}
The weight function for the Gegenbauer polynomials is the even function $w^{(\alpha)}(x) = {(1 - {x^2})^{\alpha - 1/2}}$. The Gegenbauer polynomials form a complete orthogonal basis polynomials in $L_{w^{(\alpha)}}^2[-1, 1]$, and their orthogonality relation is given by the following weighted inner product:
\begin{equation}
\left(C_m^{(\alpha)}, C_n^{(\alpha)}\right)_{w^{(\alpha)}} = \int_{ - 1}^1 {C_m^{(\alpha )}(x)\, C_n^{(\alpha )}(x)\, w^{(\alpha)}(x) dx}  = \left\| {C_n^{(\alpha )}} \right\|_{{w^{(\alpha )}}}^2 {\delta _{m,n}} = {\lambda}_n^{(\alpha )} {\delta _{m,n}},
\end{equation}
 where
\begin{equation}\label{sec:pre:eq:normak1}
{\lambda}_n^{(\alpha )} = \left\| {C_n^{(\alpha )}} \right\|_{{w^{(\alpha )}}}^2 = \frac{{{2^{1 - 2\alpha }}\,\pi\,   \Gamma (n + 2\alpha )}}{{n!\,(n + \alpha )\,{\Gamma ^2}(\alpha )}},
\end{equation}
is the normalization factor, and $\delta _{m,n}$ is the Kronecker delta function. We denote the zeroes of the Gegenbauer polynomial $C_{n+1}^{(\alpha )}(x)$ (also called Gegenbauer-Gauss nodes) by $x_{n,k}^{(\alpha)}, k = 0, \ldots, n$, and denote their set by $S_n^{(\alpha)}$. We also denote their corresponding Christoffel numbers by $\varpi_{n,k}^{(\alpha)}, k = 0, \ldots, n$, and define them by the following relation:
\begin{equation}
	{\left(\varpi _{n,k}^{(\alpha )}\right)^{ - 1}} = \sum\limits_{j = 0}^n {{{\left(\lambda _j^{(\alpha )}\right)}^{ - 1}}\,{{\left(C_j^{(\alpha )}\left(x_{n,k}^{(\alpha )}\right)\right)}^2}} ,\quad k = 0,1,2, \ldots ,n.
\end{equation}
Throughout the paper, we shall refer to the Gegenbauer polynomials by those constrained by standardization (\ref{sec:pre:eq:normaliz1}).

Let $L$ be some positive real number. The shifted Gegenbauer polynomial of degree $n$ on the interval $[0, L]$, is defined by $C_{L,n}^{(\alpha)}(x) = C_n^{(\alpha)}(2 x/L - 1)$. The shifted Gegenbauer polynomials form a complete $L_{w_L^{(\alpha)}}^2[-1, 1]$-orthogonal system with respect to the weight function,
\begin{equation}
	w_L^{(\alpha )}(x) = {(Lx - {x^2})^{\alpha  - 1/2}},
\end{equation}
 and their orthogonality relation is defined by the following weighted inner product:
\begin{equation}
\left(C_{L,m}^{(\alpha)}, C_{L,n}^{(\alpha)}\right)_{w_L^{(\alpha)}} = \int_{ 0}^L {C_{L,m}^{(\alpha )}(x)\, C_{L,n}^{(\alpha )}(x)\, w_L^{(\alpha)}(x) dx}  = \left\| {C_{L,n}^{(\alpha )}} \right\|_{{w_L^{(\alpha )}}}^2 {\delta _{m,n}} = {\lambda}_{L,n}^{(\alpha )} {\delta _{m,n}},
\end{equation}
where
\begin{equation}
{\lambda}_{L,n}^{(\alpha )} = {\left(\frac{L}{2}\right)^{2\alpha }}\lambda _n^{(\alpha )},
\end{equation}
is the normalization factor. For $\alpha = 0$ and $0.5$, we recover the shifted Chebyshev polynomials of the first kind and the shifted Legendre polynomials, respectively. We denote the zeroes of the shifted Gegenbauer polynomial $C_{L,n+1}^{(\alpha )}(x)$ (SGG nodes) by $x_{L,n,k}^{(\alpha)}, k = 0, \ldots, n$, and denote their set by $S_{L,n}^{(\alpha)}$. We also denote their corresponding Christoffel numbers by $\varpi_{L,n,k}^{(\alpha)}, k = 0, \ldots, n$. Clearly
\begin{equation}
	x_{L,n,k}^{(\alpha )} = \frac{L}{2}\left( {x_{n,k}^{(\alpha )} + 1} \right),\quad k = 0, \ldots ,n;
\end{equation}
\begin{equation}
	\varpi _{L,n,k}^{(\alpha )} = {\left( {\frac{L}{2}} \right)^{2\alpha }}\varpi _{n,k}^{(\alpha )},\quad k = 0, \ldots ,n.
\end{equation}
If we denote by $\mathbb{P}{_{n}}$, the space of all polynomials of degree at most $n, n \in \mathbb{Z}^+$, then for any $\phi  \in \mathbb{P}{_{2n + 1}}$,
\begin{align}
\int_0^L {\phi (x)\,w_L^{(\alpha )}(x)\,dx}  &= {\left( {\frac{L}{2}} \right)^{2\alpha }}\int_{ - 1}^1 {\phi \left( {\frac{L}{2}(x + 1)} \right)\,{w^{(\alpha )}}(x)\,dx} = {\left( {\frac{L}{2}} \right)^{2\alpha }}\sum\limits_{j = 0}^n {\varpi _{n,j}^{(\alpha )}\,\phi \left( {\frac{L}{2}\left(x_{n,j}^{(\alpha )} + 1\right)} \right)}\nonumber\\
&= \sum\limits_{j = 0}^n {\varpi _{L,n,j}^{(\alpha )}\,\phi \left( {x_{L,n,j}^{(\alpha )}} \right)},
\end{align}
using the standard Gegenbauer-Gauss quadrature. With the quadrature rule, we can define the discrete inner product $(\cdot, \cdot)_{L,n}$, of any two functions $u(x)$ and $v(x)$ defined on $[0, L]$, for the shifted Gegenbauer approximations as follows:
\begin{equation}
	{(u,v)_{L,n}} = \sum\limits_{j = 0}^n {\varpi _{L,n,j}^{(\alpha )}\,u\left( {x_{L,n,j}^{(\alpha )}} \right)\,v\left( {x_{L,n,j}^{(\alpha )}} \right)}.
\end{equation}
In two dimensions, we can define the bivariate shifted Gegenbauer polynomials by the following definition:
\begin{definition}
Let $\left\{ {C_{L,n}^{(\alpha )}(x)} \right\}_{n = 0}^\infty$, be a sequence of shifted Gegenbauer polynomials on $D_L = [0, L]$. The bivariate shifted Gegenbauer polynomials, $\left\{ {{}_{l,\tau }C_{n,m}^{(\alpha )}(x,t)} \right\}_{n,m = 0}^\infty$, are defined as
\begin{equation}
	{}_{l,\tau }C_{n,m}^{(\alpha )}(x,t) = C_{l,n}^{(\alpha )}(x)\,C_{\tau ,m}^{(\alpha )}(t),\;(x,t) \in D_{l,\tau }^2 = [0,l] \times [0,\tau ],\,l,\tau  \in {\mathbb{R}^ +}.
\end{equation}
\end{definition}
The family $\left\{ {_{l,\tau }C_{n,m}^{(\alpha )}(x,t)} \right\}_{n,m = 0}^\infty$, forms a complete basis for ${L^2}(D_{l,\tau }^2)$. They are orthogonal on ${L^2}(D_{l,\tau }^2)$ with respect to the weight function $w_{l,\tau }^{(\alpha )}(x,t) = w_l^{(\alpha )}(x)\,w_\tau ^{(\alpha )}(t),$ since
\begin{align}
	\int_0^l {\int_0^\tau  {{}_{l,\tau }C_{n,m}^{(\alpha )}(x,t)\,{}_{l,\tau }C_{s,k}^{(\alpha )}(x,t)\,w_{l,\tau }^{(\alpha )}(x,t)\,dx\,dt} }  &= \int_0^l {C_{l,n}^{(\alpha )}(x)\,C_{l,s}^{(\alpha )}(x)\,w_l^{(\alpha )}(x)\,dx}  \cdot \int_0^\tau  {C_{\tau ,m}^{(\alpha )}(t)\,C_{\tau ,k}^{(\alpha )}(t)\,w_\tau ^{(\alpha )}(t)\,dt} \\
	&= \left\{ \begin{array}{l}
{}_{l,\tau }\lambda _{n,m}^{(\alpha )},\quad (n,m) = (s,k),\\
0,\quad {\text{o.w.,}}
\end{array} \right.
\end{align}
where
\begin{equation}
	{}_{l,\tau }\lambda _{n,m}^{(\alpha )} = \lambda _{l,n}^{(\alpha )}{\mkern 1mu} \lambda _{\tau ,m}^{(\alpha )}\,\forall n,m.
\end{equation}

In the next section, we highlight the modal and nodal orthogonal shifted Gegenbauer interpolation, and derive the Lagrange form of the shifted Gegenbauer interpolation at the SGG nodes.

\vspace{-6pt}
\section{Orthogonal Shifted Gegenbauer Interpolation}
\label{sec:ort}
\vspace{-2pt}
The function
\begin{equation}\label{sec:ort:eq:modal}
{P_n}f(x) = \sum\limits_{j = 0}^n {{{\tilde f}_j}\,C_{L,j}^{(\alpha )}(x)} ,
\end{equation}
is the shifted Gegenbauer interpolant of a real function $f$ defined on $[0, L]$, if we compute the coefficients ${{\tilde f}_j}$ so that
\begin{equation}
	{P_n}f({x_k}) = f({x_k}),\quad k = 0, \ldots ,n,
\end{equation}
for some nodes $x_k \in [0, L], k = 0, \ldots ,n$. If we choose the interpolation points $x_k$ to be the SGG nodes, then we can simply compute the discrete coefficients using the discrete inner product created from the Gegenbauer-Gauss quadrature by the following formula:
\begin{equation}\label{sec:ort:eq:sgt}
{{\tilde f}_j} = \frac{{{{({P_n}f,C_{L,j}^{(\alpha )})}_{L,n}}}}{{\left\| {C_{L,j}^{(\alpha )}} \right\|_{w_L^{(\alpha )}}^2}} = \frac{{{{(f,C_{L,j}^{(\alpha )})}_{L,n}}}}{{\left\| {C_{L,j}^{(\alpha )}} \right\|_{w_L^{(\alpha )}}^2}} = \frac{1}{{\lambda _{L,j}^{(\alpha )}}}\sum\limits_{k = 0}^n {\varpi _{L,n,k}^{(\alpha )}\,f_{L,n,k}^{(\alpha )}\,C_{L,j}^{(\alpha )}\left(x_{L,n,k}^{(\alpha )}\right)},\quad j = 0, \ldots ,n,
\end{equation}
where $f_{L,n,k}^{(\alpha )} = f\left(x_{L,n,k}^{(\alpha )}\right)\, \forall k$. Equation \eqref{sec:ort:eq:sgt} gives the discrete shifted Gegenbauer transform. To construct the shifted Gegenbauer integration matrix, we need to represent the
orthogonal shifted Gegenbauer approximation as an interpolant through a set of node values (nodal approximation) instead of the modal approximation given by Equation \eqref{sec:ort:eq:modal}. Substituting Equation \eqref{sec:ort:eq:sgt} into \eqref{sec:ort:eq:modal} yields
\begin{align}
	{P_n}f(x) &= \sum\limits_{j = 0}^n {\frac{1}{{\lambda _{L,j}^{(\alpha )}}}\sum\limits_{k = 0}^n {\varpi _{L,n,k}^{(\alpha )}\,{f_{L,n,k}^{(\alpha)}}\,C_{L,j}^{(\alpha )}\left(x_{L,n,k}^{(\alpha )}\right)} \,C_{L,j}^{(\alpha )}(x)}\nonumber\\
	&=	\sum\limits_{k = 0}^n {\left[ {\varpi _{L,n,k}^{(\alpha )}\sum\limits_{j = 0}^n {{{\left( {\lambda _{L,j}^{(\alpha )}} \right)}^{ - 1}}\,C_{L,j}^{(\alpha )}\left(x_{L,n,k}^{(\alpha )}\right)\,C_{L,j}^{(\alpha )}(x)} } \right]\,} {f_{L,n,k}^{(\alpha)}}.
\end{align}
Hence the Lagrange form of the shifted Gegenbauer interpolation of $f$ at the SGG nodes can be written as:
\begin{equation}\label{sec:ort:eq:Lagint1}
	{P_n}f(x) = \sum\limits_{k = 0}^n {{f_{L,n,k}^{(\alpha)}}\,\mathcal{L} _{L,n,k}^{(\alpha )}(x)} ,
\end{equation}
where $\mathcal{L} _{L,n,k}^{(\alpha )}(x)$, are the Lagrange interpolating polynomials defined by
\begin{equation}\label{sec:ort:eq:Lag1}
	\mathcal{L} _{L,n,k}^{(\alpha )}(x) = \varpi _{L,n,k}^{(\alpha )}\sum\limits_{j = 0}^n {{{\left( {\lambda _{L,j}^{(\alpha )}} \right)}^{ - 1}}\,C_{L,j}^{(\alpha )}\left(x_{L,n,k}^{(\alpha )}\right)\,C_{L,j}^{(\alpha )}(x)} ,\quad k = 0, \ldots ,n.
\end{equation}
\begin{thm}
The functions $\mathcal{L} _{L,n,k}^{(\alpha )}(x), k = 0, \ldots ,n,$ defined by Equation \eqref{sec:ort:eq:Lag1} are the Lagrange interpolating polynomials of the real-valued function $f$ constructed through shifted Gegenbauer interpolation at the SGG nodes.
\end{thm}
\begin{proof}
To show that $\mathcal{L} _{L,n,k}^{(\alpha )}(x), k = 0, \ldots ,n,$ are indeed the Lagrange interpolating polynomials of the real-valued function $f$, we need only to show that $\mathcal{L} _{L,n,k}^{(\alpha )}(x_{L,n,i}^{(\alpha)}) = {\delta _{i,k}}, i, k = 0, \ldots ,n$. Since
\begin{align}
	\mathcal{L} _{L,n,k}^{(\alpha )}(x) &= {\left( {\frac{L}{2}} \right)^{2\alpha }}\,\varpi _{n,k}^{(\alpha )}\sum\limits_{j = 0}^n {{{\left( {{{\left( {\frac{L}{2}} \right)}^{2\alpha }}\lambda _j^{(\alpha )}} \right)}^{ - 1}}\,C_j^{(\alpha )}\left( {x_{n,k}^{(\alpha )}} \right)\,C_j^{(\alpha )}(x)}\nonumber\\
	&= \,\varpi _{n,k}^{(\alpha )}\sum\limits_{j = 0}^n {{{\left( {\lambda _j^{(\alpha )}} \right)}^{ - 1}}\,C_j^{(\alpha )}\left( {x_{n,k}^{(\alpha )}} \right)\,C_j^{(\alpha )}(x)}\nonumber\\
	&= \,\varpi _{n,k}^{(\alpha )}\sum\limits_{j = 0}^n {{{\left( {\lambda _j^{(\alpha )}} \right)}^{ - 1}}\,{{\left( {\frac{{j!\,\Gamma \left( {\alpha  + \frac{1}{2}} \right)}}{{\Gamma \left( {j + \alpha  + \frac{1}{2}} \right)}}} \right)}^2}\hat C_j^{(\alpha )}\left( {x_{n,k}^{(\alpha )}} \right)\,\hat C_j^{(\alpha )}(x)}\nonumber\\
	&= \,\varpi _{n,k}^{(\alpha )}\sum\limits_{j = 0}^n {{{\left( {\hat \lambda _j^{(\alpha )}} \right)}^{ - 1}}\,\hat C_j^{(\alpha )}\left( {x_{n,k}^{(\alpha )}} \right)\,\hat C_j^{(\alpha )}(x)} ,
\end{align}
where ${\hat C_j^{(\alpha )}(x)}$, are the Gegenbauer polynomials standardized by \citet{Szego1975}, and
\begin{equation}
	\hat \lambda _j^{(\alpha )} = \left\| {\hat C_j^{(\alpha )}} \right\|_{{w^{(\alpha )}}}^2 = \frac{{{2^{2\alpha  - 1}}\,{\Gamma ^2}\left( {j + \alpha  + \frac{1}{2}} \right)}}{{(j + \alpha )\,j!\,\Gamma (j + 2\alpha )}}.
\end{equation}
By Christoffel-Darboux Theorem (see \cite[Theorem 4.4]{Hesthaven2007}),
\begin{equation}
	\mathcal{L} _{L,n,k}^{(\alpha )}(x) = \varpi _{n,k}^{(\alpha )}\,\frac{{{2^{ - (2\alpha  + 1)}}\,\Gamma (n + 2)\,{\Gamma ^2}(2\alpha  + 1)}}{{\Gamma (n + 2\alpha  + 1)\,{\Gamma ^2}(\alpha  + 1)}}\,\frac{{\hat C_{n + 1}^{(\alpha )}(x)\,\hat C_n^{(\alpha )}\left(x_{n,k}^{(\alpha )}\right)}}{{x - x_{n,k}^{(\alpha )}}}.
\end{equation}
Hence $\mathcal{L} _{L,n,k}^{(\alpha )}(x_{n,i}^{(\alpha )}) = 0\,\forall i \ne k,$ since $\hat C_{n + 1}^{(\alpha )}(x_{n,i}^{(\alpha )}) = 0,\,\,i = 0, \ldots ,n$. For $i = k$, and using L'H\^{o}pital's rule, we find that
\begin{equation}
	\mathcal{L} _{L,n,k}^{(\alpha )}(x_{n,k}^{(\alpha )}) = \varpi _{n,k}^{(\alpha )}\,\frac{{{2^{ - (2\alpha  + 1)}}\,\Gamma (n + 2)\,{\Gamma ^2}(2\alpha  + 1)}}{{\Gamma (n + 2\alpha  + 1)\,{\Gamma ^2}(\alpha  + 1)}}\,\hat C_n^{(\alpha )}(x_{n,k}^{(\alpha )})\frac{d}{{dx}}\hat C_{n + 1}^{(\alpha )}(x_{n,k}^{(\alpha )}).
\end{equation}
Using Equation (5.19) in \cite{Hesthaven2007}, we can easily show that
\begin{equation}
	\mathcal{L} _{L,n,k}^{(\alpha )}(x_{n,k}^{(\alpha )}) = (n + 2\alpha  + 1)\,\hat C_n^{(\alpha )}(x_{n,k}^{(\alpha )})\,{\left( {\left( {1 - {{\left( {x_{n,k}^{(\alpha )}} \right)}^2}} \right)\,\frac{d}{{dx}}\hat C_{n + 1}^{(\alpha )}(x_{n,k}^{(\alpha )})} \right)^{ - 1}} = 1,\quad k = 0, \ldots, n.
\end{equation}
Hence
\begin{align}
\mathcal{L} _{L,n,k}^{(\alpha )}(x_{L,n,i}^{(\alpha )}) &= \varpi _{L,n,k}^{(\alpha )}\sum\limits_{j = 0}^n {{{\left( {\lambda _{L,j}^{(\alpha )}} \right)}^{ - 1}}{\mkern 1mu} C_{L,j}^{(\alpha )}\left( {x_{L,n,k}^{(\alpha )}} \right){\mkern 1mu} C_{L,j}^{(\alpha )}(x_{L,n,i}^{(\alpha )})}\nonumber\\
&= \varpi _{n,k}^{(\alpha )}\sum\limits_{j = 0}^n {{{\left( {\hat \lambda _j^{(\alpha )}} \right)}^{ - 1}}{\mkern 1mu} \hat C_j^{(\alpha )}\left( {x_{n,k}^{(\alpha )}} \right){\mkern 1mu} \hat C_j^{(\alpha )}(x_{n,i}^{(\alpha )})} = {\delta _{i,k}}. 	
\end{align}
\end{proof}

\vspace{-6pt}
\section{The Shifted Gegenbauer Integration Matrix}
\label{sec:theshi}
\vspace{-2pt}
Suppose that a real-valued function $f$ is approximated by the shifted Gegenbauer interpolant $P_nf$ given by Equation \eqref{sec:ort:eq:Lagint1}. The shifted Gegenbauer integration matrix calculated at the SGG nodes is simply a linear map, ${{{\mathbf{\hat P}}}_L^{(1)}}$, which takes a vector of $(n + 1)$ function values, ${\mathbf{F}} = {\left( {{f_{L,0}},{f_{L,1}}, \ldots ,{f_{L,n}}} \right)^T}$, to a vector of $(n + 1)$ integral values
\[\mathbf{I}_n^{(\alpha)} = {\left( {\int_0^{x_{L,n,0}^{(\alpha )}} {{P_n}f(x)\,dx} ,\int_0^{x_{L,n,1}^{(\alpha )}} {{P_n}f(x)\,dx} , \ldots ,\int_0^{x_{L,n,n}^{(\alpha )}} {{P_n}f(x)\,dx} } \right)^T},\]
such that
\begin{equation}
	\mathbf{I}_n^{(\alpha)} = {{{\mathbf{\hat P}}}_L^{(1)}}\,{\mathbf{F}}.
\end{equation}
The integration matrix ${\mathbf{\hat P}}_L^{(1)} = (\hat p_{L,i,k}^{(1)}),\,0 \le i,k \le n,$ is the first-order square shifted Gegenbauer integration matrix of size $(n + 1)$, and its elements, $\hat p_{L,i,k}^{(1)} = \hat p_{L,i,k}^{(1)}(\alpha),\,0 \le i,k \le n$, can be constructed by integrating Equation \eqref{sec:ort:eq:Lagint1} on $[0, L]$, such that
\begin{align}
	\int_0^{x_{L,n,i}^{(\alpha )}} {{P_n}f(x)\,dx} &= \sum\limits_{k = 0}^n {{f_{L,n,k}^{(\alpha)}}\,\int_0^{x_{L,n,i}^{(\alpha )}} {\mathcal{L} _{L,n,k}^{(\alpha )}(x)} \,dx}\nonumber\\
	&= \sum\limits_{k = 0}^n {\left( {\varpi _{L,n,k}^{(\alpha )}\,\sum\limits_{j = 0}^n {{{\left( {\lambda _{L,j}^{(\alpha )}} \right)}^{ - 1}}\,C_{L,j}^{(\alpha )}(x_{L,n,k}^{(\alpha )})\,\int_0^{x_{L,n,i}^{(\alpha )}} {C_{L,j}^{(\alpha )}(x)} \,dx} } \right)\,{f_{L,n,k}^{(\alpha)}}}\nonumber\\
	 &= \sum\limits_{k = 0}^n {\hat p_{L,i,k}^{(1)}\,{f_{L,n,k}^{(\alpha)}}} \label{sec:theshi:eq:squad1}.
\end{align}
Hence
\begin{equation}\label{sec:theshi:eq:pelem}
	\hat p_{L,i,k}^{(1)} = \varpi _{L,n,k}^{(\alpha )}\,\sum\limits_{j = 0}^n {{{\left( {\lambda _{L,j}^{(\alpha )}} \right)}^{ - 1}}\,C_{L,j}^{(\alpha )}(x_{L,n,k}^{(\alpha )})\,\int_0^{x_{L,n,i}^{(\alpha )}} {C_{L,j}^{(\alpha )}(x)} \,dx} ,\quad i, k = 0, \ldots ,n.
\end{equation}
We refer to the shifted Gegenbauer integration matrix, ${\mathbf{\hat P}}_L^{(1)}$, and its associated quadrature \eqref{sec:theshi:eq:squad1} by the S-matrix and the S-quadrature, respectively.
\subsection{Error Analysis of the S-Quadrature}
\label{subsec:err}
The following theorem highlights the truncation error of the shifted Gegenbauer quadrature associated with the shifted Gegenbauer integration matrix ${\mathbf{\hat P}}_L^{(1)}$.
\begin{thm}\label{subsec:err:thm1}
Let $f(x) \in C^{n + 1}[0, L]$, be interpolated by the shifted Gegenbauer polynomials at the SGG nodes, $x_{L,n,i}^{(\alpha)} \in S_{L,n}^{(\alpha)}$, then there exist a matrix ${\mathbf{\hat P}}_L^{(1)} = ({\hat p_{L,i,j}^{(1)}}),\,0 \le i,j \le n$, and some numbers $\xi_i = \xi(x_{L,n,i}^{(\alpha)}) \in (0, L), i = 0, \ldots, n$, satisfying
\begin{equation}\label{subsec:err:eq:squadki1}
\int_0^{x_{L,n,i}^{(\alpha )}} {f(x)dx}  = \sum\limits_{k = 0}^n {\hat p_{L,i,k}^{(1)}\,{f_{L,n,k}^{(\alpha)}}}  + E_{L,n}^{(\alpha )}\left( {x_{L,n,i}^{(\alpha )},{\xi _i}} \right),
\end{equation}
where $\hat p_{L,i,k}^{(1)}, i, k = 0, \ldots ,n$, are the elements of the matrix ${\mathbf{\hat P}}_L^{(1)}$, defined by Equation \eqref{sec:theshi:eq:pelem}, and
\begin{equation}\label{sec1:eq:errorkimohat}
E_{L,n}^{(\alpha )}\left( {x_{L,n,i}^{(\alpha )},{\xi _i}} \right) = {\left( {\frac{L}{2}} \right)^{n + 1}}\frac{{{f^{(n + 1)}}({\xi _i})}}{{(n + 1)!\,K_{n + 1}^{(\alpha )}}}\int_{ 0}^{x_{L,n,i}^{(\alpha )}} {C_{L,n + 1}^{(\alpha )}(x)\;dx}.
\end{equation}
\end{thm}
\begin{proof}
Set the error term of the shifted Gegenbauer interpolation as
\begin{equation}
	{R_n}(x) = f(x) - {P_n}f(x),
\end{equation}
and construct the auxiliary function
\begin{equation}
	Y(t) = {R_n}(t) - \frac{{{R_n}(x)}}{{C_{L,n + 1}^{(\alpha )}(x)}}C_{L,n + 1}^{(\alpha )}(t).
\end{equation}
Since $f \in {C^{n + 1}}[0,L]$, and ${P_n}f \in {C^\infty }[0,L],$ it follows that $Y \in {C^{n + 1}}[0,L]$. For $t = x_{L,n,i}^{(\alpha )}$, we have
\begin{equation}
	Y(x_{L,n,i}^{(\alpha )}) = {R_n}(x_{L,n,i}^{(\alpha )}) - \frac{{{R_n}(x)}}{{C_{L,n + 1}^{(\alpha )}(x)}}C_{L,n + 1}^{(\alpha )}(x_{L,n,i}^{(\alpha )}) = 0,
\end{equation}
since $x_{L,n,i}^{(\alpha )}$, are zeroes of ${{R_n}(x)}$. Moreover,
\begin{equation}
	Y(x) = {R_n}(x) - \frac{{{R_n}(x)}}{{C_{L,n + 1}^{(\alpha )}(x)}}C_{L,n + 1}^{(\alpha )}(x) = 0.
\end{equation}
Thus  $Y \in {C^{n + 1}}[0,L],$ and $Y$ is zero at the $(n + 2)$ distinct nodes $x, x_{L,n,i}^{(\alpha )}, i = 0, \ldots, n$. By the generalized Rolle's Theorem, there exists a number $\xi$ in $(0, L)$ such that ${Y^{(n + 1)}}(\xi ) = 0$. Therefore,
\begin{align}
	0 &= {Y^{(n + 1)}}(\xi ) = R_n^{(n + 1)}(\xi ) - \frac{{{R_n}(x)}}{{C_{L,n + 1}^{(\alpha )}(x)}}\frac{{{d^{n + 1}}}}{{d{t^{n + 1}}}}{\left. {C_{L,n + 1}^{(\alpha )}(t)} \right|_{t = \xi }}\nonumber\\
	&= R_n^{(n + 1)}(\xi ) - \frac{{{R_n}(x)}}{{C_{L,n + 1}^{(\alpha )}(x)}}{\left( {\frac{2}{L}} \right)^{n + 1}}\frac{{{d^{n + 1}}}}{{d{t^{n + 1}}}}{\left. {C_{n + 1}^{(\alpha )}(t)} \right|_{t = \xi }}\nonumber\\
	&= R_n^{(n + 1)}(\xi ) - {\left( {\frac{2}{L}} \right)^{n + 1}}\,(n + 1)!\,K_{n + 1}^{(\alpha )}\frac{{{R_n}(x)}}{{C_{L,n + 1}^{(\alpha )}(x)}}.
\end{align}
Since ${P_n}f \in {\mathbb{P}_n},\,{({P_n}f)^{(n + 1)}}(x)$ is identically zero, and we have
\begin{equation}
	0 = {f^{(n + 1)}}(\xi ) - {\left( {\frac{2}{L}} \right)^{n + 1}}\,(n + 1)!\,K_{n + 1}^{(\alpha )}\frac{{{R_n}(x)}}{{C_{L,n + 1}^{(\alpha )}(x)}}.
\end{equation}
\begin{equation}
	\therefore f(x) = {P_n}f(x) + {\left( {\frac{L}{2}} \right)^{n + 1}}\frac{{{f^{(n + 1)}}(\xi )}}{{(n + 1)!\,K_{n + 1}^{(\alpha )}}}\,C_{L,n + 1}^{(\alpha )}(x).
\end{equation}
\begin{align}
	\Rightarrow \int_0^{x_{L,n,i}^{(\alpha )}} {f(x)\,dx}  &= \int_0^{x_{L,n,i}^{(\alpha )}} {{P_n}f(x)\,dx}  + {\left( {\frac{L}{2}} \right)^{n + 1}}\frac{{{f^{(n + 1)}}(\xi (x_{L,n,i}^{(\alpha )}))}}{{(n + 1)!\,K_{n + 1}^{(\alpha )}}}\,\int_0^{x_{L,n,i}^{(\alpha )}} {C_{L,n + 1}^{(\alpha )}(x)\,dx}\\
	&= \sum\limits_{k = 0}^n {\hat p_{L,i,k}^{(1)}{f_{L,n,k}^{(\alpha)}}}  + E_{L,n}^{(\alpha )}\left( {x_{L,n,i}^{(\alpha )},{\xi _i}} \right).
\end{align}
\end{proof}
\subsection{Optimal S-Quadrature}
\label{subsec:opt}
To construct an optimal S-quadrature to approximate the definite integration $\int_0^{{x_i}} {f(x)\,dx}$, of an integrable function $f$, for any arbitrary integration node $x_i \in [0, L]$, we follow the idea presented by \citet{Elgindy2013}, and seek to determine the optimal Gegenbauer parameter $\alpha_i^*$, which minimizes the magnitude of the quadrature error $E_{L,n}^{(\alpha )}(x_i,\xi )$, at each integration node $x_i, i = 0, \ldots, n$. Define the smooth function $\eta _{L,i,n} {(\alpha )}$, such that
\begin{equation}
	\eta _{L,i,n} {(\alpha )} = \frac{{{2^n}}}{{K_{n + 1}^{(\alpha )}}}\int_0^{{x_i}} {C_{L, n + 1}^{(\alpha )}(x)\,dx}.
\end{equation}
The values of the optimal Gegenbauer parameters $\alpha_i^*$, can be determined through the following one-dimensional optimization problems:
\begin{equation}\label{subsec:opt:reducedprob1}
{\text{Find }}\alpha _i^* = \mathop {{\text{argmin}}}\limits_{\alpha  >  - 1/2} \eta _{L,i,n}^2(\alpha ),\quad i = 0, \ldots, n.
\end{equation}
Problems (\ref{subsec:opt:reducedprob1}) can be further converted into unconstrained one-dimensional minimization problems using the following change of variable:
\begin{equation}\label{sec:theshi:eq:change1}
	\alpha  = {t^2} - \frac{1}{2} + \varepsilon,\quad 0 < \varepsilon  < \frac{1}{2}(1 + 2\,\alpha ) \ll 1.
\end{equation}
We refer to the optimal shifted Gegenbauer integration matrix and its associated quadrature established through the solution of Problems \eqref{subsec:opt:reducedprob1} by the optimal S-matrix and the optimal S-quadrature, respectively. Notice here that for each integration node $x_i$, an optimal Gegenbauer parameter $\alpha_i^*$ is determined, and the optimal S-quadrature seeks a new set of SGG nodes as the optimal shifted Gegenbauer interpolation nodes set corresponding to the integration node $x_i$. We denote these optimal SGG interpolation nodes by $z_{L,m,i,k}^{(\alpha_i^*)}, k = 0, \ldots, m,$ for some $m \in \mathbb{Z}^+$, and we call them the adjoint SGG nodes, since their role is similar to the role of the adjoint Gegenbauer-Gauss nodes $z_{m,i,k}^{(\alpha_i^*)}, k = 0, \ldots, m,$ constructed in \cite{Elgindy2013}. Notice also that the choice of the positive integer number $m$ is free, which renders the optimal S-matrix a rectangular matrix of size $(n + 1) \times (m + 1)$ rather than a square matrix of size $(n + 1)$, as is typically the case with the standard S-matrix. Denote the optimal S-matrix by $P_L^{(1)} = (p_{L,i,k}^{(1)}), i = 0, \ldots, n; k = 0, \ldots, m$, where $p_{L,i,k}^{(1)} = p_{L,i,k}^{(1)}(\alpha_i^*)$, are the matrix elements of the $i$th row obtained using the optimal value $\alpha_i^*$. The definite integral $\int_{ 0}^{{x_i}} {f(x)\,dx}$, is then approximated by the optimal S-quadrature as follows:
\begin{equation}\label{subsec:opt:squad1}
	\int_{ 0}^{{x_i}} {f(x)\, dx}  \approx \sum\limits_{k = 0}^m {{p_{L,i,k}^{(1)}}\, f_{L,m,i,k}^{(\alpha_i^*)}},\quad i = 0, \ldots ,n,
\end{equation}
where $f_{L,m,i,k}^{(\alpha_i^*)} = f({z_{L,m,i,k}^{(\alpha_i^*)}}), \quad i = 0, \ldots ,n; k = 0, \ldots, m$.
\subsection{Error Analysis of the Optimal S-Quadrature}
\label{subsec:erranaotosq}
The following theorem describes the construction of the optimal S-matrix elements, and highlights the truncation error of the associated optimal S-quadrature.
\begin{thm}\label{subsec1:theoremnew}
\label{subsec1:kimobasha1}
Let
\begin{equation}\label{sec1:eq:ultragauss1}
    S_{L,n,m} = \{ {z_{L,m,i,k}^{(\alpha_i^*)}}|C_{L,m + 1}^{(\alpha_i^* )}({z_{L,m,i,k}^{(\alpha_i^*)}}) = 0, i = 0, \ldots ,n; k = 0, \ldots ,m\},\quad L \in \mathbb{R}^+, n, m \in \mathbb{Z}^+,
\end{equation}
be the adjoint set of SGG nodes, where $\alpha _i^*$ are the optimal Gegenbauer parameters in the sense that
\begin{equation}\label{sec1:eq:optgegepar1}
\alpha _i^* = \mathop {{\text{argmin}}}\limits_{\alpha  >  - 1/2}  \eta _{L,i,m}^2(\alpha ).
\end{equation}
Moreover, let $f(x) \in {C^{m+1} }[0, L]$, be a real-valued function approximated by the shifted Gegenbauer polynomials expansion series such that the shifted Gegenbauer coefficients are computed by interpolating the function $f(x)$ at the adjoint SGG nodes $z_{L,m,i,k}^{(\alpha_i^*)} \in S_{L,n,m}, i = 0, \ldots ,n; k = 0, \ldots ,m$. Then for any arbitrary integration nodes $x_i \in [0, L], i = 0, \ldots, n$, there exist a matrix $\mathbf{P}_L^{(1)} = ({p_{L,i,j}^{(1)}}), i = 0,\ldots, n; j = 0, \ldots ,m$, and some numbers $\xi_i \in [0, L]$, satisfying
\begin{equation}\label{sec1:eq:ultraint2kimo}
\int_{ 0}^{{x_i}} {f(x)\, dx}  = \sum\limits_{k = 0}^m {{p_{L,i,k}^{(1)}}\, f_{L,m,i,k}^{(\alpha_i^*)}}  + E_{L,m}^{(\alpha_i^* )}({x_i},\xi_i),
\end{equation}
where
\begin{equation}\label{sec1:eq:qentrieskimo}
{p_{L,i,k}^{(1)}} = \varpi _{L,m,k}^{(\alpha_i^* )}\, \sum\limits_{j = 0}^m {{{\left(\lambda _{L,j}^{(\alpha_i^* )}\right)}^{ - 1}} C_{L,j}^{(\alpha_i^* )}({z_{L,m,i,k}^{(\alpha_i^*)}})\int_{ 0}^{{x_i}} {C_{L,j}^{(\alpha_i^* )}(x)\, dx} };
\end{equation}
\begin{equation}\label{sec1:eq:errorkimo}
E_{L,m}^{(\alpha _i^*)}({x_i},{\xi _i}) = {\left( {\frac{L}{2}} \right)^{m + 1}}\frac{{{f^{(m + 1)}}({\xi _i})}}{{{2^m}\,(m + 1)!}}{\mkern 1mu} {\eta _{L,i,m}}(\alpha _i^*).
\end{equation}
\end{thm}
\begin{proof}
The function
\begin{equation}\label{sec:ort:eq:modal2}
{P_m}f(x) = \sum\limits_{j = 0}^m {{{\tilde f}_{i,j}}\,C_{L,j}^{(\alpha_i^* )}(x)} ,
\end{equation}
is the shifted Gegenbauer interpolant of the real function $f$ defined on $[0, L]$, if we compute the coefficients ${{\tilde f}_{i,j}}$ so that
\begin{equation}
	{P_m}f({z_{L,m,i,k}^{(\alpha_i^*)}}) = f_{L,m,i,k}^{(\alpha_i^*)},\quad i = 0, \ldots ,n; k = 0, \ldots, m.
\end{equation}
Hence the discrete shifted Gegenbauer transform is
\begin{equation}
	{\tilde f_{i,k}} = \frac{{{{({P_m}f,C_{L,j}^{(\alpha _i^*)})}_{L,m}}}}{{\left\| {C_{L,j}^{(\alpha _i^*)}} \right\|_{w_L^{(\alpha )}}^2}} = \frac{{{{(f,C_{L,j}^{(\alpha _i^*)})}_{L,m}}}}{{\left\| {C_{L,j}^{(\alpha _i^*)}} \right\|_{w_L^{(\alpha )}}^2}} = \frac{1}{{\lambda _{L,j}^{(\alpha _i^*)}}}\sum\limits_{k = 0}^m {\varpi _{L,m,k}^{(\alpha _i^*)}{\mkern 1mu} f_{L,m,i,k}^{(\alpha _i^*)}{\mkern 1mu} C_{L,j}^{(\alpha _i^*)}\left( {z_{L,m,i,k}^{(\alpha _i^*)}} \right)}, \quad	i = 0, \ldots ,n;\,k = 0, \ldots ,m.
\end{equation}
Following the approach presented in Section \ref{sec:ort}, we can easily show that the Lagrange form of the shifted Gegenbauer interpolation of $f$ at the adjoint SGG nodes can be written as:
\begin{equation}\label{sec:ort:eq:Lagint12}
	{P_m}f(x) = \sum\limits_{k = 0}^m {{f_{L,m,i,k}^{(\alpha_i^*)}}\,\mathcal{L} _{L,m,i,k}^{(\alpha_i^* )}(x)} ,
\end{equation}
where $\mathcal{L} _{L,m,i,k}^{(\alpha_i^*)}(x)$, are the Lagrange interpolating polynomials defined by
\begin{equation}\label{sec:ort:eq:Lag12}
	\mathcal{L} _{L,m,i,k}^{(\alpha_i^* )}(x) = \varpi _{L,m,k}^{(\alpha_i^* )}\sum\limits_{j = 0}^m {{{\left( {\lambda _{L,j}^{(\alpha_i^* )}} \right)}^{ - 1}}\,C_{L,j}^{(\alpha_i^* )}\left(z_{L,m,i,k}^{(\alpha_i^* )}\right)\,C_{L,j}^{(\alpha_i^* )}(x)} ,\quad i = 0, \ldots ,n; k = 0, \ldots, m.
\end{equation}
Therefore,
\begin{align*}
	\int_0^{{x_i}} {{P_m}f(x)\,dx} &= \sum\limits_{k = 0}^m {f_{L,m,i,k}^{(\alpha _i^*)}{\mkern 1mu} \int_0^{{x_i}} {\mathcal{L} _{L,m,i,k}^{(\alpha _i^*)}(x)\,dx} }\nonumber\\
	&= \sum\limits_{k = 0}^m {\left[ {\varpi _{L,m,k}^{(\alpha _i^*)}\sum\limits_{j = 0}^m {{{\left( {\lambda _{L,j}^{(\alpha _i^*)}} \right)}^{ - 1}}{\mkern 1mu} C_{L,j}^{(\alpha _i^*)}\left( {z_{L,m,i,k}^{(\alpha _i^*)}} \right){\mkern 1mu} \int_0^{{x_i}} {C_{L,j}^{(\alpha _i^*)}(x)\,dx} } } \right]f_{L,m,i,k}^{(\alpha _i^*)}}\\
	&= \sum\limits_{k = 0}^m {{p_{L,i,k}^{(1)}} f_{L,m,i,k}^{(\alpha_i^*)}}.\\
\Rightarrow \int_0^{{x_i}} {f(x)\,dx} &= \sum\limits_{k = 0}^m {p_{L,i,k}^{(1)}f_{L,m,i,k}^{(\alpha _i^*)}}  + E_{L,m}^{(\alpha _i^*)}({x_i},{\xi _i}),
\end{align*}
where the quadrature error term, $E_{L,m}^{(\alpha _i^*)}({x_i},{\xi _i})$, follows directly from Theorem \ref{subsec:err:thm1} on substituting the value of $\alpha$ with $\alpha_i^*$, and expanding the shifted Gegenbauer expansion series up to the $(m+1)$th-term.
\end{proof}
\subsection{Error Bounds of the Optimal S-Quadrature}
To study the error bounds of the optimal S-quadrature, we require the following two lemmas.
\label{subsec:errbounds}
\begin{lem}\label{subsec:errbounds:lem:max1}
The maximum value of the shifted Gegenbauer polynomials $C_{L,n}^{(\alpha)}(x)$, is given by
\begin{subequations}\label{subsec:errbounds:eq:linfnorm1}
\begin{empheq}[left={\mathop {{\text{max}}}\limits_{x \in [0,\,L]} \left| {C_{L,n}^{(\alpha )}(x)} \right| = {\left\| {C_{L,n}^{(\alpha )}} \right\|_{{L^\infty }[0,\,L]}} =}\empheqlbrace]{align}
&1,\quad n \ge 0 \wedge \alpha  \ge 0,\\
&\frac{{n!\,\Gamma (2\alpha )}}{{\Gamma (n + 2\alpha )}}\,\left| {\left( \begin{array}{l}
\frac{n}{2} + \alpha  - 1\\
\hfill \frac{n}{2} \hfill
\end{array} \right)} \right|,\quad \frac{n}{2} \in \mathbb{Z}_0^+ \wedge\; - \frac{1}{2} < \alpha  < 0,\\
&A^{(\alpha)}\,{n^{ - \alpha }},\quad n \gg 1 \wedge  - \frac{1}{2} < \alpha  < 0\label{subsec:errbounds:eq:linfnorm22},
\end{empheq}
\end{subequations}
where $\mathbb{Z}_0^+ = \mathbb{Z}^+ \cup \{ 0\}$, is the set of all non-negative integers, and $A^{(\alpha)} > 1$, is a constant dependent on $\alpha$, but independent of $n$. Moreover, for odd $n > 0$, and $- \frac{1}{2} < \alpha  < 0$, the maximum value of $C_{L,n}^{(\alpha)}(x)$, is bounded by the following inequality
\begin{equation}\label{subsec:errbounds:eq:linfnorm2}
	\mathop {{\text{max}}}\limits_{x \in [0,\,L]} \left| {C_{L,n}^{(\alpha )}(x)} \right| = {\left\| {C_{L,n}^{(\alpha )}} \right\|_{{L^\infty }[0,\,L]}} < \frac{{2\,n!\,\Gamma (2\alpha )\,\left| \alpha  \right|}}{{\sqrt {n\,(2\alpha  + n)} \,\Gamma (n + 2\alpha )}}\,\left| {\left( \begin{array}{l}
	\frac{{n - 1}}{2} + \alpha \\
	\hfill \frac{{n - 1}}{2} \hfill
	\end{array} \right)} \right|.
\end{equation}
\end{lem}
\begin{proof}
The proof is straightforward. Indeed, Equalities \eqref{subsec:errbounds:eq:linfnorm1} follow using Equation \eqref{sec:pre:eq:normaliz1}, Lemma 2.1 in \cite{Elgindy2013a}, and Equation (7.33.2) in \cite{Szego1975}. Since $C_{L,n+1}^{(\alpha)}(x)$, for a certain value of $0 \le x \le L$, monotonically decreases for increasing values of $\alpha$ in the range $-1/2 < \alpha < 0$, as $n \to \infty$; cf. \cite[Appendix D]{Elgindy2013}, then
\begin{equation}
	{A^{(\alpha)} } = {n^\alpha }\,{\left\| {C_{L,n}^{(\alpha )}} \right\|_{{L^\infty }[0,L]}},
\end{equation}
is also monotonically decreasing for increasing values of $\alpha \in (-1/2, 0)$. Since
\begin{equation}
	\mathop {\lim }\limits_{\alpha  \to 0} {A^{(\alpha)} } = 1,
\end{equation}
then $A^{(\alpha)} > 1\, \forall \alpha \in (-1/2, 0)$. Finally, Inequality \eqref{subsec:errbounds:eq:linfnorm2} follows using Equation \eqref{sec:pre:eq:normaliz1} and Equation (7.33.3) in \cite{Szego1975}.
\end{proof}
\begin{lem}\label{sec1:lem:1}
For a fixed $\alpha > -1/2$, the factor $(n + 1)!  K_{n + 1}^{(\alpha )}$, is of order ${n^{3/2 - \alpha }}  {(2  n/e)^n}$, for large values of $n$.
\end{lem}
\begin{proof}
The lemma is a more accurate version of Lemma 2.2 in \cite{Elgindy2013a} by realizing that $(n + 1)! = (n + 1) \cdot n! \approx n \cdot \sqrt {2\,\pi \,n} \,{\left( {n/e} \right)^n} = \sqrt {2\,\pi } \,{n^{\frac{3}{2}}}\,{\left( {n/e} \right)^n}$, as $n \to \infty$.
\end{proof}
The following theorem gives the error bounds of the optimal S-quadrature.
\begin{thm}[Error bounds]\label{sec1:thm:krooma1}
Assume that $f(x) \in C^{m+1}[0, L]$, and ${\left\| {{f^{(m + 1)}}} \right\|_{{L^\infty }[0,L]}} \le A \in {\mathbb{R}^ + }$, for some number $m \in \mathbb{Z}_0^+$. Moreover, let $\int_{0}^{{x_i}} {f(x)\,dx}$, be approximated by the optimal S-quadrature \eqref{subsec:opt:squad1} up to the $(m+1)$th shifted Gegenbauer expansion term, for each integration node $x_i, i = 0, \ldots, n$. Then there exist some positive constants $D_1^{(\alpha _i^*)}$ and $D_2^{(\alpha _i^*)}$, independent of $m$ such that the truncation error of the optimal S-quadrature, $E_{L,m}^{(\alpha _i^*)}$, is bounded by the following inequalities:
\begin{equation}
	\left| {E_{L,m}^{(\alpha _i^*)}} \right| \le \frac{{A{2^{ - 2m - 1}}\Gamma \left( {\alpha  + 1} \right){x_i}{L^{m + 1}}\Gamma \left( {m + 2\alpha  + 1} \right)}}{{\Gamma \left( {2\alpha  + 1} \right)\Gamma \left( {m + 2} \right)\Gamma \left( {m + \alpha  + 1} \right)}}\left( {\left\{ \begin{array}{l}
	1,\quad m \ge 0 \wedge \alpha  \ge 0,\\
	\frac{{\left( {m + 1} \right)!\Gamma \left( {2\alpha } \right)}}{{\Gamma \left( {m + 2\alpha  + 1} \right)}}\,\left| {\left( {\begin{array}{*{20}{c}}
{\frac{{m + 1}}{2} + \alpha  - 1}\\
{\frac{{m + 1}}{2}}
\end{array}} \right)} \right|,\quad \frac{{m + 1}}{2} \in \mathbb{Z}^+  \wedge  - \frac{1}{2} < \alpha  < 0
	\end{array} \right.} \right),
\end{equation}
\begin{equation}
	\left| {E_{L,m}^{(\alpha _i^*)}} \right| < \frac{{A{2^{ - 2m - 1}}\Gamma \left( \alpha  \right)\left| \alpha  \right|{\mkern 1mu} {x_i}{L^{m + 1}}}}{{\sqrt {\left( {m + 1} \right)\left( {2\alpha  + m + 1} \right)} {\mkern 1mu} \Gamma \left( {m + \alpha  + 1} \right)}}\,\left| {\left( {\begin{array}{*{20}{c}}
{\frac{m}{2} + \alpha }\\
{\frac{m}{2}}
\end{array}} \right)} \right|,\quad \frac{m}{2} \in \mathbb{Z}_0^ + \wedge  - \frac{1}{2} < \alpha  < 0,
\end{equation}
\begin{subequations}\label{sec1:ineq:errorkimo1}
	\begin{empheq}[left={\left| {E_{L,m}^{(\alpha _i^*)}} \right| \le}\empheqlbrace]{align}
	&B_1^{(\alpha _i^*)}\frac{{{e^m}\,{L^{m + 1}}\,{x_i}}}{{{2^{2m + 1}}\,{m^{m + 3/2 - \alpha _i^*}}}},\quad \alpha _i^* \ge 0 \wedge m \gg 1,\\
	&B_2^{(\alpha _i^*)}\frac{{{e^m}\,{L^{m + 1}}\,{x_i}}}{{{2^{2m + 1}}\,{m^{m + 3/2}}}},\quad - \frac{1}{2} < \alpha _i^* < 0  \wedge m \gg 1,
	\end{empheq}
\end{subequations}
for all $i = 0, \ldots, n$, where $B_1^{(\alpha _i^*)} = A D_1^{(\alpha _i^*)}$, and $B_2^{(\alpha _i^*)} = B_1^{(\alpha _i^*)} D_2^{(\alpha _i^*)}$.
\end{thm}
\begin{proof}
The proof is straightforward using Equation \eqref{sec1:eq:errorkimo}, and Lemmas \ref{subsec:errbounds:lem:max1} and \ref{sec1:lem:1}.
\end{proof}
\subsection{The Relation Between the S-matrix and the Gegenbauer Integration Matrix}
\label{subsec:rel}
Let ${{{\mathbf{\hat P}}}^{(1)}} = (\hat p_{i,k}^{(1)}),i,k = 0, \ldots, n,$ be the first-order square Gegenbauer integration matrix of size $(n + 1)$, constructed by Theorem 2.1 in \cite{Elgindy2013}, and consider the Lagrange form of the shifted Gegenbauer interpolation of a real-valued function $f$ at the SGG nodes, $P_nf$, given by Equation \ref{sec:ort:eq:Lagint1}. Since
\begin{align}
\sum\limits_{k = 0}^n {\hat p_{L,i,k}^{(1)}{\mkern 1mu} f_{L,n,k}^{(\alpha )}}  &= \sum\limits_{k = 0}^n {\hat p_{L,i,k}^{(1)}{\mkern 1mu} ({P_n}f)_{L,n,k}^{(\alpha )}}  = \int_0^{x_{L,n,i}^{(\alpha )}} {{P_n}f(x){\mkern 1mu} dx}  = \frac{L}{2}\int_{ - 1}^{{2\,x_{L,n,i}^{(\alpha )}/L} - 1} {{P_n}f\left( {\frac{L}{2}(x + 1)} \right)\,dx}\nonumber\\
& = \frac{L}{2}\int_{ - 1}^{x_{n,i}^{(\alpha )}} {{P_n}f\left( {\frac{L}{2}(x + 1)} \right)\,dx}  = \frac{L}{2}\sum\limits_{k = 0}^n {\hat p_{ik}^{(1)}\,{P_n}f\left( {\frac{L}{2}\left( {x_{n,k}^{(\alpha )} + 1} \right)} \right)}  = \frac{L}{2}\sum\limits_{k = 0}^n {\hat p_{ik}^{(1)}\,({P_n}f)_{L,n,k}^{(\alpha )}}\nonumber\\
& = \frac{L}{2}\sum\limits_{k = 0}^n {\hat p_{ik}^{(1)}\,f_{L,n,k}^{(\alpha )}},
\end{align}
where $({P_n}f)_{L,n,k}^{(\alpha )} = {P_n}f(x_{L,n,k}^{(\alpha )})$, then it follows that
\begin{equation}
	\hat p_{L,i,k}^{(1)} = \frac{L}{2}\,\hat p_{ik}^{(1)},\quad i,k = 0, \ldots ,n,
\end{equation}
or in matrix form,
\begin{equation}
	{{{\mathbf{\hat P}}}_L^{(1)}} = \frac{L}{2}\,{{{\mathbf{\hat P}}}^{(1)}}.
\end{equation}
Moreover, using the change of variable
\begin{equation}
	{t_i} = \frac{L}{2}({s_i} + 1),\quad i = 0, \ldots ,n,
\end{equation}
and Cauchy's formula for repeated integration, we can calculate the $q$-fold definite integrals of ${{P_n}f}$ on $[0, x_{L,n,i}^{(\alpha )}], i = 0, \ldots ,n$, as follows:
\begin{align}
&\sum\limits_{k = 0}^n {\hat p_{L,i,k}^{(q)}\,f_{L,n,k}^{(\alpha )}}  = \sum\limits_{k = 0}^n {\hat p_{L,i,k}^{(q)}{\mkern 1mu} ({P_n}f)_{L,n,k}^{(\alpha )}}  = \int_0^{x_{L,n,i}^{(\alpha )}} {\int_0^{{t_{q - 1}}} { \ldots \int_0^{{t_2}} {\int_0^{{t_1}} {{P_n}f({t_0})\,d{t_0}d{t_1} \ldots d{t_{q - 2}}d{t_{q - 1}}} } } } \nonumber\\
&= {\left( {\frac{L}{2}} \right)^q}\,\int_{ - 1}^{2\,x_{L,n,i}^{(\alpha )}/L - 1} {\int_{ - 1}^{2\,{t_{q - 1}}/L - 1} { \ldots \int_{ - 1}^{2\,{t_2}/L - 1} {\int_{ - 1}^{2\,{t_1}/L - 1} {{P_n}f\left( {\frac{L}{2}({s_0} + 1)} \right)\,d{s_0}d{s_1} \ldots d{s_{q - 2}}d{s_{q - 1}}} } } } \nonumber\\
	&= {\left( {\frac{L}{2}} \right)^q}\,\int_{ - 1}^{x_{n,i}^{(\alpha )}} {\int_{ - 1}^{{s_{q - 1}}} { \ldots \int_{ - 1}^{{s_2}} {\int_{ - 1}^{{s_1}} {{P_n}f\left( {\frac{L}{2}({s_0} + 1)} \right)\,d{s_0}d{s_1} \ldots d{s_{q - 2}}d{s_{q - 1}}} } } }  = {\left( {\frac{L}{2}} \right)^q}\,\sum\limits_{k = 0}^n {\hat p_{ik}^{(q)}\,f_{L,n,k}^{(\alpha )}} \nonumber\\
	& = {\left( {\frac{L}{2}} \right)^q}\,\frac{1}{{(q - 1)!}}\int_{ - 1}^{x_{n,i}^{(\alpha )}} {{{\left( {x_{n,i}^{(\alpha )} - t} \right)}^{q - 1}}\,{P_n}f\left( {\frac{L}{2}(t + 1)} \right)\,dt}  = {\left( {\frac{L}{2}} \right)^q}\,\frac{1}{{(q - 1)!}}\,\sum\limits_{k = 0}^n {\hat p_{ik}^{(1)}\,{{\left( {x_{n,i}^{(\alpha )} - x_{n,k}^{(\alpha )}} \right)}^{q - 1}}\,f_{L,n,k}^{(\alpha )}},
\end{align}
where ${\mathbf{\hat P}}_L^{(q)} = \left( {\hat p_{L,i,k}^{(q)}} \right),\, i,k = 0, \ldots ,n,$ is the S-matrix of order $q, q \in \mathbb{Z}^+$. Hence,
\begin{equation}
	\hat p_{L,i,k}^{(q)} = {\left( {\frac{L}{2}} \right)^q}\,{\mkern 1mu} \hat p_{i,k}^{(q)} = {\left( {\frac{L}{2}} \right)^q}\,\frac{{{{\left( {x_{n,i}^{(\alpha )} - x_{n,k}^{(\alpha )}} \right)}^{q - 1}}}}{{(q - 1)!}}\,\hat p_{i,k}^{(1)},\quad i,k = 0, \ldots ,n,
\end{equation}
or in matrix form,
\begin{equation}
	{{{\mathbf{\hat P}}}_L^{(q)}} = {\left( {\frac{L}{2}} \right)^q}\,{{{\mathbf{\hat P}}}^{(q)}} = {\left( {\frac{L}{2}} \right)^q}\,\frac{1}{{(q - 1)!}}\left( {\left( {{\mathbf{x}}_n^{(\alpha )} \otimes {{\mathbf{J}}_{1,n + 1}}} \right) - \left( {{{\left( {{\mathbf{x}}_n^{(\alpha )}} \right)}^T} \otimes {{\mathbf{J}}_{n + 1,1}}} \right)} \right) \circ {{{\mathbf{\hat P}}}^{(1)}},
\end{equation}
where ${\mathbf{x}}_n^{(\alpha )} = {[x_{n,0}^{(\alpha )},x_{n,1}^{(\alpha )}, \ldots ,x_{n,n}^{(\alpha )},]^T}, {{\mathbf{J}}_{i,j}}$ is the all ones matrix of size $i \times j$, ``$\otimes$'' and ``$\circ$'' denote the Kronecker product and Hadamard product (entrywise product), respectively. Hence the S-matrices of higher-orders can be generated directly from the first-order Gegenbauer integration matrix. Similarly, we can show that the optimal S-matrices of distinct orders are related with the Gegenbauer integration matrices constructed by Theorem 2.2 in \cite{Elgindy2013} by the following equations:
\begin{equation}\label{subsec:rel:eq:smat1}
p_{L,i,k}^{(q)} = {\left( {\frac{L}{2}} \right)^q}\,{\mkern 1mu} p_{i,k}^{(q)} = {\left( {\frac{L}{2}} \right)^q}\,\frac{{{{\left( {x_{n,i}^{(\alpha )} - z_{m,i,k}^{(\alpha_i^* )}} \right)}^{q - 1}}}}{{(q - 1)!}}\,p_{i,k}^{(1)},\quad i = 0, \ldots ,n;k = 0, \ldots ,m.
\end{equation}
The optimal S-matrices of distinct orders can therefore be calculated efficiently using Equations \eqref{subsec:rel:eq:smat1}, and Algorithms 2.1 and 2.2 in \cite{Elgindy2013}. Since the convergence properties of the S-quadratures and the optimal S-quadratures are inherited from the convergence properties of the corresponding Gegenbauer integration matrices and optimal Gegenbauer integration matrices, respectively, the following useful result is straightforward.
\begin{cor}[Convergence of the optimal S-quadrature]\label{subsec:rel:squad}\text{}\\
Assume that $f(x) \in C^{m+1}[0, L]$, and ${\left\| {{f^{(m + 1)}}} \right\|_{{L^\infty }[0,L]}} \le A \in {\mathbb{R}^ + },$ for some number $m \in \mathbb{Z}^+$. Moreover, let $\int_{0}^{{x_i}} {f(x)\,dx}$, be approximated by the optimal S-quadrature \eqref{subsec:opt:squad1} up to the $(m+1)$th shifted Gegenbauer expansion term, for each integration node $x_i \in [0, L], i = 0, \ldots, n$. Then the optimal S-quadrature converges to the optimal shifted Chebyshev quadrature in the $L^{\infty}$-norm, as $m \to \infty$.
\end{cor}
\begin{proof}
The corollary can be established easily using Equations \eqref{subsec:rel:eq:smat1} and Theorem 2.4 in \cite{Elgindy2013}.
\end{proof}

\vspace{-6pt}
\section{The SGPM}
\label{sec:theshi2}
\vspace{-2pt}
We commence our numerical scheme by recasting Problem $\mathcal{P}$ into its integral formulation. Thus twice integrating Equation \eqref{int:eq:teleg1} with respect to $t$ yields,
\begin{equation}\label{sec:theshi2:eq:tel2}
	u(x,t) - \kappa (x,t) + \beta_1 \int_0^t {u(x,\sigma )\,d\sigma }  + \beta_2 \int_0^t {\int_0^{{\sigma _2}} {u(x,{\sigma _1})\,d{\sigma _1}d{\sigma _2}} }  = \int_0^t {\int_0^{{\sigma _2}} {({u_{xx}}(x,{\sigma _1}) + f(x,{\sigma _1}))\,d{\sigma _1}d{\sigma _2}} },
\end{equation}
where
\begin{equation}
	\kappa (x,t) = (\beta_1 {\mkern 1mu} t + 1){\mkern 1mu} {g_1}(x) + t{\mkern 1mu} {g_2}(x).
\end{equation}
Using the substitution
\begin{equation}
	{u_{xx}}(x,t) = \phi (x,t),
\end{equation}
for some unknown function $\phi$, we can recover the unknown solution function $u$ and its first-order partial derivative $u_x$ in terms of $\phi$ by successive integration as follows:
\begin{align}
{u_x}(x,t) &= \int_0^x {\phi (\sigma ,t)\,d\sigma }  + {c_1}(t),\\
u(x,t) &= \int_0^x {\int_0^{{\sigma _2}} {\phi ({\sigma _1},t)\,d{\sigma _1}\,d{\sigma _2}} }  + x\,{c_1}(t) + {c_2}(t),
\end{align}
where ${c_1}(t)$ and ${c_2}(t)$ are some arbitrary functions in $t$. Using Dirichlet boundary conditions \eqref{int:dbc1} and \eqref{int:dbc2}, we find that
\begin{align}
	{c_1}(t) &= \frac{1}{l}\,\left( {{h_2}(t) - {h_1}(t) - \int_0^l {\int_0^{{\sigma _2}} {\phi ({\sigma _1},t)\,d{\sigma _1}\,d{\sigma _2}} } } \right);\\
	{c_2}(t) &= {h_1}(t).
\end{align}
Let ${\theta _{l,x}} = x/l$, and define the function $\psi (x,t)$, such that
\begin{equation}
	\psi (x,t) = \theta _{l,x}\,\left( {{h_2}(t) - {h_1}(t)} \right) + {h_1}(t),
\end{equation}
Moreover, let
\begin{align}
I_{q,x}^{(x)}(\phi (x,t)) &= \int_0^x {\int_0^{{\sigma _{q - 1}}} { \ldots \int_0^{{\sigma _2}} {\int_0^{{\sigma _1}} {\phi ({\sigma _0},t){\mkern 1mu} d{\sigma _0}d{\sigma _1} \ldots d{\sigma _{q - 2}}d{\sigma _{q - 1}}} } } },\\
I_{q,t}^{(t)}(\phi (x,t)) &= \int_0^t {\int_0^{{\sigma _{q - 1}}} { \ldots \int_0^{{\sigma _2}} {\int_0^{{\sigma _1}} {\phi (x,{\sigma _0}){\mkern 1mu} d{\sigma _0}d{\sigma _1} \ldots d{\sigma _{q - 2}}d{\sigma _{q - 1}}} } } },
\end{align}
and define the operator
\begin{equation}
	{J_{l,x}} = {I_{2,x}^{(x)}} - {\theta _{l,x}}\,{I_{2,l}^{(x)}}.
\end{equation}
Then we can simply write the unknown solution $u(x,t)$ as follows:
\begin{align}
	u(x,t) &= \int_0^x {\int_0^{{\sigma _2}} {\phi ({\sigma _1},t)\,d{\sigma _1}\,d{\sigma _2}}  - \theta _{l,x}\, \int_0^l {\int_0^{{\sigma _2}} {\phi ({\sigma _1},t)\,d{\sigma _1}\,d{\sigma _2}} } }  + \psi (x,t) = {J_{l,x}}\, \phi(x,t) + \psi(x,t) \label{sec:theshi2:eq:unknsol1}.
\end{align}
Hence Equation \eqref{sec:theshi2:eq:tel2} becomes
\begin{equation}\label{sec:theshi2:eq:intpdek1}
\Omega \,\phi  = \Psi ,
\end{equation}
where
\begin{align}
	\Omega  &= {J_{l,x}} + \beta_1 \,I_{1,t}^{(t)}\,{J_{l,x}} + \beta_2 \,I_{2,t}^{(t)}\,{J_{l,x}} - I_{2,t}^{(t)},\label{sec:theshi2:eq:omegpsi1}\\
	\Psi  &= \hat \psi  - \left( {\beta_1 \,I_{1,t}^{(t)} + \beta_2 \,I_{2,t}^{(t)}} \right)\,\psi  + I_{2,t}^{(t)}f;\label{sec:theshi2:eq:omegpsi2}
\end{align}
\begin{equation}
		\hat \psi (x,t) = \kappa (x,t) - \psi (x,t).	
\end{equation}
If we expand the unknown function $\phi(x,t)$, in a truncated series of bivariate shifted Gegenbauer polynomials as follows:
\begin{equation}
	\phi (x,t) \approx \sum\limits_{n = 0}^{{N_x}} {\sum\limits_{m = 0}^{{N_t}} {{{\hat \phi }_{n,m}}\,{}_{l,\tau }C_{n,m}^{(\alpha )}(x,t)} },
\end{equation}
then we can compute the continuous coefficients of the truncation, ${{\hat \phi }_{n,m}}, n = 0, \ldots, N_x; m = 0, \ldots, N_t$, using the two dimensional weighted inner product as follows:
\begin{equation}
{{\hat \phi }_{n,m}} = \frac{{{{\left( {\phi ,{}_{l,\tau }C_{n,m}^{(\alpha )}(x,t)} \right)}_{w_{l,\tau }^{(\alpha )}}}}}{{\left\| {{}_{l,\tau }C_{n,m}^{(\alpha )}} \right\|_{w_{l,\tau }^{(\alpha )}}^2}} = \frac{{\int_0^\tau  {\int_0^l {\phi (x,t)\,{}_{l,\tau }C_{n,m}^{(\alpha )}(x,t)\,w_{l,\tau }^{(\alpha )}(x,t)\,dx\,dt} } }}{{\int_0^\tau  {\int_0^l {{{\left( {{}_{l,\tau }C_{n,m}^{(\alpha )}(x,t)} \right)}^2}\,w_{l,\tau }^{(\alpha )}(x,t)\,dx\,dt} } }}.
\end{equation}
Instead, we lay a grid of SGG nodes, $\left(x_{l,N_x,i}^{(\alpha)}, t_{\tau,N_t,j}^{(\alpha)}\right), i = 0, \ldots, N_x; j = 0, \ldots, N_t$, on the rectangular domain $D_{l,\tau }^2$, and approximate the function $\phi$ by interpolation at those nodes. Let
\begin{equation}
	{\phi _{s,k}} = \phi \left(x_{l,{N_x},s}^{(\alpha )},t_{\tau ,{N_t},k}^{(\alpha )}\right),\,s = 0, \ldots ,{N_x};k = 0, \ldots ,{N_t}.
\end{equation}
 The polynomial interpolant of $\phi$ in two dimensions can be written in terms of the discrete coefficients $\tilde \phi_{n,m}$, or in the equivalent Lagrange form as follows:
\begin{equation}\label{sec:theshi2:eq:modlag1}
{P_{{N_x},{N_t}}}\phi (x,t) = \sum\limits_{n = 0}^{{N_x}} {\sum\limits_{m = 0}^{{N_t}} {{{\tilde \phi }_{n,m}}\,{}_{l,\tau }C_{n,m}^{(\alpha )}(x,t)} }  = \sum\limits_{s = 0}^{{N_x}} {\sum\limits_{k = 0}^{{N_t}} {{\phi _{s,k}}\,{}_{l,\tau }\mathcal{L} _{{N_x},{N_t},s,k}^{(\alpha )}(x,t)} } ,
\end{equation}
where
\begin{equation}\label{sec:theshi2:eq:coeff22}
{}_{l,\tau }\mathcal{L} _{{N_x},{N_t},s,k}^{(\alpha )}(x,t) = \mathcal{L} _{l,{N_x},s}^{(\alpha )}(x)\,\mathcal{L} _{\tau ,{N_t},k}^{(\alpha )}(t).
\end{equation}
Clearly,
\begin{equation}
	{}_{l,\tau }\mathcal{L} _{{N_x},{N_t},s,k}^{(\alpha )}(x_{l,{N_x},i}^{(\alpha )},t_{\tau ,{N_t},j}^{(\alpha )}) = {\delta _{i,s}}\,{\delta _{j,k}}, i, s = 0, \ldots, N_x; j, k = 0, \ldots, N_t,
\end{equation}
and by construction, we find that
\begin{align}
\int_0^{x_{l,{N_x},i}^{(\alpha )}} {{P_{{N_x},{N_t}}}\phi (x,t_{\tau ,{N_t},j}^{(\alpha )})\,dx}  &= \sum\limits_{s = 0}^{{N_x}} {\sum\limits_{k = 0}^{{N_t}} {{\phi _{s,k}}\,\int_0^{x_{l,{N_x},i}^{(\alpha )}} {{}_{l,\tau }\mathcal{L} _{{N_x},{N_t},s,k}^{(\alpha )}(x,t_{\tau ,{N_t},j}^{(\alpha )})\,dx} } } \nonumber\\
 &= \sum\limits_{s = 0}^{{N_x}} {\hat p_{l,i,s}^{(1)}\,{\phi _{s,j}}} ,\quad i = 0, \ldots ,{N_x};j = 0, \ldots ,{N_t},\\
\int_0^{t_{\tau ,{N_t},j}^{(\alpha )}} {{P_{{N_x},{N_t}}}\phi (x_{l,{N_x},i}^{(\alpha )},t)\,dt}  &= \sum\limits_{s = 0}^{{N_x}} {\sum\limits_{k = 0}^{{N_t}} {{\phi _{s,k}}\,\int_0^{t_{\tau ,{N_t},j}^{(\alpha )}} {{}_{l,\tau }\mathcal{L} _{{N_x},{N_t},s,k}^{(\alpha )}(x_{l,{N_x},i}^{(\alpha )},t)\,dt} } } \nonumber\\
 &= \sum\limits_{k = 0}^{{N_t}} {\hat p_{\tau ,j,k}^{(1)}\,{\phi _{k,i}}} ,\quad i = 0, \ldots ,{N_x};j = 0, \ldots ,{N_t}.
\end{align}
To determine the bivariate discrete shifted Gegenbauer transform, we can first determine the intermediate values, ${{\bar \phi }_n}\left( {t_{\tau ,{N_t},k}^{(\alpha )}} \right), n = 0, \ldots ,{N_x};k = 0, \ldots ,{N_t}$, in the x-direction such that
\begin{equation}
	{{\bar \phi }_n}\left( {t_{\tau ,{N_t},k}^{(\alpha )}} \right) = \frac{1}{{\left\| {C_{l,n}^{(\alpha )}} \right\|_{w_l^{(\alpha )}}^2}}\,\sum\limits_{s = 0}^{{N_x}} {{\phi _{s,k}}\,C_{l,n}^{(\alpha )}(x_{l,{N_x},s}^{(\alpha )})\,\varpi _{l,{N_x},s}^{(\alpha )}} ,\quad n = 0, \ldots ,{N_x};k = 0, \ldots ,{N_t}.
\end{equation}
Therefore,
\begin{equation}\label{sec:theshi2:eq:disccoeffkimo1}
{{\tilde \phi }_{n,m}} = \frac{1}{{\left\| {_{l,\tau }C_{n,m}^{(\alpha )}} \right\|_{w_{l,\tau }^{(\alpha )}}^2}}{\mkern 1mu} \sum\limits_{s = 0}^{{N_x}} {\sum\limits_{k = 0}^{{N_t}} {_{l,\tau }\varpi _{{N_x},{N_t},s,k}^{(\alpha )}{\mkern 1mu} {\phi _{s,k}}{\,_{l,\tau }}C_{n,m}^{(\alpha )}\left(x_{l,{N_x},s}^{(\alpha )},t_{\tau ,{N_t},k}^{(\alpha )}\right)} } ,\quad n = 0, \ldots ,{N_x};m = 0, \ldots ,{N_t},
\end{equation}
where
\begin{equation}
	{}_{l,\tau }\varpi _{{N_x},{N_t},s,k}^{(\alpha )} = \varpi _{l,{N_x},s}^{(\alpha )}\,\varpi _{\tau ,{N_t},k}^{(\alpha )},\quad \quad s = 0, \ldots ,{N_x};k = 0, \ldots ,{N_t},
\end{equation}
are the two-dimensional Christoffel numbers corresponding to the SGG nodes, $\left(x_{l,N_x,i}^{(\alpha)}, t_{\tau,N_t,j}^{(\alpha)}\right), i = 0, \ldots, N_x; j = 0, \ldots, N_t$. To find the equations for the grid point values $\phi_{i,j}$, we require that $\phi$ satisfies the integral formulation of the hyperbolic telegraph PDE \eqref{sec:theshi2:eq:intpdek1} at the interior SGG nodes such that
\begin{equation}\label{sec:theshi2:hyppde2}
{}_{l,\tau,}\Omega _{i,j }\,{\phi _{i,j}} = {}_{l,\tau}\Psi _{i,j},\quad i = 0, \ldots ,{N_x};j = 0, \ldots ,{N_t},
\end{equation}
where
\begin{gather}
{}_{l,\tau,}\Omega _{i,j }\,{\phi _{i,j}} = {J_{l,x_{l,{N_x},i}^{(\alpha )}}}\,\phi \left( {x,t_{\tau ,{N_t},j}^{(\alpha )}} \right) - I_{2,t_{\tau ,{N_t},j}^{(\alpha )}}^{(t)}\,\phi \left( {x_{l,{N_x},i}^{(\alpha )},t} \right) + \left( {\beta_1 \,I_{1,t_{\tau ,{N_t},j}^{(\alpha )}}^{(t)}\,{J_{l,x_{l,{N_x},i}^{(\alpha )}}} + \beta_2 \,I_{2,t_{\tau ,{N_t},j}^{(\alpha )}}^{(t)}\,{J_{l,x_{l,{N_x},i}^{(\alpha )}}}} \right)\phi (x,t),\nonumber\\
\hspace{11.1cm} i = 0, \ldots ,{N_x};j = 0, \ldots ,{N_t},\\
{}_{l,\tau}\Psi _{i,j} = {{\hat \psi }_{i,j}} - \left( {\beta_1 \,I_{1,t_{\tau ,{N_t},j}^{(\alpha )}}^{(t)} + \beta_2 \,I_{2,t_{\tau ,{N_t},j}^{(\alpha )}}^{(t)}} \right)\,\psi \left( {x_{l,{N_x},i}^{(\alpha )},t} \right) + I_{2,t_{\tau ,{N_t},j}^{(\alpha )}}^{(t)}{f\left( {x_{l,{N_x},i}^{(\alpha )},t} \right)},\quad i = 0, \ldots ,{N_x};j = 0, \ldots ,{N_t},\\
	{{\hat \psi }_{i,j}} = \hat \psi \left( {x_{l,{N_x},i}^{(\alpha )},t_{\tau ,{N_t},j}^{(\alpha )}} \right),\quad i = 0, \ldots ,{N_x};j = 0, \ldots ,{N_t}.
\end{gather}
Let
\begin{equation}
	x_{l,{N_x},{N_x} + 1}^{(\alpha )} = l;
	\end{equation}
\begin{equation}
{\psi _{i,(j,k)}} = \psi \left( {x_{l,{N_x},i}^{(\alpha )},z_{\tau ,{M_t},j,k}^{(\alpha _i^*)}} \right),\quad i = 0, \ldots ,{N_x};j = 0, \ldots ,{N_t};k = 0, \ldots ,{M_t},
\end{equation}
for some $M_t \in \mathbb{Z}^+$. We can approximate the terms in Equation \eqref{sec:theshi2:hyppde2} using the S-quadrature and the optimal S-quadrature as follows:
\begin{subequations}\label{sec:theshi2:subeqs1}
\begin{align}
	{J_{l,x_{l,{N_x},i}^{(\alpha )}}}\,\phi \left( {x,t_{\tau ,{N_t},j}^{(\alpha )}} \right) &\approx {{\tilde J}_{l,x_{l,{N_x},i}^{(\alpha )}}}\,\phi \left( {x,t_{\tau ,{N_t},j}^{(\alpha )}} \right) = \left( {\sum\limits_{k = 0}^{{N_x}} {\hat p_{l,i,k}^{(2)}}  - {\theta _{l,x_{l,{N_x},i}^{(\alpha )}}}\,\sum\limits_{k = 0}^{{N_x}} {\hat p_{l,{N_x} + 1,k}^{(2)}} } \right)\,{\phi _{k,j}},\\
	I_{2,t_{\tau ,{N_t},j}^{(\alpha )}}^{(t)}\,\phi \left( {x_{l,{N_x},i}^{(\alpha )},t} \right) &\approx \tilde I_{2,t_{\tau ,{N_t},j}^{(\alpha )}}^{(t)}\,\phi \left( {x_{l,{N_x},i}^{(\alpha )},t} \right) = \sum\limits_{k = 0}^{{N_t}} {\hat p_{\tau ,j,k}^{(2)}\,{\phi _{i,k}}} ,\\
I_{1,t_{\tau ,{N_t},j}^{(\alpha )}}^{(t)}\,{J_{l,x_{l,{N_x},i}^{(\alpha )}}}\,\phi (x,t) &\approx \tilde I_{1,t_{\tau ,{N_t},j}^{(\alpha )}}^{(t)}\,{\tilde J_{l,x_{l,{N_x},i}^{(\alpha )}}}\,\phi (x,t) = \sum\limits_{k = 0}^{{N_t}} {\hat p_{\tau ,j,k}^{(1)}\,\left( {\sum\limits_{s = 0}^{{N_x}} {\hat p_{l,i,s}^{(2)}}  - {\theta _{l,x_{l,{N_x},i}^{(\alpha )}}}\,\sum\limits_{s = 0}^{{N_x}} {\hat p_{l,{N_x} + 1,s}^{(2)}} } \right)\,{\phi _{s,k}}},\\
I_{2,t_{\tau ,{N_t},j}^{(\alpha )}}^{(t)}\,{J_{l,x_{l,{N_x},i}^{(\alpha )}}}\,\phi (x,t) &\approx \tilde I_{2,t_{\tau ,{N_t},j}^{(\alpha )}}^{(t)}\,{\tilde J_{l,x_{l,{N_x},i}^{(\alpha )}}}\,\phi (x,t) = \sum\limits_{k = 0}^{{N_t}} {\hat p_{\tau ,j,k}^{(2)}\,\left( {\sum\limits_{s = 0}^{{N_x}} {\hat p_{l,i,s}^{(2)}}  - {\theta _{l,x_{l,{N_x},i}^{(\alpha )}}}\,\sum\limits_{s = 0}^{{N_x}} {\hat p_{l,{N_x} + 1,s}^{(2)}} } \right)\,{\phi _{s,k}}} ,\\
I_{1,t_{\tau ,{N_t},j}^{(\alpha )}}^{(t)}\,\psi \left( {x_{l,{N_x},i}^{(\alpha )},t} \right) &\approx \tilde I_{1,t_{\tau ,{N_t},j}^{(\alpha )}}^{(t)}\,\psi \left( {x_{l,{N_x},i}^{(\alpha )},t} \right) = \sum\limits_{k = 0}^{{M_t}} {p_{\tau ,j,k}^{(1)}\,{\psi _{i,(j,k)}}} ,\\
I_{2,t_{\tau ,{N_t},j}^{(\alpha )}}^{(t)}\,\psi \left( {x_{l,{N_x},i}^{(\alpha )},t} \right) &\approx \tilde I_{2,t_{\tau ,{N_t},j}^{(\alpha )}}^{(t)}\,\psi \left( {x_{l,{N_x},i}^{(\alpha )},t} \right) = \sum\limits_{k = 0}^{{M_t}} {p_{\tau ,j,k}^{(2)}\,{\psi _{i,(j,k)}}} ;\\
I_{2,t_{\tau ,{N_t},j}^{(\alpha )}}^{(t)}f\left( {x_{l,{N_x},i}^{(\alpha )},t} \right) &\approx \tilde I_{2,t_{\tau ,{N_t},j}^{(\alpha )}}^{(t)}f\left( {x_{l,{N_x},i}^{(\alpha )},t} \right) = \sum\limits_{k = 0}^{{M_t}} {p_{\tau ,j,k}^{(2)}\,{f_{i,(j,k)}}} .
\end{align}
\end{subequations}

Hence the discrete analogue of Equations \eqref{sec:theshi2:hyppde2} can be written as:
\begin{align}
	&{{\tilde J}_{l,x_{l,{N_x},i}^{(\alpha )}}}\,\phi \left( {x,t_{\tau ,{N_t},j}^{(\alpha )}} \right) - \tilde I_{2,t_{\tau ,{N_t},j}^{(\alpha )}}^{(t)}\,\phi \left( {x_{l,{N_x},i}^{(\alpha )},t} \right) + \left( {\beta_1 \,\tilde I_{1,t_{\tau ,{N_t},j}^{(\alpha )}}^{(t)}\,{{\tilde J}_{l,x_{l,{N_x},i}^{(\alpha )}}} + \beta_2 \,\tilde I_{2,t_{\tau ,{N_t},j}^{(\alpha )}}^{(t)}\,{{\tilde J}_{l,x_{l,{N_x},i}^{(\alpha )}}}} \right)\phi (x,t) = {{\hat \psi }_{i,j}}\nonumber\\
& - \left( {\beta_1 \,\tilde I_{1,t_{\tau ,{N_t},j}^{(\alpha )}}^{(t)} + \beta_2 \,\tilde I_{2,t_{\tau ,{N_t},j}^{(\alpha )}}^{(t)}} \right)\,\psi \left( {x_{l,{N_x},i}^{(\alpha )},t} \right) + \tilde I_{2,t_{\tau ,{N_t},j}^{(\alpha )}}^{(t)}f\left( {x_{l,{N_x},i}^{(\alpha )},t} \right),
\end{align}
or simply in shorthand notation as:
\begin{equation}\label{sec:theshi2:eq:phisol1}
{}_{l,\tau}\Omega _{i,j }^{{N_x},{N_t}}\,{\phi _{i,j}} = {}_{l,\tau}\Psi _{i,j}^{{M_t}},\quad i = 0, \ldots ,{N_x};j = 0, \ldots ,{N_t}.
\end{equation}
The solution of the linear system \eqref{sec:theshi2:eq:phisol1} provides the values of $\phi _{i,j}, i = 0, \ldots ,{N_x};j = 0, \ldots ,{N_t},$ at the
SGG nodes, $\left(x_{l,N_x,i}^{(\alpha)}, t_{\tau,N_t,j}^{(\alpha)}\right), i = 0, \ldots, N_x; j = 0, \ldots, N_t$; hence the discrete coefficients $\tilde \phi_{n,m}, i = 0, \ldots, N_x; j = 0, \ldots, N_t$, from Equation \eqref{sec:theshi2:eq:disccoeffkimo1}. Notice that the coefficient matrix, ${}_{l,\tau}\Omega _{i,j }^{{N_x},{N_t}}$, of the linear system \eqref{sec:theshi2:eq:phisol1} generated by the SGPM is full, but in compensation, the high order of the basis functions gives high accuracy for given $N_x, N_t; M_t$.

Using Equations \eqref{sec:theshi2:eq:unknsol1} and \eqref{sec:theshi2:eq:modlag1}, we can approximate the unknown solution $u$ by the bivariate shifted Gegenbauer interpolant, ${P_{{N_x},{N_t}}}u$, as follows:
\begin{equation}\label{sec:theshi2:eq:solentirext1}
u(x,t) \approx {P_{{N_x},{N_t}}}u(x,t) = \sum\limits_{n = 0}^{{N_x}} {\sum\limits_{m = 0}^{{N_t}} {{{\tilde \phi }_{n,m}}{\mkern 1mu} {J_{l,x}}\; {}_{l,\tau }C_{n,m}^{(\alpha )}(x,t)} }  + \psi (x,t).
\end{equation}
Hence the approximate values of the unknown solution $u$ can be determined at the SGG nodes, $\left(x_{l,N_x,i}^{(\alpha)}, t_{\tau,N_t,j}^{(\alpha)}\right)$, $i = 0, \ldots, N_x; j = 0, \ldots, N_t$, through the following equations:
\begin{align}
u\left( {x_{l,{N_x},i}^{(\alpha )},t_{\tau ,{N_t},j}^{(\alpha )}} \right) &\approx {P_{{N_x},{N_t}}}u\left( {x_{l,{N_x},i}^{(\alpha )},t_{\tau ,{N_t},j}^{(\alpha )}} \right) = \left( {\sum\limits_{j = 0}^{{N_x}} {\hat p_{l,i,j}^{(2)}}  - {\theta _{l,x_{l,{N_x},i}^{(\alpha )}}}\,\sum\limits_{j = 0}^{{N_x}} {\hat p_{l,{N_x} + 1,j}^{(2)}} } \right)\,{\phi _{i,j}} + {\psi _{i,j}}\nonumber\\
 &= \sum\limits_{n = 0}^{{N_x}} {\sum\limits_{m = 0}^{{N_t}} {{{\tilde \phi }_{n,m}}\,C_{\tau ,m}^{(\alpha )}(t_{\tau ,{N_t},j}^{(\alpha )})\,\left( {\sum\limits_{k = 0}^{{N_x}} {\hat p_{l,i,k}^{(2)}}  - {\theta _{l,x_{l,{N_x},i}^{(\alpha )}}}\,\sum\limits_{k = 0}^{{N_x}} {\hat p_{l,{N_x} + 1,k}^{(2)}} } \right)\,C_{l,n}^{(\alpha )}(x_{l,{N_x},k}^{(\alpha )})} }  + {\psi _{i,j}},\nonumber\\
 &{}\hspace{9cm} i = 0, \ldots ,{N_x}; j = 0, \ldots ,{N_t} \label{sec:theshi2:eq:phisol12},
\end{align}
where
\begin{equation}
	{\psi _{i,j}} = \psi \left( {x_{l,{N_x},i}^{(\alpha )},t_{\tau ,{N_t},j}^{(\alpha )}} \right),\quad i = 0, \ldots ,{N_x};j = 0, \ldots ,{N_t}.
\end{equation}
Since Formula \eqref{subsec:err:eq:squadki1} is exact for all polynomials $h_n(x) \in \mathbb{P}_n$, Equations \eqref{sec:theshi2:eq:phisol12} provide exact formulae for the bivariate shifted Gegenbauer interpolant, ${P_{{N_x},{N_t}}}u$, at the SGG nodes, $\left(x_{l,N_x,i}^{(\alpha)}, t_{\tau,N_t,j}^{(\alpha)}\right), i = 0, \ldots, N_x; j = 0, \ldots, N_t$.
\subsection{Global Collocation Matrix and Right Hand Side Constructions for Solving the Collocation Equations}
\label{subsec:mrhs11}
To put the pointwise representation of the linear system \eqref{sec:theshi2:eq:phisol1} into the standard matrix system form $\mathbf{A x} = \mathbf{b}$, we introduce the mapping $n = \text{index}(i, j): n = i + j\, (N_x + 1)$. Thus the matrix elements of the global collocation matrix $\mathbf{A}$ can be calculated by the following equations:
\begin{subequations}\label{eq:formmatcons11}
\begin{align}
{\mathbf{A}_{{\text{index}}\left( {i,j} \right),{\text{index}}\left( {s,k} \right)}} &= \left( {\hat p_{l,i,s}^{(2)} - {\theta _{l,x_{l,{N_x},i}^{(\alpha )}}}\hat p_{l,{N_x} + 1,s}^{(2)}} \right)\,\left( {{\beta _1}\hat p_{\tau ,j,k}^{(1)} + {\beta _2}\hat p_{\tau ,j,k}^{(2)}} \right),\;s = 0, \ldots ,{N_x};k = 0, \ldots ,{N_t};s \ne i,k \ne j,\\
{\mathbf{A}_{{\text{index}}\left( {i,j} \right),{\text{index}}\left( {k,j} \right)}} &= \left( {\hat p_{l,i,k}^{(2)} - {\theta _{l,x_{l,{N_x},i}^{(\alpha )}}}\,\hat p_{l,{N_x} + 1,k}^{(2)}} \right)\,\left( {{\beta _1}\,\hat p_{\tau ,j,j}^{(1)} + {\beta _2}\,\hat p_{\tau ,j,j}^{(2)} + 1} \right),\;k = 0, \ldots ,{N_x};k \ne i,\\
{\mathbf{A}_{{\text{index}}\left( {i,j} \right),{\text{index}}\left( {i,k} \right)}} &= \left( {\hat p_{l,i,i}^{(2)} - {\theta _{l,x_{l,{N_x},i}^{(\alpha )}}}\,\hat p_{l,{N_x} + 1,i}^{(2)}} \right)\,\left( {{\beta _1}\,\hat p_{\tau ,j,k}^{(1)} + {\beta _2}\,\hat p_{\tau ,j,k}^{(2)}} \right) - \hat p_{\tau ,j,k}^{(2)},\;k = 0, \ldots ,{N_t};k \ne j;\\
{\mathbf{A}_{{\text{index}}\left( {i,j} \right),{\text{index}}\left( {i,j} \right)}} &= \left( {\hat p_{l,i,i}^{(2)} - {\theta _{l,x_{l,{N_x},i}^{(\alpha )}}}\,\hat p_{l,{N_x} + 1,i}^{(2)}} \right)\,\left( {{\beta _1}\,\hat p_{\tau ,j,j}^{(1)} + {\beta _2}\,\hat p_{\tau ,j,j}^{(2)}} + 1 \right) - \hat p_{\tau ,j,j}^{(2)}.
\end{align}
\end{subequations}
Algorithm \ref{sec:theshi2:alg1matrix} implements these formulas, and computes the right hand side of the linear system \eqref{sec:theshi2:eq:phisol1} by arranging the two-dimensional array $(\text{RHS})_{i,j} = \Psi _{l,\tau ,i,j}^{{M_t}}$, in the form of a vector array, $\{(\text{RHS})_n\}_{n=0}^L$. Clearly, the construction of the global collocation matrix $\mathbf{A}$ requires the storage of its $(L + 1) \times (L + 1)$ elements, where $L = N_x + N_t + N_x N_t$. It can be shown that the construction of the matrix $\mathbf{A}$ requires exactly \[
1 + \left( {1 + {N_x}} \right)\left( {1 + 5{{\left( {1 + {N_t}} \right)}^2}\left( {1 + {N_x}} \right)} \right),\]
 multiplications and divisions, and
\[2 + \left( {1 + {N_t}} \right)\left( {1 + {N_x}} \right)\left( {6 + 5{N_x} + {N_t}\left( {5 + 4{N_x}} \right)} \right),\]
 additions and subtractions, for a total number of
\[N_t^2\left( {{N_x} + 1} \right)\left( {9{N_x} + 10} \right) + {N_t}\left( {{N_x} + 1} \right)\left( {19{N_x} + 21} \right) + 2{N_x}\left( {5{N_x} + 11} \right) + 15,\]
flops. On the other hand, the construction of the right hand side requires
\[\left( {5 + 3{M_t}} \right)\left( {1 + {N_t}} \right)\left( {1 + {N_x}} \right),\]
multiplications, and $3$ additions and subtractions, for a total number of
\[3 + \left( {5 + 3{M_t}} \right)\left( {1 + {N_t}} \right)\left( {1 + {N_x}} \right),\]
flops. Hence, the total computational cost (TCC) of Algorithm \ref{sec:theshi2:alg1matrix} is given by
\begin{align*}
	\text{TCC} &= 3{M_t}\left( {{N_t} + 1} \right)\left( {{N_x} + 1} \right) + N_t^2\left( {{N_x} + 1} \right)\left( {9{N_x} + 10} \right) + {N_t}\left( {{N_x} + 1} \right)\left( {19{N_x} + 26} \right) + {N_x}\left( {10{N_x} + 27} \right) + 23\\
	&= O\left( {{N_x}{N_t}({N_x}{N_t} + {M_t})} \right),\quad \text{as } {N_x},{N_t},{M_t} \to \infty.
\end{align*}
For systems of small or moderate size, the solution by a direct solver is easy to implement; however, for large grids, the storage requirement of the global collocation matrix elements could be prohibitive, making fast iterative solvers more appropriate. 
Fortunately, the present numerical scheme converges exponentially fast for sufficiently smooth solutions using relatively small number of grids as we show later in Sections \ref{sec:conerr1} and \ref{sec:numerical}. In Section \ref{sec:numerical} we show also through a numerical example that the time complexity required for the calculation of the approximate solution ${P_{{N_x},{N_t}}}u$, at the collocation points $\left(x_{l,N_x,i}^{(\alpha)}, t_{\tau,N_t,j}^{(\alpha)}\right), i = 0, \ldots, N_x; j = 0, \ldots, N_t$, using a direct solver implementing Algorithm \ref{sec:theshi2:alg1matrix} is approximately of $O\left((L + 1)^{2}\right)$, as $L \to \infty$, for relatively small values of $M_t$.
\subsection{Global Approximate Solution Over the Whole Solution Domain}
\label{subsec:GAS1}
To approximate the unknown solution $u$ at any point $(x, t) \in D_{l,\tau }^2$, through Equation \eqref{sec:theshi2:eq:solentirext1}, we need to calculate the integrals ${{J_{l,x}}\,C_{l,n}^{(\alpha )}(x)}, n = 0, \ldots, N_x$. Integrating Equations (A.11) in \cite{Elgindy2013} on $[-1, 1]$, yield the following equations:
\begin{align}
	\int_{ - 1}^x {\int_{ - 1}^{{\sigma _1}} {C_0^{(\alpha )}({\sigma _0})\,d{\sigma _0}\,d{\sigma _1}} }  &= \frac{1}{2}{\left( {1 + x} \right)^2} = C_0^{(\alpha )}(x) + C_1^{(\alpha )}(x) + d_1^{(\alpha )}\,\left( {C_2^{(\alpha )}(x) - C_0^{(\alpha )}(x)} \right),\\
	\int_{ - 1}^x {\int_{ - 1}^{{\sigma _1}} {C_1^{(\alpha )}({\sigma _0})\,d{\sigma _0}\,d{\sigma _1}} } &= \frac{1}{6}\left( { - 2 + x} \right){\left( {1 + x} \right)^2} = \frac{1}{{12\,{{(\alpha  + 1)}_2}}}\,\left( - 3\,(\alpha  + 2)\,(2\alpha  + 1)\,C_0^{(\alpha )}(x) \right.\nonumber\\
	&\left. - 3\,(\alpha  + 1)\,(2\alpha  + 3)\,C_1^{(\alpha )}(x) + (\alpha  + 1)\,(2\alpha  + 1)\,C_3^{(\alpha )}(x) + (\alpha  + 2)\,(2\alpha  - 1) \right),
\end{align}
	\begin{empheq}[left={\int_{ - 1}^x {\int_{ - 1}^{{\sigma _1}} {C_j^{(\alpha )}({\sigma _0})\,d{\sigma _0}\,d{\sigma _1}} } =}\empheqlbrace]{align}
	&\frac{1}{6}\left( { - 2 + x} \right)x{\left( {1 + x} \right)^2},\quad j = 2 \wedge \alpha  = 0,\nonumber\\
	&\frac{1}{{4\,(\alpha  + j)}}\,\left( {\frac{1}{{\alpha  + j + 1}}\left( {d_{2,j}^{(\alpha )}\,C_{j + 2}^{(\alpha )}(x) + d_{3,j}^{(\alpha )}\,C_j^{(\alpha )}(x)} \right) + d_{4,j}^{(\alpha )}\,C_{j - 2}^{(\alpha )}(x) + \Im _j^{(\alpha )}(x)} \right),\nonumber\\
	&j \ge 2:j \ne 2 \vee \alpha  \ne 0\nonumber
	\end{empheq}
\begin{empheq}[left={\hspace{3.9cm}=}\empheqlbrace]{align}
&\frac{1}{{48}}\left( { - 9\,{T_0}(x) - 8\left( {2\,{T_1}(x) + {T_2}(x)} \right) + {T_4}(x)} \right),\quad j = 2 \wedge \alpha  = 0,\\
&\frac{1}{{4\,(\alpha  + j)}}\,\left( {\frac{1}{{\alpha  + j + 1}}\left( {d_{2,j}^{(\alpha )}\,C_{j + 2}^{(\alpha )}(x) + d_{3,j}^{(\alpha )}\,C_j^{(\alpha )}(x)} \right) + d_{4,j}^{(\alpha )}\,C_{j - 2}^{(\alpha )}(x) + \Im _j^{(\alpha )}(x)} \right),\nonumber\\
&j \ge 2:j \ne 2 \vee \alpha  \ne 0,\nonumber
	\end{empheq}

\vspace{-1.5cm}
where
\begin{subequations}
\begin{gather}
	d_1^{(\alpha)} = \frac{{1 + 2\alpha }}{{4\,(1 + \alpha )}},\\
	d_{2,j}^{(\alpha )} = \frac{{{{(2\alpha  + j)}_2}}}{{{{(j + 1)}_2}}},\\
	d_{3,j}^{(\alpha )} =  - \frac{{2\,(\alpha  + j)}}{{\alpha  + j - 1}},\\
	d_{4,j}^{(\alpha )} = \frac{{{{(j - 1)}_2}}}{{(\alpha  + j - 1)\,{{(2\alpha  + j - 2)}_2}}},\,\,j \ne 2 \vee \alpha  \ne 0,\\
\Im _j^{(\alpha )}(x) = \frac{{4\,{{( - 1)}^j}\,(2\alpha  - 1)\,(\alpha  + j)}}{{{{(j + 1)}_2}\,{{(2\alpha  + j - 2)}_2}}}\,\left( {({j^2} + 2\alpha \,(j + 2) - 4)\,x + {j^2} + 2\alpha \,(j + 1) - 1} \right),
\end{gather}
\end{subequations}
and ${(x)_n} = \Gamma (x + n)/\Gamma (x) = x\,(x + 1)\, \ldots \,(x + n - 1)$, is the Pochhammer symbol (rising factorial). Therefore, the double integrals of $C_{l,j}^{(\alpha)}(x), j = 0, 1, \ldots$, on $[0, l]$, can be calculated as follows:
\begin{align}
\int_0^x {\int_0^{{\sigma _1}} {C_{l,0}^{(\alpha )}({\sigma _0})\,d{\sigma _0}\,d{\sigma _1}} }  &= \frac{1}{2}{x^2} = \frac{{{l^2}}}{4}\left( {C_{l,0}^{(\alpha )}(x) + C_{l,1}^{(\alpha )}(x) + d_1^{(\alpha )}\,\left( {C_{l,2}^{(\alpha )}(x) - C_{l,0}^{(\alpha )}(x)} \right)} \right),\\
\int_0^x {\int_0^{{\sigma _1}} {C_{l,1}^{(\alpha )}({\sigma _0})\,d{\sigma _0}\,d{\sigma _1}} }  &= \frac{1}{{6l}}{x^2}\,( - 3\,l + 2x) = \frac{{{l^2}}}{{48\,{{(\alpha  + 1)}_2}}}\,\left( - 3\,(\alpha  + 2)\,(2\alpha  + 1)\,C_{l,0}^{(\alpha )}(x) - 3\,(\alpha  + 1)\,(2\alpha  + 3)\,C_{l,1}^{(\alpha )}(x)\right.\nonumber\\
&\left. + (\alpha  + 1)\,(2\alpha  + 1)\,C_{l,3}^{(\alpha )}(x) + (\alpha  + 2)\,(2\alpha  - 1) \right),
\end{align}
\begin{empheq}[left={\int_0^x {\int_0^{{\sigma _1}} {C_{l,j}^{(\alpha )}({\sigma _0})\,d{\sigma _0}\,d{\sigma _1}} } =}\empheqlbrace]{align}
&\frac{1}{{6{l^2}}}\left( {l - 2x} \right)\left( {3l - 2x} \right){x^2},\quad j = 2 \wedge \alpha  = 0,\nonumber\\
&\frac{{{l^2}}}{{16\,(\alpha  + j)}}\,\left( {\frac{1}{{\alpha  + j + 1}}\left( {d_{2,j}^{(\alpha )}\,C_{l,j + 2}^{(\alpha )}(x) + d_{3,j}^{(\alpha )}\,C_{l,j}^{(\alpha )}(x)} \right) + d_{4,j}^{(\alpha )}\,C_{l,j - 2}^{(\alpha )}(x) + \Im _{l,j}^{(\alpha )}(x)} \right),\nonumber\\
&j \ge 2:j \ne 2 \vee \alpha  \ne 0\nonumber
\end{empheq}
\begin{empheq}[left={\hspace{3.8cm}=}\empheqlbrace]{align}
&\frac{{{l^2}}}{{192}}\left( { - 9\,{T_{l,0}}(x) - 8\left( {2\,{T_{l,1}}(x) + {T_{l,2}}(x)} \right) + {T_{l,4}}(x)} \right),\quad j = 2 \wedge \alpha  = 0,\\
&\frac{{{l^2}}}{{16\,(\alpha  + j)}}\,\left( {\frac{1}{{\alpha  + j + 1}}\left( {d_{2,j}^{(\alpha )}\,C_{l,j + 2}^{(\alpha )}(x) + d_{3,j}^{(\alpha )}\,C_{l,j}^{(\alpha )}(x)} \right) + d_{4,j}^{(\alpha )}\,C_{l,j - 2}^{(\alpha )}(x) + \Im _{l,j}^{(\alpha )}(x)} \right),\nonumber\\
&j \ge 2:j \ne 2 \vee \alpha  \ne 0,\nonumber
\end{empheq}

where $T_{l,n}(x)$, is the $n$th-degree shifted Chebyshev polynomial of the first kind, and $\Im _{l,j}^{(\alpha )}(x) = \Im _j^{(\alpha )}(2x/l - 1)\,\forall j$. Hence,
\begin{align}
	{J_{l,x}}C_{l,0}^{(\alpha )}(x) &= \frac{1}{2}x\,(x - l) = \frac{{{l^2}}}{4}d_1^{(\alpha )}\left( {C_{l,2}^{(\alpha )}(x) - C_{l,0}^{(\alpha )}(x)} \right),\\
	{J_{l,x}}C_{l,1}^{(\alpha )}(x) &= \frac{1}{{6l}}\,(l - 2x)\,(l - x)\,x = - \frac{{{l^2}\left( {1 + 2\alpha } \right)}}{{48\left( {2 + \alpha } \right)}}\,\left( {C_{l,1}^{(\alpha )}(x) - C_{l,3}^{(\alpha )}(x)} \right),
	\end{align}
\begin{empheq}[left={{J_{l,x}}C_{l,j}^{(\alpha )}(x) =}\empheqlbrace]{align}
&\frac{1}{{6{l^2}}}\left( {l - x} \right)x\left( {{l^2} + 4lx - 4{x^2}} \right),\quad j = 2 \wedge \alpha  = 0,\nonumber\\
&\nu _{1,l,j}^{(\alpha )}\,\left( {{\nu _{2,j}}\,C_{l,j - 2}^{(\alpha )}(x) - \nu _{3,j}^{(\alpha )}\left( {\nu _{4,j}^{(\alpha )}\left( {\nu _{5,j}^{(\alpha )}\,C_{l,j}^{(\alpha )}(x) + \nu _{6,j}^{(\alpha )}\,C_{l,j + 2}^{(\alpha )}(x)} \right) + \wp _{l,j}^{(\alpha )}(x)} \right)} \right),\nonumber\\
&j \ge 2:j \ne 2 \vee \alpha  \ne 0\nonumber
\end{empheq}
\begin{empheq}[left={\hspace{2.18cm}=}\empheqlbrace]{align}
&\frac{{{l^2}}}{{192}}\left( { - 9\,{T_{l,0}}(x) - 8\,{T_{l,2}}(x) + {T_{l,4}}(x) + 16} \right),\quad j = 2 \wedge \alpha  = 0,\\
&\nu _{1,l,j}^{(\alpha )}\,\left( {{\nu _{2,j}}\,C_{l,j - 2}^{(\alpha )}(x) - \nu _{3,j}^{(\alpha )}\left( {\nu _{4,j}^{(\alpha )}\left( {\nu _{5,j}^{(\alpha )}\,C_{l,j}^{(\alpha )}(x) + \nu _{6,j}^{(\alpha )}\,C_{l,j + 2}^{(\alpha )}(x)} \right) + \wp _{l,j}^{(\alpha )}(x)} \right)} \right),\nonumber\\
&j \ge 2:j \ne 2 \vee \alpha  \ne 0,\nonumber
\end{empheq}

where
\begin{subequations}
\begin{gather}
	\nu _{1,l,j}^{(\alpha )} = \frac{{{l^2}}}{{16\,{{(\alpha  + j - 1)}_2}\,{{(2\alpha  + j - 2)}_2}}},\\
	{\nu _{2,j}} = {(j - 1)_2},\\
	\nu _{3,j}^{(\alpha )} = \frac{1}{{{{(j + 1)}_2}\,(\alpha  + j + 1)}},\\
	\nu _{4,j}^{(\alpha )} = {(2\alpha  + j - 2)_2},\\
\nu _{5,j}^{(\alpha )} = 2\,{(j + 1)_2}\,(\alpha  + j),\\
\nu _{6,j}^{(\alpha )} =  - (\alpha  + j - 1)\,{(2\alpha  + j)_2};\\
\wp _{l,j}^{(\alpha )}(x) = \frac{4}{l}\,\left( {4\,(\alpha  - 2)\alpha  + 3} \right)\,{\left( {\alpha  + j - 1} \right)_3}\,\left( {{{( - 1)}^j}\,(l - x) + x} \right).
\end{gather}
\end{subequations}
Using Equation \eqref{sec:pre:eq:hyp1}, we can also show without stating the proof that
\begin{align}
{J_{l,x}}C_{l,n}^{(\alpha )}(x) = \frac{{\left( {4\,\left( {\alpha  - 2} \right)\alpha  + 3} \right)\,l}}{{4\,{{(j + 1)}_2}\,{{(2\alpha  + j - 2)}_2}}}\,\left( {l\,{_2}{F_1}\left( { - j - 2,j + 2\alpha  - 2;\alpha  - \frac{3}{2};1 - \frac{x}{l}} \right) + {{( - 1)}^j}\,(x - l) - x} \right),\nonumber\\
\alpha  < \frac{5}{2} \wedge \alpha  \ne \frac{1}{2},\frac{3}{2}.
\end{align}

Using the above formulae, the SGPM directly approximates the solution at any point in the range of integration; on the other hand, finite-difference schemes, for instance, must require a further step of interpolation.

\subsection{Convergence and Error Analysis}
\label{sec:conerr1}
The following theorem gives the bounds on the discrete shifted Gegenbauer coefficients $\tilde \phi_{n,m}\,\forall n,m$.
\begin{thm}\label{sec:conerr1:thm:coeffconv1}
Let $u(x,t) \in C^2\left(D_{l,\tau }^2\right)$, be the solution of Problem $\mathcal{P}$. Suppose also that $u$ is interpolated by the shifted Gegenbauer polynomials at the SGG nodes, $\left(x_{l,N_x,i}^{(\alpha)}, t_{\tau,N_t,j}^{(\alpha)}\right), i = 0, \ldots, N_x; j = 0, \ldots, N_t$, on the rectangular domain $D_{l,\tau }^2$. Then the discrete shifted Gegenbauer coefficients $\tilde \phi_{n,m}$, given by Equations \eqref{sec:theshi2:eq:disccoeffkimo1} are bounded by the following inequalities:
\begin{empheq}[left={\left| {{{\tilde \phi }_{n,m}}} \right| \le}\empheqlbrace]{align}\label{sec:conerr1:ineq4}
	&\frac{{4\,(n + \alpha )\,(m + \alpha )\,\Gamma (n + 2\alpha )\,\Gamma (m + 2\alpha )}}{{{\Gamma ^2}(2\alpha  + 1)\,\Gamma (n + 1)\,\Gamma (m + 1)}}\,{\left\| {{u_{xx}}} \right\|_{{L^\infty }\left( {D_{l,\tau }^2} \right)}},\quad \alpha  \ge 0,\\
	&\frac{{\left| {(n + \alpha )\,(m + \alpha )} \right|}}{{{\alpha ^2}}}\,\left| {\left( \begin{array}{l}
	\frac{n}{2} + \alpha  - 1\\
	\hfill \frac{n}{2} \hfill
	\end{array} \right)\,\left( \begin{array}{l}
	\frac{m}{2} + \alpha  - 1\\
	\hfill \frac{m}{2} \hfill
	\end{array} \right)} \right|\;{\left\| {{u_{xx}}} \right\|_{{L^\infty }\left( {D_{l,\tau }^2} \right)}},\quad  \frac{n}{2},\frac{m}{2} \in \mathbb{Z}_0^ + \wedge - \frac{1}{2} < \alpha  < 0 ,\nonumber\\
	&\frac{{4\,\left| {(n + \alpha )\,(m + \alpha )} \right|\,\Gamma \left( {\frac{{n + 1}}{2} + \alpha } \right)\,\Gamma \left( {\frac{{m + 1}}{2} + \alpha } \right)}}{{{\Gamma ^2}(\alpha  + 1)\,\Gamma \left( {\frac{{n + 1}}{2}} \right)\,\Gamma \left( {\frac{{m + 1}}{2}} \right)\,\sqrt[4]{{{n^2}\,{m^2}\,{{(n + 2\alpha)}^2}\,{{(m + 2\alpha)}^2}}}}}\,{\left\| {{u_{xx}}} \right\|_{{L^\infty }\left( {D_{l,\tau }^2} \right)}},\quad  \frac{n}{2},\frac{m}{2} \notin \mathbb{Z}_0^ +  \wedge - \frac{1}{2} < \alpha  < 0,\nonumber\\
	&\frac{{2\,\Gamma \left( {\frac{{m + 1}}{2} + \alpha } \right)\,\left| {(n + \alpha )\,(m + \alpha )} \right|}}{{{\alpha ^2}\,\sqrt {m\,(m + 2\alpha )} \,\Gamma \left( {\frac{{m + 1}}{2}} \right)\,\Gamma (\alpha )}}\,\left| {\left( \begin{array}{l}
	\frac{n}{2} + \alpha  - 1\\
	\hfill \frac{n}{2} \hfill
	\end{array} \right)} \right|\;{\left\| {{u_{xx}}} \right\|_{{L^\infty }\left( {D_{l,\tau }^2} \right)}},\quad \frac{n}{2} \in \mathbb{Z}_0^ +  \wedge \frac{m}{2} \notin \mathbb{Z}_0^ +  \wedge - \frac{1}{2} < \alpha  < 0,\nonumber\\
	&\frac{{2\,\Gamma \left( {\frac{{n + 1}}{2} + \alpha } \right)\,\left| {(n + \alpha )\,(m + \alpha )} \right|}}{{{\alpha ^2}\,\sqrt {n\,(n + 2\alpha )} \,\Gamma \left( {\frac{{n + 1}}{2}} \right)\,\Gamma (\alpha )}}\,\left| {\left( \begin{array}{l}
	\frac{m}{2} + \alpha  - 1\\
	\hfill \frac{m}{2} \hfill
	\end{array} \right)} \right|\;{\left\| {{u_{xx}}} \right\|_{{L^\infty }\left( {D_{l,\tau }^2} \right)}},\quad \frac{n}{2} \notin \mathbb{Z}_0^ +  \wedge \frac{m}{2} \in \mathbb{Z}_0^ +  \wedge - \frac{1}{2} < \alpha  < 0,\nonumber
\end{empheq}

\vspace{-4.5cm}
Moreover, as $n, m \to \infty$, the discrete shifted Gegenbauer coefficients $\tilde \phi_{n,m}$, are asymptotically bounded by:
\begin{subequations}
\begin{empheq}[left={\left| {{{\tilde \phi }_{n,m}}} \right| \simlt}\empheqlbrace]{align}
&D_1^{^{(\alpha )}}\,{(nm)^{2\alpha }}\,{\left\| {{u_{xx}}} \right\|_{{L^\infty }\left( {D_{l,\tau }^2} \right)}},\quad \alpha  \ge 0,\label{sec:conerr1:asy1}\\
&D_2^{^{(\alpha )}}\,{(nm)^\alpha }\,{\left\| {{u_{xx}}} \right\|_{{L^\infty }\left( {D_{l,\tau }^2} \right)}},\quad  - \frac{1}{2} < \alpha  < 0,\label{sec:conerr1:asy2}
\end{empheq}
\end{subequations}
where
\begin{subequations}\label{sec:conerr1:asy1coeffD12}
\begin{align}
	D_1^{^{(\alpha )}} &= \frac{{{\pi ^2}\,{2^{3 - 4\alpha }}\,{e^{ - 4\alpha  - 2}}}}{{{\Gamma ^2}\left( {\alpha  + \frac{1}{2}} \right)\,{\Gamma ^2}(\alpha  + 1)}},\\
	D_2^{^{(\alpha )}} &= A_1^{(\alpha)}\, A_2^{(\alpha)} \,D_1^{^{(\alpha )}},
\end{align}
\end{subequations}
for some constants $A_1^{(\alpha)}, A_2^{(\alpha)} > 1$, dependent on $\alpha$, but independent of $n$ and $m$.
\end{thm}
\begin{proof}
Since
\begin{align*}
	{{\tilde \phi }_{n,m}} &= \frac{1}{{\left\| {{}_{l,\tau }C_{n,m}^{(\alpha )}} \right\|_{w_{l,\tau }^{(\alpha )}}^2}}\,\sum\limits_{s = 0}^{{N_x}} {\sum\limits_{k = 0}^{{N_t}} {{}_{l,\tau }\varpi _{{N_x},{N_t},s,k}^{(\alpha )}\,{\phi _{s,k}}\,{}_{l,\tau }C_{n,m}^{(\alpha )}\left( {x_{l,{N_x},s}^{(\alpha )},t_{\tau ,{N_t},k}^{(\alpha )}} \right)} }\\
 &= {\left( {\lambda _{n,m}^{(\alpha )}} \right)^{ - 1}}\,\sum\limits_{s = 0}^{{N_x}} {\sum\limits_{k = 0}^{{N_t}} {\varpi _{{N_x},{N_t},s,k}^{(\alpha )}\,{\phi _{s,k}}\,C_{n,m}^{(\alpha )}\left( {x_{{N_x},s}^{(\alpha )},t_{{N_t},k}^{(\alpha )}} \right)} },
\end{align*}
where
\begin{align}
\lambda _{n,m}^{(\alpha )} &= \lambda _n^{(\alpha )}{\mkern 1mu} \lambda _m^{(\alpha )}\, \forall n,m,\label{sec:conerr1:lambk1}\\
\varpi _{{N_x},{N_t},s,k}^{(\alpha )} &= \varpi _{{N_x},s}^{(\alpha )}\,\varpi _{{N_t},k}^{(\alpha )},\quad s = 0, \ldots ,{N_x};k = 0, \ldots ,{N_t};\\
C_{n,m}^{(\alpha )}(x,t) &= C_n^{(\alpha )}(x)\,C_m^{(\alpha )}(t)\,\forall n,m\label{sec:conerr1:cnm}.
\end{align}
Then,
\begin{equation}
\left|{{\tilde \phi }_{n,m}}\right| \le \frac{{\pi \,{\Gamma ^2}\left( {\alpha  + \frac{1}{2}} \right){{\left( {\lambda _{n,m}^{(\alpha )}} \right)}^{ - 1}}}}{{{\Gamma ^2}\left( {\alpha  + 1} \right)}}\,{\left\| {C_{n,m}^{(\alpha )}} \right\|_{{L^\infty }\left(D^2\right)}}\,{\left\| {{u_{xx}}} \right\|_{{L^\infty }\left( {D_{l,\tau }^2} \right)}},\quad n = 0, \ldots ,{N_x};m = 0, \ldots ,{N_t}, D^2 = [-1, 1] \times [-1, 1] \label{sec:conerr1:ineq:coeb1}.
\end{equation}
Hence Inequalities \eqref{sec:conerr1:ineq4} follow directly from Inequality \eqref{sec:conerr1:ineq:coeb1}, Equations \eqref{sec:pre:eq:normak1}, \eqref{sec:conerr1:lambk1}, \eqref{sec:conerr1:cnm}, and Lemma \ref{subsec:errbounds:lem:max1}. Now, since
\begin{equation}
	\Gamma (x + 1) < \sqrt {2\,\pi \,x} \,{\left( {\frac{x}{e}} \right)^x}\,{\left( {x\,\sinh \left( {\frac{1}{x}} \right)} \right)^{\frac{x}{2}}}\,\left( {1 + \frac{1}{{1620\,{x^5}}}} \right)\;\forall x > 0;
\end{equation}
cf. \cite{Horst2009}, then we can easily show that
\begin{align}
{\left( {\lambda _n^{(\alpha )}} \right)^{ - 1}}{\mkern 1mu} {\kern 1pt} \sum\limits_{s = 0}^{{N_x}} {\left| {\varpi _{{N_x},s}^{(\alpha )}{\mkern 1mu} {\kern 1pt} C_n^{(\alpha )}\left( {x_{l,{N_x},s}^{(\alpha )}} \right)} \right|}  &< \frac{{\pi \,{2^{\frac{3}{2} - 2\alpha }}\,{e^{ - n - 2\alpha  - 1}}\,(n + \alpha )}}{{n!\,\Gamma \left( {\alpha  + \frac{1}{2}} \right)\,\Gamma (\alpha  + 1)\,(n + 2\alpha )}}\,{(n + 2\alpha  + 1)^{n + 2\alpha  + \frac{1}{2}}}\,\left( {\frac{1}{{1620\,{{(n + 2\alpha  + 1)}^5}}} + 1} \right)\nonumber\\
& {\left( {(n + 2\alpha  + 1)\,\sinh \left( {\frac{1}{{n + 2\alpha  + 1}}} \right)} \right)^{\frac{{n + 1}}{2} + \alpha }}\, \forall \alpha \ge 0.
\end{align}
Hence,
\begin{align}
\left| {{{\tilde \phi }_{n,m}}} \right| &< \frac{{{\pi ^2}\,{2^{3 - 4\alpha }}\,{e^{ - n - m - 4\alpha  - 2}}\,(n + \alpha )\,(m + \alpha )}}{{n!\,m!\,{\Gamma ^2}\left( {\alpha  + \frac{1}{2}} \right)\,{\Gamma ^2}(\alpha  + 1)\,(n + 2\alpha )\,(m + 2\alpha )}}\,{(n + 2\alpha  + 1)^{n + 2\alpha  + \frac{1}{2}}}\,{(m + 2\alpha  + 1)^{m + 2\alpha  + \frac{1}{2}}}\,\nonumber\\
	& \left( {\frac{1}{{1620\,{{(n + 2\alpha  + 1)}^5}}} + 1} \right)\,\left( {\frac{1}{{1620\,{{(m + 2\alpha  + 1)}^5}}} + 1} \right)\,{\left( {(n + 2\alpha  + 1)\,\sinh \left( {\frac{1}{{n + 2\alpha  + 1}}} \right)} \right)^{\frac{{n + 1}}{2} + \alpha }}\,\nonumber\\
	& {\left( {(m + 2\alpha  + 1)\,\sinh \left( {\frac{1}{{m + 2\alpha  + 1}}} \right)} \right)^{\frac{{m + 1}}{2} + \alpha }}\,{\left\| {{u_{xx}}} \right\|_{{L^\infty }\left( {D_{l,\tau }^2} \right)}}\\
	& \sim D_1^{(\alpha )}\,{(nm)^{2\alpha }}\,{\left\| {{u_{xx}}} \right\|_{{L^\infty }\left( {D_{l,\tau }^2} \right)}}\, \forall \alpha \ge 0,\quad \text{as }n,m \to \infty,\nonumber
\end{align}
which proves the asymptotic inequality \eqref{sec:conerr1:asy1}. Similarly, we can prove the asymptotic inequality \eqref{sec:conerr1:asy2} using Equation \eqref{subsec:errbounds:eq:linfnorm22}.
\end{proof}
\vspace{-2.5cm}
The following two corollaries underline the cases when the solution $u$ of Problem $\mathcal{P}$ is a polynomial or an analytic function.
\begin{cor}
Let $r \in \mathbb{Z}^+,\, u(x,t) \in \mathbb{P}_r$, in $x$ be the solution of Problem $\mathcal{P}$. Suppose also that $u$ is interpolated by the shifted Gegenbauer polynomials at the SGG nodes, $\left(x_{l,N_x,i}^{(\alpha)}, t_{\tau,N_t,j}^{(\alpha)}\right), i = 0, \ldots, N_x; j = 0, \ldots, N_t$, on the rectangular domain $D_{l,\tau }^2$. Then there exists a positive constant $C$, such that the discrete shifted Gegenbauer coefficients $\tilde \phi_{n,m}$, given by Equations \eqref{sec:theshi2:eq:disccoeffkimo1} are bounded by the following inequalities:
\begin{empheq}[left={\left| {{{\tilde \phi }_{n,m}}} \right| \le}\empheqlbrace]{align}\label{sec:conerr1:ineq4cor1}
	&\frac{{4\, C_l\, r^4\,(n + \alpha )\,(m + \alpha )\,\Gamma (n + 2\alpha )\,\Gamma (m + 2\alpha )}}{{{\Gamma ^2}(2\alpha  + 1)\,\Gamma (n + 1)\,\Gamma (m + 1)}}\,{\left\| {{u}} \right\|_{{L^\infty }\left( {D_{l,\tau }^2} \right)}},\quad \alpha  \ge 0,\\
	&\frac{C_l\, r^4\,{\left| {(n + \alpha )\,(m + \alpha )} \right|}}{{{\alpha ^2}}}\,\left| {\left( \begin{array}{l}
	\frac{n}{2} + \alpha  - 1\\
	\hfill \frac{n}{2} \hfill
	\end{array} \right)\,\left( \begin{array}{l}
	\frac{m}{2} + \alpha  - 1\\
	\hfill \frac{m}{2} \hfill
	\end{array} \right)} \right|\;{\left\| {{u}} \right\|_{{L^\infty }\left( {D_{l,\tau }^2} \right)}},\quad  \frac{n}{2},\frac{m}{2} \in \mathbb{Z}_0^ + \wedge - \frac{1}{2} < \alpha  < 0 ,\nonumber\\
	&\frac{{4\,C_l\,r^4\,\,\left| {(n + \alpha )\,(m + \alpha )} \right|\,\Gamma \left( {\frac{{n + 1}}{2} + \alpha } \right)\,\Gamma \left( {\frac{{m + 1}}{2} + \alpha } \right)}}{{{\Gamma ^2}(\alpha  + 1)\,\Gamma \left( {\frac{{n + 1}}{2}} \right)\,\Gamma \left( {\frac{{m + 1}}{2}} \right)\,\sqrt[4]{{{n^2}\,{m^2}\,{{(n + 2\alpha)}^2}\,{{(m + 2\alpha)}^2}}}}}\,{\left\| {{u}} \right\|_{{L^\infty }\left( {D_{l,\tau }^2} \right)}},\quad  \frac{n}{2},\frac{m}{2} \notin \mathbb{Z}_0^ +  \wedge - \frac{1}{2} < \alpha  < 0,\nonumber\\
	&\frac{{2\,C_l\,r^4\,\Gamma \left( {\frac{{m + 1}}{2} + \alpha } \right)\,\left| {(n + \alpha )\,(m + \alpha )} \right|}}{{{\alpha ^2}\,\sqrt {m\,(m + 2\alpha )} \,\Gamma \left( {\frac{{m + 1}}{2}} \right)\,\Gamma (\alpha )}}\,\left| {\left( \begin{array}{l}
	\frac{n}{2} + \alpha  - 1\\
	\hfill \frac{n}{2} \hfill
	\end{array} \right)} \right|\;{\left\| {{u}} \right\|_{{L^\infty }\left( {D_{l,\tau }^2} \right)}},\quad \frac{n}{2} \in \mathbb{Z}_0^ +  \wedge \frac{m}{2} \notin \mathbb{Z}_0^ +  \wedge - \frac{1}{2} < \alpha  < 0,\nonumber\\
	&\frac{{2\,C_l\,r^4\,\Gamma \left( {\frac{{n + 1}}{2} + \alpha } \right)\,\left| {(n + \alpha )\,(m + \alpha )} \right|}}{{{\alpha ^2}\,\sqrt {n\,(n + 2\alpha )} \,\Gamma \left( {\frac{{n + 1}}{2}} \right)\,\Gamma (\alpha )}}\,\left| {\left( \begin{array}{l}
	\frac{m}{2} + \alpha  - 1\\
	\hfill \frac{m}{2} \hfill
	\end{array} \right)} \right|\;{\left\| {{u}} \right\|_{{L^\infty }\left( {D_{l,\tau }^2} \right)}},\quad \frac{n}{2} \notin \mathbb{Z}_0^ +  \wedge \frac{m}{2} \in \mathbb{Z}_0^ +  \wedge - \frac{1}{2} < \alpha  < 0,\nonumber
\end{empheq}

\vspace{-4.5cm}
where $C_l = 4 C/l^2$, is a positive constant independent of $n$ and $m$. Moreover, as $n, m \to \infty$, the discrete shifted Gegenbauer coefficients $\tilde \phi_{n,m}$, are asymptotically bounded by:
\begin{subequations}
\begin{empheq}[left={\left| {{{\tilde \phi }_{n,m}}} \right| \simlt}\empheqlbrace]{align}
&D_{1,l}^{^{(\alpha )}}\,r^4\,{(nm)^{2\alpha }}\,{\left\| {{u}} \right\|_{{L^\infty }\left( {D_{l,\tau }^2} \right)}},\quad \alpha  \ge 0,\label{sec:conerr1:asy1cor1}\\
&D_{2,l}^{^{(\alpha )}}\,r^4\,{(nm)^\alpha }\,{\left\| {{u}} \right\|_{{L^\infty }\left( {D_{l,\tau }^2} \right)}},\quad  - \frac{1}{2} < \alpha  < 0,\label{sec:conerr1:asy2cor1}
\end{empheq}
\end{subequations}
where
\begin{subequations}
\begin{align}
	D_{1,l}^{^{(\alpha )}} &= C_l\, D_1^{^{(\alpha )}},\\
	D_{2,l}^{^{(\alpha )}} &= C_l\,D_2^{^{(\alpha )}},
\end{align}
\end{subequations}
and $D_1^{^{(\alpha )}}; D_2^{^{(\alpha )}}$, are as given by Equations \eqref{sec:conerr1:asy1coeffD12}.
\end{cor}
\begin{proof}
The corollary follows from the inverse inequality of differentiation for any polynomial $U(x) \in \mathbb{P}_r$; cf. \cite[Inequality (9.5.4)]{Canuto1988}.
\end{proof}
\begin{cor}
Let $u(x,t)$, be the solution of Problem $\mathcal{P}$. Suppose also that $U(x) \equiv u\left(x, \mathop {{\text{argmax}}}\nolimits_{0 \le t \le \tau } \left| {{u_{xx}}(x,t)} \right|\right)$, is an analytic function on $[0, l]$, and that there exist constants $\rho \ge 1$, and $C(l, \rho)$, such that for every $k \ge 0$,
\begin{equation}
	{\left\| {{U^{(k)}}} \right\|_{{L^\infty }[0,l]}} \le C(l,\rho )\,\frac{{k!}}{{{\rho ^k}}}.
\end{equation}
If $u$ is interpolated by the shifted Gegenbauer polynomials at the SGG nodes, $\left(x_{l,N_x,i}^{(\alpha)}, t_{\tau,N_t,j}^{(\alpha)}\right)$, $i = 0, \ldots, N_x; j = 0, \ldots, N_t$, on the rectangular domain $D_{l,\tau }^2$, then the discrete shifted Gegenbauer coefficients $\tilde \phi_{n,m}$, given by Equations \eqref{sec:theshi2:eq:disccoeffkimo1} are bounded by the following inequalities:
\begin{empheq}[left={\left| {{{\tilde \phi }_{n,m}}} \right| \le}\empheqlbrace]{align}\label{sec:conerr1:ineq4cor2}
	&\frac{{8\,C(l,\rho)\,(n + \alpha )\,(m + \alpha )\,\Gamma (n + 2\alpha )\,\Gamma (m + 2\alpha )}}{{{\rho^2\,\Gamma ^2}(2\alpha  + 1)\,\Gamma (n + 1)\,\Gamma (m + 1)}},\quad \alpha  \ge 0,\\
	&\frac{2\,C(l,\rho)\,{\left| {(n + \alpha )\,(m + \alpha )} \right|}}{{(\rho\,{\alpha})^2}}\,\left| {\left( \begin{array}{l}
	\frac{n}{2} + \alpha  - 1\\
	\hfill \frac{n}{2} \hfill
	\end{array} \right)\,\left( \begin{array}{l}
	\frac{m}{2} + \alpha  - 1\\
	\hfill \frac{m}{2} \hfill
	\end{array} \right)} \right|,\quad  \frac{n}{2},\frac{m}{2} \in \mathbb{Z}_0^ + \wedge - \frac{1}{2} < \alpha  < 0 ,\nonumber\\
	&\frac{{8\,C(l,\rho)\,\left| {(n + \alpha )\,(m + \alpha )} \right|\,\Gamma \left( {\frac{{n + 1}}{2} + \alpha } \right)\,\Gamma \left( {\frac{{m + 1}}{2} + \alpha } \right)}}{{{\rho^2\,\Gamma ^2}(\alpha  + 1)\,\Gamma \left( {\frac{{n + 1}}{2}} \right)\,\Gamma \left( {\frac{{m + 1}}{2}} \right)\,\sqrt[4]{{{n^2}\,{m^2}\,{{(n + 2\alpha)}^2}\,{{(m + 2\alpha)}^2}}}}},\quad  \frac{n}{2},\frac{m}{2} \notin \mathbb{Z}_0^ +  \wedge - \frac{1}{2} < \alpha  < 0,\nonumber\\
	&\frac{{4\,C(l,\rho)\,\Gamma \left( {\frac{{m + 1}}{2} + \alpha } \right)\,\left| {(n + \alpha )\,(m + \alpha )} \right|}}{{(\rho\,{\alpha})^2\,\sqrt {m\,(m + 2\alpha )} \,\Gamma \left( {\frac{{m + 1}}{2}} \right)\,\Gamma (\alpha )}}\,\left| {\left( \begin{array}{l}
	\frac{n}{2} + \alpha  - 1\\
	\hfill \frac{n}{2} \hfill
	\end{array} \right)} \right|,\quad \frac{n}{2} \in \mathbb{Z}_0^ +  \wedge \frac{m}{2} \notin \mathbb{Z}_0^ +  \wedge - \frac{1}{2} < \alpha  < 0,\nonumber\\
	&\frac{{4\,C(l,\rho)\,\Gamma \left( {\frac{{n + 1}}{2} + \alpha } \right)\,\left| {(n + \alpha )\,(m + \alpha )} \right|}}{{(\rho\,{\alpha})^2\,\sqrt {n\,(n + 2\alpha )} \,\Gamma \left( {\frac{{n + 1}}{2}} \right)\,\Gamma (\alpha )}}\,\left| {\left( \begin{array}{l}
	\frac{m}{2} + \alpha  - 1\\
	\hfill \frac{m}{2} \hfill
	\end{array} \right)} \right|,\quad \frac{n}{2} \notin \mathbb{Z}_0^ +  \wedge \frac{m}{2} \in \mathbb{Z}_0^ +  \wedge - \frac{1}{2} < \alpha  < 0.\nonumber
\end{empheq}

\vspace{-4.5cm}
Moreover, as $n, m \to \infty$, the discrete shifted Gegenbauer coefficients $\tilde \phi_{n,m}$, are asymptotically bounded by:
\begin{subequations}
\begin{empheq}[left={\left| {{{\tilde \phi }_{n,m}}} \right| \simlt}\empheqlbrace]{align}
&D_{1,l,\rho}^{^{(\alpha )}}\,{(nm)^{2\alpha }},\quad \alpha  \ge 0,\label{sec:conerr1:asy1cor2}\\
&D_{2,l,\rho}^{^{(\alpha )}}\,{(nm)^\alpha },\quad  - \frac{1}{2} < \alpha  < 0,\label{sec:conerr1:asy2cor2}
\end{empheq}
\end{subequations}
where
\begin{subequations}\label{sec:conerr1:asy1coeffD12cor2}
\begin{align}
	D_{1,l,\rho}^{^{(\alpha )}} &= \frac{2}{\rho^2}\,C(l,\rho)\,D_1^{^{(\alpha )}},\\
	D_{2,l,\rho}^{^{(\alpha )}} &= \frac{2}{\rho^2}\,C(l,\rho)\,D_2^{^{(\alpha )}},
\end{align}
\end{subequations}
and $D_1^{^{(\alpha )}}; D_2^{^{(\alpha )}}$, are as given by Equations \eqref{sec:conerr1:asy1coeffD12}.
\end{cor}

Theorem \ref{sec:conerr1:thm:coeffconv1} and its corollaries show that the coefficients of the bivariate shifted Gegenbauer expansions are bounded for $-1/2 < \alpha \le 0$, as $n, m \to \infty$, but their magnitudes may asymptotically grow very large for $\alpha > 0$, breaking the stability of the numerical scheme. Moreover, if we denote the asymptotic bound on the coefficients $\tilde \phi_{n,m}$, by $B_{n,m}^+ \forall n, m, \alpha \ge 0$, and by $B_{n,m}^- \forall n, m, \alpha < 0$, then Theorem \ref{sec:conerr1:thm:coeffconv1} also manifests that the coefficients of the bivariate shifted Gegenbauer expansions decay faster for negative $\alpha$-values than for non-negative $\alpha$-values in the sense that
\begin{equation}
	\frac{{B_{n,m}^ + }}{{B_{n,m}^ - }} = \frac{{D_1^{({\alpha _ + })}}}{{A_1^{({\alpha _ - })}\,A_2^{({\alpha _ - })}\,D_1^{({\alpha _ - })}}}\,{(nm)^{2\,{\alpha _ + } - {\alpha _ - }}},\quad \text{as }n, m \to \infty,
\end{equation}
where $\alpha_+$ and $\alpha_-$ are the Gegenbauer parameters of non-negative and negative values. This may suggest at first glance that numerical discretizations at the SGG nodes are preferable for negative $\alpha$-values. However, we shall prove in the sequel that the asymptotic truncation error is minimized in the Chebyshev norm exactly at $\alpha = 0$; that is, when applying the shifted Chebyshev basis polynomials.

\begin{rem}
Collocations at positive and large values of the Gegenbauer parameter $\alpha$ as $n, m \to \infty$ are not recommended as can be inferred from Theorem \ref{sec:conerr1:thm:coeffconv1} and its corollaries. However, the potential large magnitudes of the expansion coefficients are not the only reason for the instability of the numerical scheme in such cases. In fact, recently \citet{Elgindy2013} have also pointed out that the Gegenbauer quadrature `may become sensitive to round-off errors for positive and large values of the parameter $\alpha$ due to the narrowing effect of the Gegenbauer weight function $w^{(\alpha)}(x)$', which drives the quadrature to become more extrapolatory; cf. \cite[p. 90]{Elgindy2013}.
\end{rem}

The following theorem highlights the total truncation error of the SGPM.
\begin{thm}
Consider the integral formulation of the hyperbolic telegraph PDE \eqref{sec:theshi2:eq:intpdek1}, and let
\begin{equation}
	\Theta\, \phi = \Omega \,\phi - \Psi,
\end{equation}
where $\Omega; \Psi$, are as defined by Equations \eqref{sec:theshi2:eq:omegpsi1} and \eqref{sec:theshi2:eq:omegpsi2}. Also let
\begin{align}
	{}_{l,\tau} \Theta _{i,j }\, \phi _{i,j } = {}_{l,\tau}\Omega _{i,j }\,{\phi _{i,j}} - {}_{l,\tau}\Psi _{i,j},\quad i = 0, \ldots ,{N_x};j = 0, \ldots ,{N_t},\\
	{}_{l,\tau} \tilde \Theta _{i,j }^{{N_x},{N_t}}\, \phi _{i,j } = {}_{l,\tau}\Omega _{i,j }^{{N_x},{N_t}}\,{\phi _{i,j}} - {}_{l,\tau}\Psi _{i,j}^{{M_t}},\quad i = 0, \ldots ,{N_x};j = 0, \ldots ,{N_t},
\end{align}
be the integral hyperbolic telegraph PDE \eqref{sec:theshi2:eq:intpdek1} at the interior SGG nodes $\left(x_{l,N_x,i}^{(\alpha)}, t_{\tau,N_t,j}^{(\alpha)}\right), i = 0, \ldots, N_x$; $j = 0, \ldots, N_t$, and its discretization using the SGPM presented in Section \ref{sec:theshi2}, respectively. Then the total truncation error
\begin{equation}
{E_T}({\phi _{i,j}}) = {\,_{l,\tau }}{\Theta _{i,j}}{\mkern 1mu} {\phi _{i,j}}{ - _{l,\tau }}\tilde \Theta _{i,j}^{{N_x},{N_t}}{\mkern 1mu} {\phi _{i,j}},\;i = 0, \ldots ,{N_x};j = 0, \ldots ,{N_t},
\end{equation}
is bounded by
\begin{align*}
	\left| {{E_T}({\phi _{i,j}})} \right| &\le \frac{{{2^{ - N_x - N_t - M_t - 4}}}}{{\left| {K_{{N_x} + 1}^{(\alpha )}\,K_{{N_t} + 1}^{(\alpha )}\,K_{{M_t} + 1}^{(\alpha )}} \right|\,({N_x} + 1)!\,({N_t} + 1)!\,({M_t} + 1)!}}\,\left({2^{{M_t}}}\left| {K_{{M_t} + 1}^{(\alpha )}} \right|\,({M_t} + 1)!\,\left({2^{{N_t} + 2}}\,{l^{{N_x} + 1}}\,x_{l,{N_x},i}^{(\alpha )}\,\left( {l + x_{l,{N_x},i}^{(\alpha )}} \right)\right.\right.\\
 &\left.\left.\left| {K_{{N_t} + 1}^{(\alpha )}} \right|\,({N_t} + 1)!\,M_\phi ^{(x)}\,{\left\| {C_{l,{N_x} + 1}^{(\alpha )}} \right\|_{{L^\infty }[0,l]}} + {\tau ^{{N_t} + 1}}\,t_{\tau ,{N_t},j}^{(\alpha )}\,{\left\| {C_{l,{N_t} + 1}^{(\alpha )}} \right\|_{{L^\infty }[0,\tau ]}}\,\left({2^{{N_x} + 1}}\,\left| {K_{{N_x} + 1}^{(\alpha )}} \right|\,({N_x} + 1)!\,M_\phi ^{(t)}\right.\right.\right.\\
&\left.\left.\left. \left( {x_{l,{N_x},i}^{(\alpha )}\,\left( {l + x_{l,{N_x},i}^{(\alpha )}} \right)\left( {2\,\left| {{\beta _1}} \right| + t_{\tau ,{N_t},j}^{(\alpha )}\,\left| {{\beta _2}} \right|} \right) + 2\,t_{\tau ,{N_t},j}^{(\alpha )}} \right) + \,{l^{{N_x} + 1}}\,x_{l,{N_x},i}^{(\alpha )}\,\left( {l + x_{l,{N_x},i}^{(\alpha )}} \right)\,\left( {2\,\left| {{\beta _1}} \right| + t_{\tau ,{N_t},j}^{(\alpha )}\,\left| {{\beta _2}} \right|} \right)\right.\right.\right.\\
&\left.\left.\left. M_\phi ^{(x,t)}\,{\left\| {C_{l,{N_x} + 1}^{(\alpha )}} \right\|_{{L^\infty }[0,l]}}\right)\right) + {2^{{N_x} + {N_t} + 2}}\,{\tau ^{{M_t} + 1}}\,t_{\tau ,{N_t},j}^{(\alpha )}\left| {K_{{N_x} + 1}^{(\alpha )}\,K_{{N_t} + 1}^{(\alpha )}} \right|\,({N_x} + 1)!\,({N_t} + 1)!\,{\left\| {C_{l,{M_t} + 1}^{(\alpha )}} \right\|_{{L^\infty }[0,\tau ]}}\right.\\
&\left.\left( {t_{\tau ,{N_t},j}^{(\alpha )}\,\left( {M_\psi ^{(t)}\,\left| {{\beta _2}} \right| + M_f^{(t)}} \right) + 2\,M_\psi ^{(t)}\,\left| {{\beta _1}} \right|} \right)\right)\numberthis \label{sec:conerr1:eq:trunc1}
\end{align*}
where
\begin{subequations}
\begin{align}
	M_\phi ^{(x)} &= \mathop {{\text{max}}}\limits_{(x,t) \in D_{l,\tau }^2} \left| {D_x^{({N_x} + 1)}\phi (x,t)} \right|,\\
	M_\phi ^{(t)} &= \mathop {{\text{max}}}\limits_{(x,t) \in D_{l,\tau }^2} \left| {D_t^{({N_t} + 1)}\phi (x,t)} \right|,\\
	M_\psi ^{(t)} &= \mathop {{\text{max}}}\limits_{(x,t) \in D_{l,\tau }^2} \left| {D_t^{({M_t} + 1)}\psi (x,t)} \right|,\\
	M_f^{(t)} &= \mathop {{\text{max}}}\limits_{(x,t) \in D_{l,\tau }^2} \left| {D_t^{({M_t} + 1)}f(x,t)} \right|,\\
	M_\phi ^{(x,t)} &= \mathop {{\text{max}}}\limits_{(x,t) \in D_{l,\tau }^2} \left| {D_{x,t}^{({N_x} + {N_t} + 2)}\phi (x,t)} \right|,
\end{align}
\end{subequations}
\[D_x^{({N_x} + 1)} \equiv \frac{{{\partial ^{{N_x} + 1}}}}{{\partial {x^{^{{N_x} + 1}}}}},\,D_t^{({N_t} + 1)} \equiv \frac{{{\partial ^{{N_t} + 1}}}}{{\partial {t^{^{{N_t} + 1}}}}};\,D_{x,t}^{({N_x} + {N_t} + 2)} \equiv \frac{{{\partial ^{{N_x} + {N_t} + 2}}}}{{\partial {x^{^{{N_x} + 1}}}\,\partial {t^{^{{N_t} + 1}}}}}.\]
Moreover, as $N_x, N_t, M_t \to \infty$, the total truncation error is asymptotically bounded by
\begin{subequations}
\begin{empheq}[left={\left| {{E_T}({\phi _{i,j}})} \right| \simlteq}\empheqlbrace]{align}
&{c_1}{\left( {{N_x}\,{N_t}\,{M_t}} \right)^{\alpha  + \frac{1}{2}}}\,\left({4^{ - {N_x} - {N_t}}}\,N_x^{ - {N_x} - \alpha  - \frac{5}{2}}\,N_t^{ - {N_t} - \alpha  - \frac{5}{2}}\,M_t^{ - \alpha  - \frac{1}{2}}\,\left({c_2}\,{e^{{N_t}}}\,{\tau ^{{N_t}}}\,N_t^{\alpha  + \frac{1}{2}}\right.\right.\nonumber\\
&\left.\left.\,\left( {{c_3}\,{e^{{N_x}}}\,{l^{{N_x}}}\,N_x^{\alpha  + \frac{1}{2}} + {c_4}\,{4^{{N_x}}}\,N_x^{{N_x} + 2}} \right) + {c_5}\,{4^{{N_t}}}\,{e^{{N_x}}}{l^{{N_x}}}\,N_x^{\alpha  + \frac{1}{2}}\,N_t^{{N_t} + 2}\,\right) + {c_6}{\left( {\frac{e}{4}} \right)^{{M_t}}}\,\right.\nonumber\\
&\left.{\tau ^{{M_t}}}\,M_t^{ - {M_t} - 2}\,{\left( {{N_x}\,{N_t}} \right)^{ - \alpha  - \frac{1}{2}}}\right),\quad \alpha  \ge 0,\label{sec:conerr1:eq:empasym1}\\
&{c_1}\,{\left( {{N_x}\,{N_t}\,{M_t}} \right)^{\alpha  + \frac{1}{2}}}\,\left({4^{ - {N_x} - {N_t}}}\,N_x^{ - N_x - \alpha - 2}\,N_t^{ - {N_t} - \alpha  - 2}\,M_t^{ - \alpha  - \frac{1}{2}}\,\left({c_2}\,A_1^{(\alpha )}\,{e^{{N_t}}}{\tau ^{{N_t}}}\right.\right.\nonumber\\
&\left.\left.\left( {{c_3}\,A_2^{(\alpha )}\,{e^{{N_x}}}\,{l^{{N_x}}} + {c_4}\,{4^{{N_x}}}\,N_x^{{N_x} + \frac{3}{2}}} \right) + {c_5}\,A_2^{(\alpha )}\,{4^{{N_t}}}\,{e^{{N_x}}}\,{l^{{N_x}}}\,N_t^{{N_t} + \frac{3}{2}}\right)\right.\nonumber\\
&\left. + {c_6}\,A_3^{(\alpha )}\,{\left( {\frac{e}{4}} \right)^{{M_t}}}\,{\tau ^{{M_t}}}\,M_t^{ - {M_t} - \alpha  - 2}\,{\left( {{N_x}\,{N_t}} \right)^{ - \alpha  - \frac{1}{2}}}\right),\quad - \frac{1}{2} < \alpha  < 0,\label{sec:conerr1:eq:empasym2}
\end{empheq}
\end{subequations}
for some positive constants $c_s, s = 1, \ldots, 6; A_k^{(\alpha)} > 1, k = 1, 2, 3$, independent of $N_x, N_t; M_t$.
\end{thm}
\begin{proof}
Consider the discrete approximations ${{\tilde J}_{l,x_{l,{N_x},i}^{(\alpha )}}}{\mkern 1mu} \phi \left( {x,t_{\tau ,{N_t},j}^{(\alpha )}} \right),\,\tilde I_{2,t_{\tau ,{N_t},j}^{(\alpha )}}^{(t)}{\mkern 1mu} \phi \left( {x_{l,{N_x},i}^{(\alpha )},t} \right),$ $\tilde I_{1,t_{\tau ,{N_t},j}^{(\alpha )}}^{(t)}{\mkern 1mu} {{\tilde J}_{l,x_{l,{N_x},i}^{(\alpha )}}}{\mkern 1mu} \phi (x,t), \tilde I_{2,t_{\tau ,{N_t},j}^{(\alpha )}}^{(t)}{\mkern 1mu} {{\tilde J}_{l,x_{l,{N_x},i}^{(\alpha )}}}{\mkern 1mu} \phi (x,t), \tilde I_{1,t_{\tau ,{N_t},j}^{(\alpha )}}^{(t)}{\mkern 1mu} \psi \left( {x_{l,{N_x},i}^{(\alpha )},t} \right),\, \tilde I_{2,t_{\tau ,{N_t},j}^{(\alpha )}}^{(t)}{\mkern 1mu} \psi \left( {x_{l,{N_x},i}^{(\alpha )},t} \right),\,\tilde I_{2,t_{\tau ,{N_t},j}^{(\alpha )}}^{(t)}f\left( {x_{l,{N_x},i}^{(\alpha )},t} \right)$, defined by Equations \eqref{sec:theshi2:subeqs1} at the SGG nodes, $\left(x_{l,N_x,i}^{(\alpha)}, t_{\tau,N_t,j}^{(\alpha)}\right), i = 0, \ldots, N_x; j = 0, \ldots, N_t$, and denote the truncation errors associated with them by ${E_{i,j}}\left( {\tilde J{\mkern 1mu} \phi } \right), {E_{i,j}}\left( {\tilde I_2^{(t)}{\mkern 1mu} \phi } \right),\,{E_{i,j}}\left( {\tilde I_1^{(t)}{\mkern 1mu} \tilde J{\mkern 1mu} \phi } \right), {E_{i,j}}\left( {\tilde I_2^{(t)}{\mkern 1mu} \tilde J{\mkern 1mu} \phi } \right),\,{E_{i,j}}\left( {\tilde I_1^{(t)}{\mkern 1mu} \psi } \right),\,{E_{i,j}}\left( {\tilde I_2^{(t)}{\mkern 1mu} \psi } \right)$, ${E_{i,j}}\left( {\tilde I_2^{(t)}f} \right)$, respectively. Then
\small{
\begin{equation}
\left| {{E_T}({\phi _{i,j}})} \right| \le \left| {{E_{i,j}}\left( {\tilde J{\mkern 1mu} \phi } \right)} \right| + \left| {{\beta _1}} \right|\,\left| {{E_{i,j}}\left( {\tilde I_1^{(t)}{\mkern 1mu} \tilde J{\mkern 1mu} \phi } \right)} \right| + \left| {{\beta _2}} \right|\,\left| {{E_{i,j}}\left( {\tilde I_2^{(t)}{\mkern 1mu} \tilde J{\mkern 1mu} \phi } \right)} \right| + \left| {{E_{i,j}}\left( {\tilde I_2^{(t)}{\mkern 1mu} \phi } \right)} \right| + \left| {{\beta _1}} \right|\,\left| {{E_{i,j}}\left( {\tilde I_1^{(t)}\,\psi } \right)} \right| + \left| {{\beta _2}} \right|\,\left| {{E_{i,j}}\left( {\tilde I_2^{(t)}{\mkern 1mu} \psi } \right)} \right|.
\end{equation}
}
Inequality \eqref{sec:conerr1:eq:trunc1} follows then from the error formulae \eqref{sec1:eq:errorkimohat} and \eqref{sec1:eq:errorkimo} by observing that
\begin{subequations}
\begin{align}
\left| {{E_{i,j}}\left( {\tilde J\,\phi } \right)} \right| &\le \frac{{{2^{ - {N_x} - 2}}{\mkern 1mu} {l^{{N_x} + 1}}\,x_{l,{N_x},i}^{(\alpha )}{\mkern 1mu} \left( {l + x_{l,{N_x},i}^{(\alpha )}} \right){\mkern 1mu} M_\phi ^{(x)}}}{{({N_x} + 1)!{\mkern 1mu} \left| {K_{{N_x} + 1}^{(\alpha )}} \right|}}{\left\| {C_{l,{N_x} + 1}^{(\alpha )}} \right\|_{{L^\infty }[0,l]}},\\
\left| {{E_{i,j}}\left( {\tilde I_2^{(t)}\,\phi } \right)} \right| &\le \frac{{{2^{ - {N_t} - 2}}{\mkern 1mu} {\mkern 1mu} {\tau ^{{N_t} + 1}}\,{{\left( {t_{\tau ,{N_t},j}^{(\alpha )}} \right)}^2}{\mkern 1mu} M_\phi ^{(t)}}}{{({N_t} + 1)!\,{\mkern 1mu} \left| {K_{{N_t} + 1}^{(\alpha )}} \right|}}{\left\| {C_{\tau ,{N_t} + 1}^{(\alpha )}} \right\|_{{L^\infty }[0,\tau ]}},\\
\left| {{E_{i,j}}\left( {\tilde I_1^{(t)}\,\tilde J\,\phi } \right)} \right| &\le \frac{{{2^{ - {N_t} - 3}}{\mkern 1mu} {\mkern 1mu} {\tau ^{{N_t} + 1}}\,x_{l,{N_x},i}^{(\alpha )}{\mkern 1mu} t_{\tau ,{N_t},j}^{(\alpha )}{\mkern 1mu} \left( {l + x_{l,{N_x},i}^{(\alpha )}} \right){\mkern 1mu} {{\left\| {C_{\tau ,{N_t} + 1}^{(\alpha )}} \right\|}_{{L^\infty }[0,\tau ]}}}}{{({N_t} + 1)!\,{\mkern 1mu} \left| {K_{{N_t} + 1}^{(\alpha )}} \right|}}{\mkern 1mu} \left( {2{\mkern 1mu} M_\phi ^{(t)} + \frac{{{2^{ - {N_x}}}{\mkern 1mu} {l^{{N_x} + 1}}{\mkern 1mu} M_\phi ^{(x,t)}{\mkern 1mu} {{\left\| {C_{l,{N_x} + 1}^{(\alpha )}} \right\|}_{{L^\infty }[0,l]}}}}{{({N_x} + 1)!{\mkern 1mu} \left| {K_{{N_x} + 1}^{(\alpha )}} \right|}}} \right),\\
\left| {{E_{i,j}}\left( {\tilde I_2^{(t)}\,\tilde J\,\phi } \right)} \right| &\le \frac{{{2^{ - {N_t} - 4}}{\mkern 1mu} {\tau ^{{N_t} + 1}}{\mkern 1mu} x_{l,{N_x},i}^{(\alpha )}{\mkern 1mu} {{\left( {t_{\tau ,{N_t},j}^{(\alpha )}} \right)}^2}{\mkern 1mu} \left( {l + x_{l,{N_x},i}^{(\alpha )}} \right){\mkern 1mu} {{\left\| {C_{\tau ,{N_t} + 1}^{(\alpha )}} \right\|}_{{L^\infty }[0,\tau ]}}}}{{({N_t} + 1)!\,{\mkern 1mu} \left| {K_{{N_t} + 1}^{(\alpha )}} \right|}}{\mkern 1mu} \left( 2{\mkern 1mu} M_\phi ^{(t)} + \frac{{{2^{ - {N_x}}}{\mkern 1mu} {l^{{N_x} + 1}}{\mkern 1mu} M_\phi ^{(x,t)}{\mkern 1mu} {{\left\| {C_{l,{N_x} + 1}^{(\alpha )}} \right\|}_{{L^\infty }[0,l]}}}}{{({N_x} + 1)!{\mkern 1mu} \left| {K_{{N_x} + 1}^{(\alpha )}} \right|}} \right),\\
\left| {{E_{i,j}}\left( {\tilde I_1^{(t)}{\mkern 1mu} \psi } \right)} \right| &\le {\left( {\frac{\tau }{2}} \right)^{{M_t} + 1}}\frac{{t_{\tau ,{N_t},j}^{(\alpha )}\,M_\psi ^{(t)}}}{{({M_t} + 1)!\,\left| {K_{{M_t} + 1}^{(\alpha )}} \right|}}{\left\| {C_{\tau ,{M_t} + 1}^{(\alpha )}} \right\|_{{L^\infty }[0,\tau ]}},\\
\left| {{E_{i,j}}\left( {\tilde I_2^{(t)}\,\psi } \right)} \right| &\le \frac{{{2^{ - {M_t} - 2}}{\mkern 1mu} {\tau ^{{M_t} + 1}}{\mkern 1mu} {{\left( {t_{\tau ,{N_t},j}^{(\alpha )}} \right)}^2}{\mkern 1mu} M_\psi ^{(t)}}}{{({M_t} + 1)!\,{\mkern 1mu} \left| {K_{{M_t} + 1}^{(\alpha )}} \right|}}{\left\| {C_{\tau ,{M_t} + 1}^{(\alpha )}} \right\|_{{L^\infty }[0,\tau ]}};\\
\left| {{E_{i,j}}\left( {\tilde I_2^{(t)}\,f} \right)} \right| &\le \frac{{{2^{ - {M_t} - 2}}{\mkern 1mu} {\tau ^{{M_t} + 1}}{\mkern 1mu} {{\left( {t_{\tau ,{N_t},j}^{(\alpha )}} \right)}^2}{\mkern 1mu} M_f^{(t)}}}{{({M_t} + 1)!\,{\mkern 1mu} \left| {K_{{M_t} + 1}^{(\alpha )}} \right|}}{\left\| {C_{\tau ,{M_t} + 1}^{(\alpha )}} \right\|_{{L^\infty }[0,\tau ]}}.
\end{align}
\end{subequations}
The asymptotic inequalities \eqref{sec:conerr1:eq:empasym1} and \eqref{sec:conerr1:eq:empasym2} follow by taking the limits of Inequality \eqref{sec:conerr1:eq:trunc1} when $N_x, N_t, M_t \to \infty$, and using Lemmas \ref{subsec:errbounds:lem:max1} and \ref{sec1:lem:1}.
\end{proof}
The following corollary manifests that the asymptotic truncation error is minimized in the Chebyshev norm when applying the shifted Chebyshev basis polynomials.
\begin{cor}\label{cor:optcheb1}
Let $u(x,t) \in C^2\left(D_{l,\tau }^2\right)$, be the solution of Problem $\mathcal{P}$. Suppose also that $u$ is interpolated by the shifted Gegenbauer polynomials at the SGG nodes, $\left(x_{l,N_x,i}^{(\alpha)}, t_{\tau,N_t,j}^{(\alpha)}\right), i = 0, \ldots, N_x; j = 0, \ldots, N_t$, on the rectangular domain $D_{l,\tau }^2$. Then the shifted Chebyshev basis is optimal in the Chebyshev norm as $N_x, N_t; M_t \to \infty$.
\end{cor}
\begin{proof}
If we denote the asymptotic truncation error bounds for positive and negative $\alpha$-values by ${{E_T^+}({\phi _{i,j}})}$ and ${{E_T^-}({\phi _{i,j}})}$, respectively, then we find by comparing the asymptotic formulae \eqref{sec:conerr1:eq:empasym1} and \eqref{sec:conerr1:eq:empasym2} that
\begin{equation}
	\left|{{E_T^+}({\phi _{i,j}})}\right| > \left|{{E_T^-}({\phi _{i,j}})}\right|\, \forall i, j.
\end{equation}
That is, the asymptotic truncation error bound is bigger for discretizations at the SGG nodes $\left(x_{l,N_x,i}^{(\alpha)}, t_{\tau,N_t,j}^{(\alpha)}\right), i = 0, \ldots, N_x$; $j = 0, \ldots, N_t$, with positive $\alpha$-values, and the gap grows wider for larger positive $\alpha$-values. However, at $\alpha = 0$, we find that the asymptotic truncation error bound, ${{E_T^0}({\phi _{i,j}})}$, is smaller in magnitude than $\left|{{E_T^-}({\phi _{i,j}})}\right|$, since
\begin{align}
\left| {E_T^ - ({\phi _{i,j}})} \right| - \left| {E_T^0({\phi _{i,j}})} \right| &= {c_1}\,{4^{ - {N_x} - {N_t} - {M_t}}}\,N_x^{ - {N_x} - 2}\,N_t^{ - {N_t} - 2}\,M_t^{ - {M_t} - \frac{3}{2}}\,\left({4^{{N_t}}}\,\sqrt {{N_x}} \,N_t^{{N_t} + 2}\left(\left( {A_2^{(\alpha )} - 1} \right)\,{c_5}\,{4^{{M_t}}}\,{e^{{N_x}}}\,{l^{{N_x}}}\,M_t^{{M_t} + \frac{3}{2}}\right.\right.\nonumber\\
&\left.\left. + \left( {A_3^{(\alpha )} - 1} \right)\,{c_6}\,{4^{{N_x}}}\,{e^{{M_t}}}\,{\tau ^{{M_t}}}\,N_x^{{N_x} + \frac{3}{2}}\right) + {c_2}\,{4^{{M_t}}}\,{e^{{N_t}}}\,{\tau ^{{N_t}}}\,M_t^{{M_t} + \frac{3}{2}}\left({c_3}\,{e^{{N_x}}}\,{l^{{N_x}}}\,\sqrt {{N_x}\,{N_t}} \right.\right.\nonumber\\
&\left.\left.\left( {A_1^{(\alpha )}\,A_2^{(\alpha )} - 1} \right) + \left( {A_1^{(\alpha )} - 1} \right)\,{c_4}\,{4^{{N_x}}}\,\sqrt {{N_t}} \,N_x^{{N_x} + 2}\right)\right) > 0\,\forall i,j.
\end{align}
Hence the shifted Chebyshev basis is optimal in the Chebyshev norm as $N_x, N_t; M_t \to \infty$.
\end{proof}
Notice that Corollary \ref{cor:optcheb1} is only valid for large numbers of expansion terms; however, the shifted Chebyshev basis is not necessarily optimal for small/medium range of expansion terms, where other members of the shifted Gegenbauer family of polynomials could exhibit faster convergence rates as recently shown by \citet{Elgindy2012d} and \citet{Elgindy2013a}. The following corollary highlights the convergence order of the SGPM.
\begin{cor}\label{cor:infiniteorder1}
Let $u(x,t) \in C^2\left(D_{l,\tau }^2\right)$, be the solution of Problem $\mathcal{P}$. Suppose also that $u$ is interpolated by the shifted Gegenbauer polynomials at the SGG nodes, $\left(x_{l,N_x,i}^{(\alpha)}, t_{\tau,N_t,j}^{(\alpha)}\right), i = 0, \ldots, N_x; j = 0, \ldots, N_t$, on the rectangular domain $D_{l,\tau }^2$. Then the total truncation error of the SGPM is of
\begin{subequations}
\begin{empheq}[left={}\empheqlbrace]{align}
	&O\left( {{{\left( {\frac{{e\,{l_{\max }}}}{4}} \right)}^{2\,{N_{{\text{max}}}}}}N_{\max }^{2\alpha }N_{\min }^{ - {N_{\min }} - 3/2}} \right),\quad \alpha  \ge 0,\\
	&O\left( {{{\left( {\frac{{e\,{l_{\max }}}}{4}} \right)}^{2\,{N_{{\text{max}}}}}}N_{\min }^{ - {N_{\min }} - 3/2}} \right),\quad \alpha  < 0,
	\end{empheq}
	\end{subequations}
	as $N_x, N_t, M_t \to \infty$, where ${l_{\max }} = \max \{ l,\tau \} ,\;{N_{\max }} = \max \{ {N_x},{N_t},{M_t}\} ;{N_{\min }} = \min \{ {N_x},{N_t},{M_t}\}$.	
\end{cor}
\begin{proof}
The proof follows directly from Inequalities \eqref{sec:conerr1:eq:empasym1} and \eqref{sec:conerr1:eq:empasym2}.
\end{proof}
Corollary \ref{cor:infiniteorder1} shows that the total truncation error of the SGPM decays faster than any finite power of $1/N_{\min}$, for $-1/2 < \alpha \le 0$, exhibiting an ``\textit{infinite order}'' or ``\textit{exponential}'' convergence.
\vspace{-6pt}
\section{Numerical Experiments}
\label{sec:numerical}
\vspace{-2pt}
In this section, we apply the proposed SGPM on four well-studied test examples with known exact solutions in the literature. Comparisons with other competitive numerical schemes are presented to assess the accuracy and efficiency of the SGPM. The numerical experiments were conducted on a personal laptop equipped with an Intel(R) Core(TM) i7-2670QM CPU with 2.20GHz speed running on a Windows 7 64-bit operating system. The resulting algebraic linear system of equations were solved using MATLAB's ``mldivide'' Algorithm provided with MATLAB V. R2013b (8.2.0.701). The optimal S-matrix was constructed using Algorithm 2.2 in \cite{Elgindy2013} with $M_{\text{max}} = 20$. The change of variable \eqref{sec:theshi:eq:change1} was applied using $\varepsilon = 3\, \text{eps}$, where $\text{eps} \approx 2.2204 \times 10^{-16}$, is the machine's floating-point relative accuracy. The solutions $\alpha_i^*$ of the one-dimensional optimization problems \eqref{subsec:opt:reducedprob1} were obtained using MATLAB's line search algorithm ``fminbnd'' with the termination tolerance ``TolX'' set at $\text{eps}/2$. All of the test examples were discretized at the shifted Chebyshev-Gauss nodes $\left(x_{l,N,i}^{(0)}, t_{\tau,N,j}^{(0)}\right) \in D_{l,\tau }^2, i, j = 0, \ldots, N$, defined by
\begin{equation}
\left( {x_{l,N,i}^{(0 )},t_{\tau ,{N_t},j}^{(0 )}} \right) = \left( {\frac{l}{2}\,\left( {1 - \cos \left( {{\mkern 1mu} \frac{{(2\,i - 1)\,\pi }}{{2\,n}}} \right)} \right),\frac{\tau }{2}\,\left( {1 - \cos \left( {{\mkern 1mu} \frac{{(2\,j - 1)\,\pi }}{{2\,n}}} \right)} \right)} \right),\quad i,j = 1, \ldots ,N.
\end{equation}
In all of the numerical experiments, we report the $l^1-, l^2-, l^{\infty}-$norms of the absolute error matrix,
\[\mathcal{E} = \left(\left|u\left(x_{l,N,i}^{(0)}, t_{\tau,N,j}^{(0)}\right) - P_{N,N}u\left(x_{l,N,i}^{(0)}, t_{\tau,N,j}^{(0)}\right)\right|\right),\]
defined by
\begin{align*}
{\left\| {\mathcal{E}} \right\|_1} &= \mathop {{\text{max}}}\limits_j \sum\limits_{i = 1}^{N} {\left| {u\left(x_{l,N,i}^{(0)}, t_{\tau,N,j}^{(0)}\right) - {P_{{N},{N}}}u\left(x_{l,N,i}^{(0)}, t_{\tau,N,j}^{(0)}\right)} \right|} ,\\
{\left\| {\mathcal{E}} \right\|_2} &= \sqrt {{\lambda_{\text{max}}}\left({{\mathcal{E}}^T}\,{\mathcal{E}}\right)} ,\\
{\left\| {\mathcal{E}} \right\|_\infty } &= \mathop {{\text{max}}}\limits_i \sum\limits_{j = 1}^{N} {\left| {u\left(x_{l,N,i}^{(0)}, t_{\tau,N,j}^{(0)}\right) - {P_{{N},{N}}}u\left(x_{l,N,i}^{(0)}, t_{\tau,N,j}^{(0)}\right)} \right|},
\end{align*}
where ${\mathcal{E}}^T$ denotes the transpose of ${\mathcal{E}}$ and ${\lambda_{\text{max}}}\left({{\mathcal{E}}^T}\,{\mathcal{E}}\right)$, is the maximum eigenvalue of ${{\mathcal{E}}^T}\,{\mathcal{E}}$. We shall also give the $L^{\infty}$-norm of the absolute error, $\left|u(x, t) - P_{N,N}u(x, t)\right|$, defined by
\[E_{N,N}^\infty  = {\left\| {u(x,t) - {P_{N,N}}u(x,t)} \right\|_{{L^\infty }(D_{l,\tau }^2)}} = \mathop {{\text{max}}}\limits_{(x,t) \in D_{l,\tau }^2} \left| {u(x,t) - {P_{N,N}}u(x,t)} \right|,\]
and provide the average elapsed CPU time in $100$ runs (AECPUT) taken by the SGPM to calculate the approximate solutions. Moreover, in each test example, we plot the absolute error function,
\begin{equation}
	{E}_{N,N}(x,t) = \left| u(x,t) - {P_{N,N}}u(x,t) \right|, \quad (x,t) \in D_{l,\tau}^2,
\end{equation}
and sketch the exact solution against its approximation over the entire domain $D_{l,\tau}^2$.
\paragraph{\textbf{Example 1}} Consider Problem $\mathcal{P}$ with $l = \tau = \beta_1 = \beta_2 = 1, f(x,t) = x^2 + t - 1, g_1(x) = x^2, g_2(x) = 1, h_1(t) = t; h_2(t) = 1 + t$. The exact solution of the problem is $u(x,t) = x^2 + t$. The plots of the exact solution, its approximation, and the absolute error function on $D_{1,1}^2$ using $N = M_t = 4$, are shown in Figure \ref{sec:numerical:fig:ex1exapprox1}. The $l^1-, l^2-, l^{\infty}-$norms of the absolute error matrix ${\mathcal{E}}$, the $L^{\infty}$-norm of the absolute error, and the AECPUT are shown in Table \ref{sec:numerical:tab:norm1}. The plots and the numerical results demonstrate the power of the SGPM, showing fast computations with errors of very small magnitudes using relatively small number of expansion terms. Moreover, Table \ref{sec:numerical:tab:norm1} manifests the ability of the SGPM to achieve higher order approximations by increasing the number of expansion terms in the optimal S-quadrature (i.e. the value of $M_t$) while preserving the same size of the global collocation matrix $\mathbf{A}$; hence the dimension of the linear system \eqref{sec:theshi2:eq:phisol1}. A plot of the elapsed CPU time to compute the approximate solution, ${P_{{N},{N}}}u$, at the collocation points $\left(x_{l,N,i}^{(0)}, t_{\tau,N,j}^{(0)}\right), i, j = 0, \ldots, N$, versus the total number of unknowns $(L + 1)$ using various values of $N$ and $M_t$ is shown in Figure \ref{sec:numerical:fig:CPU_Time2}, where we observe an approximate time complexity of $O\left((L + 1)^{2}\right)$, as $L \to \infty$, for relatively small values of $M_t$.

\begin{figure}[ht]
\centering
\includegraphics[scale=0.9]{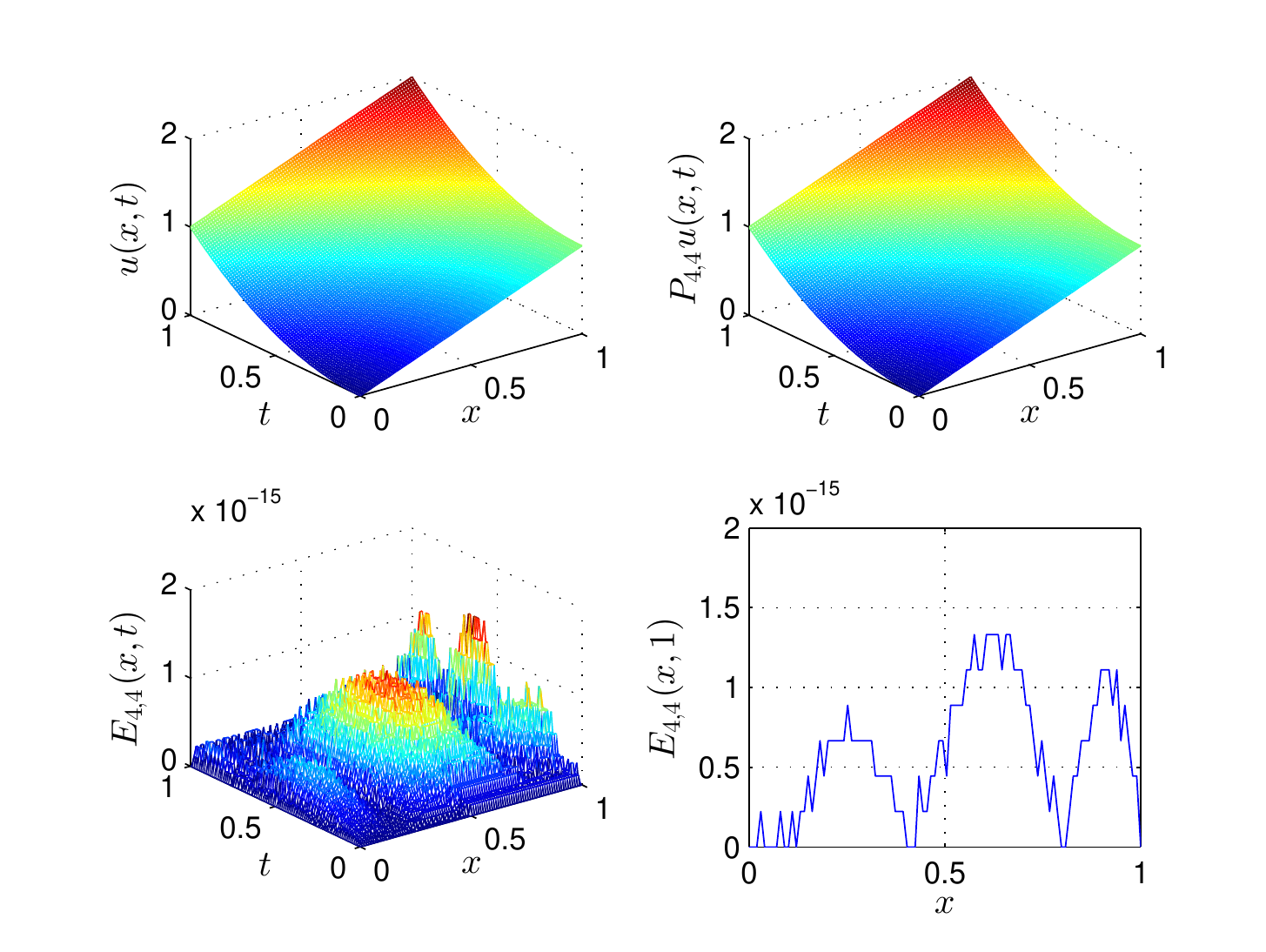}
\caption{The numerical simulation of the SGPM on Example 1. The figure shows the plots of the exact solution $u(x,t)$  on $D_{1,1}^2$ (upper left), its approximation $P_{4,4}u(x,t)$ (upper right), the absolute error function ${E}_{4,4}(x,t)$ (lower left), and its values at the final time, ${E}_{4,4}(x,1)$ (lower right). The optimal S-matrix was constructed using $M_t = 4$, and the plots were generated using $100$ linearly spaced nodes in the $x$- and $t$-directions from $0$ to $1$.}
\label{sec:numerical:fig:ex1exapprox1}
\end{figure}

\begin{table}[ht]
\begin{center} 
\scalebox{0.7}{
\resizebox{\textwidth}{!}{ %
\begin{tabular}{cccccc}
\toprule
\multicolumn{6}{c}{\textbf{Example 1}} \\
\cmidrule(r){1-6}
$N, M_t$ & $4, 4$ & $6, 6$ & $6, 9$ & $8, 8$ & $8, 12$ \\
\cmidrule(r){1-6}
${\left\| {\mathcal{E}} \right\|_1}$ & $1.665 \times 10^{-15}$ & $1.810 \times 10^{-14}$ & $4.552 \times 10^{-15}$ & $3.408 \times 10^{-14}$ & $1.554 \times 10^{-14}$ \\
 ${\left\| {\mathcal{E}} \right\|_2}$ & $1.127 \times 10^{-15}$ & $1.448 \times 10^{-14}$ & $2.947 \times 10^{-15}$ & $2.457 \times 10^{-14}$ & $1.095 \times 10^{-14}$ \\
${\left\| {\mathcal{E}} \right\|_\infty }$ & $1.665 \times 10^{-15}$ & $2.087 \times 10^{-14}$ & $4.663 \times 10^{-15}$ & $4.025 \times 10^{-14}$ & $1.649 \times 10^{-14}$ \\
$E_{N,N}^\infty$ & $1.332 \times 10^{-15}$ & $3.553 \times 10^{-15}$ & $2.887 \times 10^{-15}$ & $4.219 \times 10^{-15}$ & $4.663 \times 10^{-15}$ \\
AECPUT & $0.012$s & $0.028$s & $0.029$s & $0.279$s & $0.290$s \\
\bottomrule
\end{tabular}
}}
\caption{The $l^1-, l^2-, l^{\infty}-$norms of the absolute error matrix ${\mathcal{E}}$, the $L^{\infty}$-norm of the absolute error, and the AECPUT of Example 1.}
\label{sec:numerical:tab:norm1}
\end{center}
\end{table}
\begin{figure}[ht]
\centering
\includegraphics[scale=0.6]{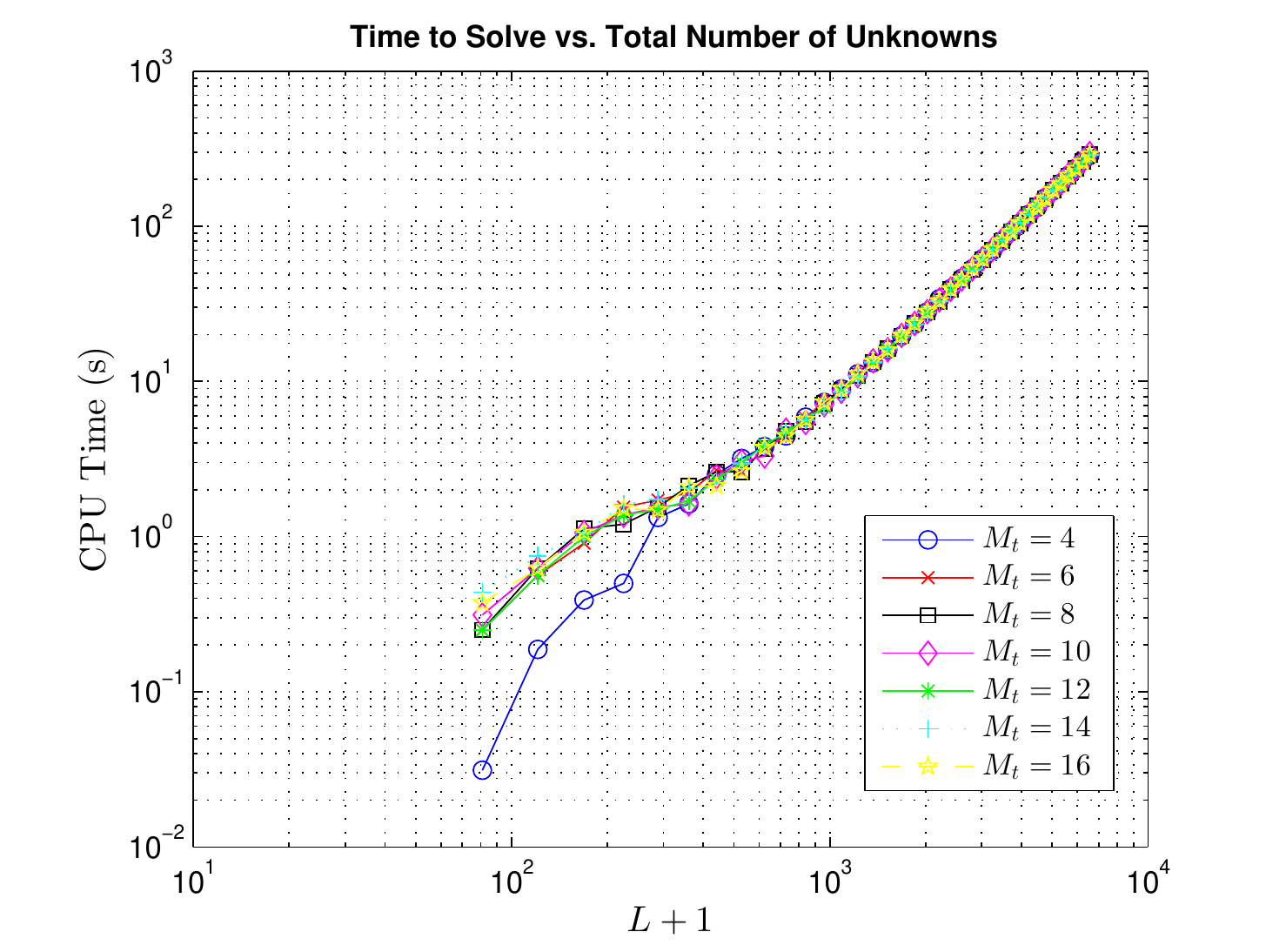}
\caption{Plot of the CPU time versus the total number of unknowns $(L + 1)$ in log-log scale using $N = 8(2)80$ (i.e. $80 \le L \le 6560$), and $M_t = 4(2)16$. The CPU time curves eventually approach a straight line with an average asymptotic slope of approximately $2$, for increasing values of $L$, indicating an approximate time complexity of $O\left((L + 1)^{2}\right)$, as $L \to \infty$, for relatively small values of $M_t$.}
\label{sec:numerical:fig:CPU_Time2}
\end{figure}

\paragraph{\textbf{Example 2}} Consider Problem $\mathcal{P}$ with $l = \tau = 1, \beta_1 = 10, \beta_2 = 24, g_1(x) = x^4\, (x-1)^4, g_2(x) = 2\,x^4\,(x-1)^4, h_1(t) = h_2(t) = 0; f(x,t) = 4\,{e^{2\,t}}\,{x^2}\,{\left( {x - 1} \right)^2}\,\left( {12\,{x^4} - 24\,{x^3} - 2\,{x^2} + 14\,x - 3} \right)$. The exact solution of the problem is $u(x,t) = {x^4}{\mkern 1mu} {(x - 1)^4}{\mkern 1mu} {e^{2{\mkern 1mu} t}}$. The plots of the exact solution, its approximation, and the absolute error function on $D_{1,1}^2$ using $N = M_t = 8$, are shown in Figure \ref{sec:numerical:fig:ex2exapprox2}. The $l^1-, l^2-, l^{\infty}-$norms of the absolute error matrix ${\mathcal{E}}$, the $L^{\infty}$-norm of the absolute error, and the AECPUT are shown in Table \ref{sec:numerical:tab:norm2}. A comparison between \citeauthor{Mohanty2005}'s finite difference method \cite{Mohanty2005}, \citeauthor{Ding2012}'s non-polynomial spline method \cite{Ding2012}, and the SGPM is also shown in Table \ref{sec:numerical:tab:norm22}. The plots and the numerical comparisons show the rapid convergence rates and the memory minimizing feature of the SGPM. For instance, \citeauthor{Ding2012}'s non-polynomial spline method \cite{Ding2012} requires the solution of a linear system of equations of order $119 \times 119$, versus $9 \times 9$, for the SGPM to establish the same order of accuracy.

\begin{figure}[ht]
\centering
\includegraphics[scale=0.9]{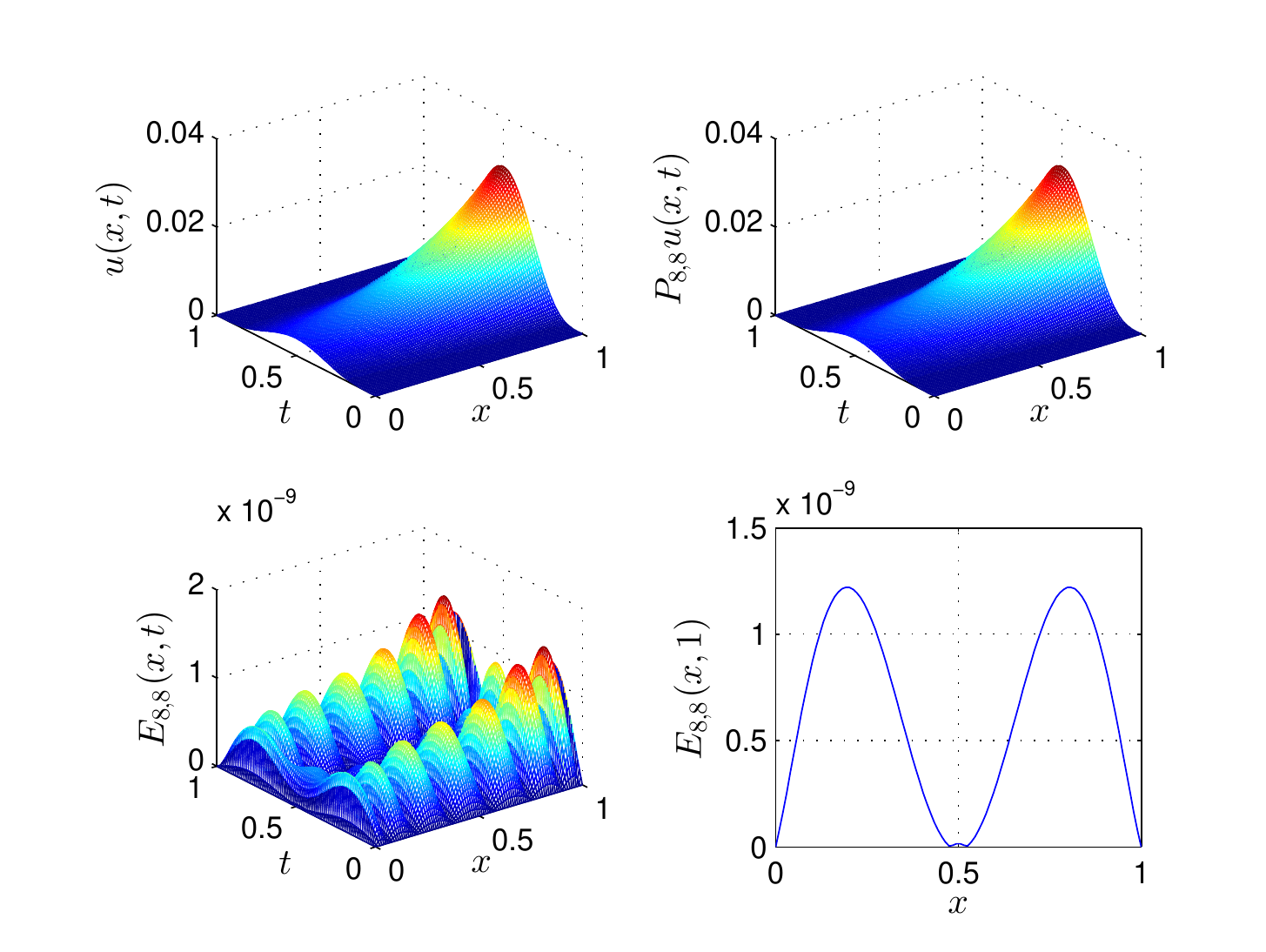}
\caption{The numerical simulation of the SGPM on Example 2. The figure shows the plots of the exact solution $u(x,t)$  on $D_{1,1}^2$ (upper left), its approximation $P_{8,8}u(x,t)$ (upper right), the absolute error function ${E}_{8,8}(x,t)$ (lower left), and its values at the final time, ${E}_{8,8}(x,1)$ (lower right). The optimal S-matrix was constructed using $M_t = 8$, and the plots were generated using $100$ linearly spaced nodes in the $x$- and $t$-directions from $0$ to $1$.}
\label{sec:numerical:fig:ex2exapprox2}
\end{figure}
\begin{table}[ht]
\begin{center} 
\scalebox{0.6}{
\resizebox{\textwidth}{!}{ %
\begin{tabular}{ccccc}
\toprule
\multicolumn{5}{c}{\textbf{Example 2}} \\
\cmidrule(r){1-5}
$N, M_t$ & $8, 8$ & $10,10$ & $12,12$ & $14,14$\\
\cmidrule(r){1-5}
${\left\| {\mathcal{E}} \right\|_1}$ & $ 5.418 \times 10^{-9}$ & $ 1.883 \times 10^{-11}$ & $ 8.119 \times 10^{-14}$ & $ 1.076 \times 10^{-14}$\\
${\left\| {\mathcal{E}} \right\|_2}$ & $ 3.759 \times 10^{-9}$ & $ 1.542 \times 10^{-11}$ & $ 6.271 \times 10^{-14}$ & $ 8.545 \times 10^{-15}$\\
${\left\| {\mathcal{E}} \right\|_\infty}$ & $ 5.089 \times 10^{-9}$ & $ 2.018 \times 10^{-11}$ & $ 8.952 \times 10^{-14}$ & $ 1.158 \times 10^{-14}$\\
$E_{N,N}^\infty$ & $ 1.420 \times 10^{-9}$ & $ 4.222 \times 10^{-12}$ & $ 1.331 \times 10^{-14}$ & $ 1.697 \times 10^{-15}$\\
AECPUT & $0.279$s & $0.553$s & $1.001$s & $1.257$s \\
\bottomrule
\end{tabular}
}}
\caption{The $l^1-, l^2-, l^{\infty}-$norms of the absolute error matrix ${\mathcal{E}}$, the $L^{\infty}$-norm of the absolute error, and the AECPUT of Example 2.}
\label{sec:numerical:tab:norm2}
\end{center}
\end{table}
\begin{table}[ht]
\begin{center} 
\scalebox{1}{
\resizebox{\textwidth}{!}{ %
\begin{tabular}{ccccccc}
\toprule
\multicolumn{7}{c}{\textbf{Example 2}} \\
\cmidrule(r){1-7}
$x$ & \textbf{\citeauthor{Mohanty2005}'s method \cite{Mohanty2005}} & \textbf{\citeauthor{Ding2012}'s method \cite{Ding2012}} & \multicolumn{4}{c}{\textbf{Present method}}\\
 & Difference scheme (11), $\eta = \tfrac{1}{24}, \gamma = \tfrac{1}{3}$ & $\lambda_1 = \tfrac{1}{12}; \lambda_2 = \tfrac{5}{6}$ & $N = M_t = 8$ & $N = M_t = 10$ & $N = M_t = 12$ & $N = M_t = 14$\\
\cmidrule(r){1-7}
$0.2$ & $ 4.956 \times 10^{-3}$ & $1.922 \times 10^{-9}$ & $1.220 \times 10^{-9}$ & $3.939 \times 10^{-12}$ & $1.266 \times 10^{-14}$ & $1.560 \times 10^{-15}
$\\
$0.4$ & $ 2.393\times 10^{-2}$ & $2.815 \times 10^{-10}$ & $2.740 \times 10^{-10}$ & $4.527 \times 10^{-13}$ & $4.819 \times 10^{-15}$ & $9.021 \times 10^{-16}$\\
$0.6$ & $ 2.393 \times 10^{-2}$ & $2.815 \times 10^{-10}$ & $2.740 \times 10^{-10}$ & $4.526 \times 10^{-13}$ & $4.774 \times 10^{-15}$ & $1.013 \times 10^{-15}$\\
$0.8$ & $ 4.956 \times 10^{-3}$ & $1.922 \times 10^{-9}$ & $1.220 \times 10^{-9}$ & $3.939 \times 10^{-12}$ & $1.252 \times 10^{-14}$ & $1.443 \times 10^{-15}$\\
\bottomrule
\end{tabular}
}}
\caption{A comparison of Example 2 between \citeauthor{Mohanty2005}'s finite difference method \cite{Mohanty2005}, \citeauthor{Ding2012}'s non-polynomial spline method \cite{Ding2012}, and the SGPM. The table lists the absolute errors at $x = 0.2, 0.4, 0.6; 0.8,$ and $t = 1$. The results of \citeauthor{Mohanty2005}'s method \cite{Mohanty2005} and \citeauthor{Ding2012}'s method \cite{Ding2012} are exactly as quoted from Ref. \cite{Ding2012}.}
\label{sec:numerical:tab:norm22}
\end{center}
\end{table}
\paragraph{\textbf{Example 3}} Consider Problem $\mathcal{P}$ with $l = \tau = 1, \beta_1 = 12, \beta_2 = 4, g_1(x) = \sin(x), g_2(x) = 0, h_1(t) = 0, h_2(t) = \sin(1)\,\cos(t); f(x,t) = 4\,(\cos (t) - 3\,\sin (t))\,\sin (x)$. The exact solution of the problem is $u(x,t) = \sin(x)\,\cos(t)$. The plots of the exact solution, its approximation, and the absolute error function on $D_{1,1}^2$ using $N = M_t = 4$, are shown in Figure \ref{sec:numerical:fig:ex3exapprox3}. The $l^1-, l^2-, l^{\infty}-$norms of the absolute error matrix ${\mathcal{E}}$, the $L^{\infty}$-norm of the absolute error, and the AECPUT are shown in Table \ref{sec:numerical:tab:norm3}. A comparison between \citeauthor{Dosti2012}'s quartic B-spline collocation method \cite{Dosti2012}, \citeauthor{Mittal2013}'s cubic B-spline collocation method \cite{Mittal2013}, and the SGPM is also shown in Table \ref{sec:numerical:tab:norm33}. Table \ref{sec:numerical:tab:norm3} shows one of the advantageous ingredients of the SGPM: ``\textit{the ability to achieve higher-order approximations while preserving the same dimension of the linear system \eqref{sec:theshi2:eq:phisol1}}''. Table \ref{sec:numerical:tab:norm33} shows the power of the presented scheme, which constructs higher-order approximations using as small as $5$ expansion terms in both spatial and temporal directions.

\begin{figure}[ht]
\centering
\includegraphics[scale=0.9]{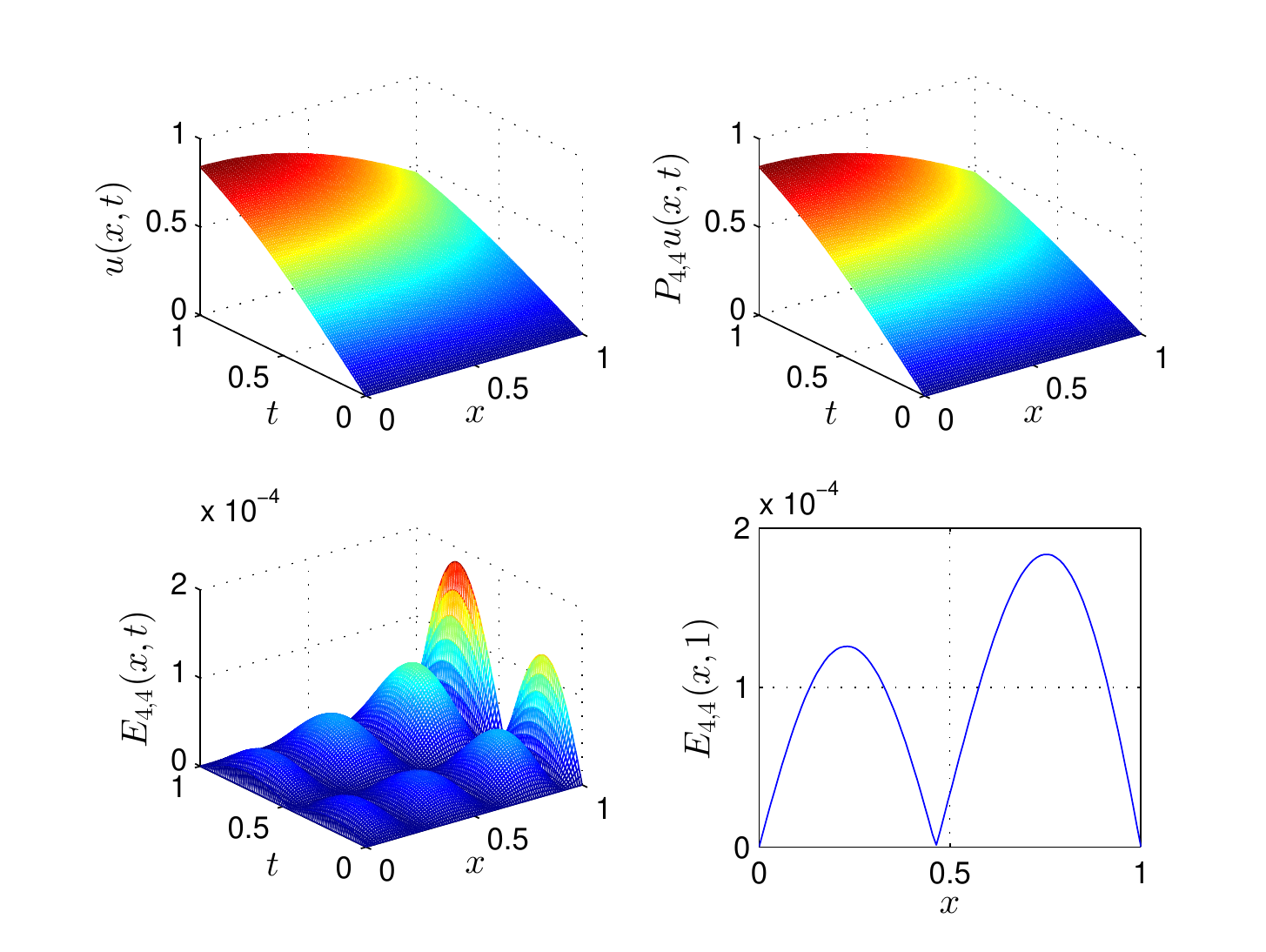}
\caption{The numerical simulation of the SGPM on Example 3. The figure shows the plots of the exact solution $u(x,t)$  on $D_{1,1}^2$ (upper left), its approximation $P_{4,4}u(x,t)$ (upper right), the absolute error function ${E}_{4,4}(x,t)$ (lower left), and its values at the final time, ${E}_{4,4}(x,1)$ (lower right). The optimal S-matrix was constructed using $M_t = 4$, and the plots were generated using $100$ linearly spaced nodes in the $x$- and $t$-directions from $0$ to $1$.}
\label{sec:numerical:fig:ex3exapprox3}
\end{figure}
\begin{table}[ht]
\begin{center} 
\scalebox{0.6}{
\resizebox{\textwidth}{!}{ %
\begin{tabular}{ccccc}
\toprule
\multicolumn{5}{c}{\textbf{Example 3}} \\
\cmidrule(r){1-5}
$N, M_t$ & $4, 4$ & $4,5$ & $4,6$ & $6,6$\\
\cmidrule(r){1-5}
${\left\| {\mathcal{E}} \right\|_1}$ & $ 2.384 \times 10^{-4}$ & $1.036 \times 10^{-4}$ & $9.391 \times 10^{-5}$ & $2.061 \times 10^{-7}$\\
${\left\| {\mathcal{E}} \right\|_2}$ & $ 1.762 \times 10^{-4}$ & $7.438 \times 10^{-5}$ & $6.855 \times 10^{-5}$ & $1.533 \times 10^{-7}$\\
${\left\| {\mathcal{E}} \right\|_\infty}$ & $ 2.460 \times 10^{-4}$ & $1.053 \times 10^{-4}$ & $9.731 \times 10^{-5}$ & $2.696 \times 10^{-7}$\\
$E_{N,N}^\infty$ & $ 1.834 \times 10^{-4}$ & $8.060 \times 10^{-5}$ & $7.348 \times 10^{-5}$ & $1.160 \times 10^{-7}$\\
AECPUT & $0.012$s & $0.012$s & $0.013$s & $0.027$s \\
\bottomrule
\end{tabular}
}}
\caption{The $l^1-, l^2-, l^{\infty}-$norms of the absolute error matrix ${\mathcal{E}}$, the $L^{\infty}$-norm of the absolute error, and the AECPUT of Example 3.}
\label{sec:numerical:tab:norm3}
\end{center}
\end{table}
\begin{table}[ht]
\begin{center} 
\scalebox{1}{
\resizebox{\textwidth}{!}{ %
\begin{tabular}{ccccccc}
\toprule
\multicolumn{7}{c}{\textbf{Example 3}} \\
\cmidrule(r){1-7}
$t$ & \textbf{\citeauthor{Dosti2012}'s method \cite{Dosti2012}} & \textbf{\citeauthor{Mittal2013}'s method \cite{Mittal2013}} & \multicolumn{4}{c}{\textbf{Present method}}\\
 & $h = 0.005, \Delta\,t = 0.001$ & $h = 0.005, \Delta\,t = 0.001$ & $N = M_t = 4$ & $N = 4, M_t = 5$ & $N = 4, M_t = 6$ & $N = M_t = 6$\\
\cmidrule(r){1-7}
$0.2$ & $2.425 \times 10^{-5}$ & $6.827 \times 10^{-5}$ & $1.458 \times 10^{-5}$ & $7.449 \times 10^{-6}$ & $7.114 \times 10^{-6}$ & $4.188 \times 10^{-8}
$\\
$0.4$ & $7.932 \times 10^{-5}$ & $1.494 \times 10^{-4}$ & $4.905 \times 10^{-5}$ & $2.030 \times 10^{-5}$ & $1.922 \times 10^{-5}$ & $3.340 \times 10^{-8}$\\
$0.6$ & $1.210 \times 10^{-4}$ & $2.241 \times 10^{-4}$ & $1.942 \times 10^{-6}$ & $1.127 \times 10^{-6}$ & $1.571 \times 10^{-6}$ & $9.030 \times 10^{-9}$\\
$0.8$ & $1.488 \times 10^{-4}$ & $2.898 \times 10^{-4}$ & $8.293 \times 10^{-5}$ & $3.369 \times 10^{-5}$ & $3.154 \times 10^{-5}$ & $2.172 \times 10^{-8}$\\
$1$ & $1.646 \times 10^{-4}$ & $3.439 \times 10^{-4}$ & $1.834 \times 10^{-4}$ & $8.060 \times 10^{-5}$ & $7.348 \times 10^{-5}$ & $1.160 \times 10^{-7}$\\
\bottomrule
\end{tabular}
}}
\caption{A comparison of Example 3 between \citeauthor{Dosti2012}'s  quartic B-spline collocation method \cite{Dosti2012}, \citeauthor{Mittal2013}'s cubic B-spline collocation method \cite{Mittal2013}, and the SGPM. The table lists the maximum absolute errors at $t = 0.2, 0.4, \ldots, 1$. The results of \citeauthor{Dosti2012}'s method \cite{Dosti2012} and \citeauthor{Mittal2013}'s method \cite{Mittal2013} are exactly as quoted from Ref. \cite{Mittal2013}.}
\label{sec:numerical:tab:norm33}
\end{center}
\end{table}
\paragraph{\textbf{Example 4}} Consider Problem $\mathcal{P}$ with $l = \tau = 1, \beta_1 = 20, \beta_2 = 25, g_1(x) = \sinh(x), g_2(x) = -2 \sinh(x), h_1(t) = 0, h_2(t) = e^{-2 t} \sinh(1); f(x,t) = -12 e^{-2 t} \sinh(x)$. The exact solution of the problem is $u(x,t) = e^{-2 t} \sinh(x)$. The plots of the exact solution, its approximation, and the absolute error function on $D_{1,1}^2$ using $N = M_t = 4$, are shown in Figure \ref{sec:numerical:fig:ex4exapprox4}. The $l^1-, l^2-, l^{\infty}-$norms of the absolute error matrix ${\mathcal{E}}$, the $L^{\infty}$-norm of the absolute error, and the AECPUT are shown in Table \ref{sec:numerical:tab:norm4}. A comparison between
\citeauthor{Mohanty2005}'s finite difference method \cite{Mohanty2005} and the \citeauthor{Pandit2015} combined Crank-Nicolson finite difference and Haar wavelets numerical scheme \cite{Pandit2015}, and the SGPM is also shown in Table \ref{sec:numerical:tab:norm44}, which shows the root mean square error (RMS) for the three numerical schemes.
Both tables show the fast execution times, the exponential convergence, and the cost economization features of the SGPM.

\begin{figure}[ht]
\centering
\includegraphics[scale=0.9]{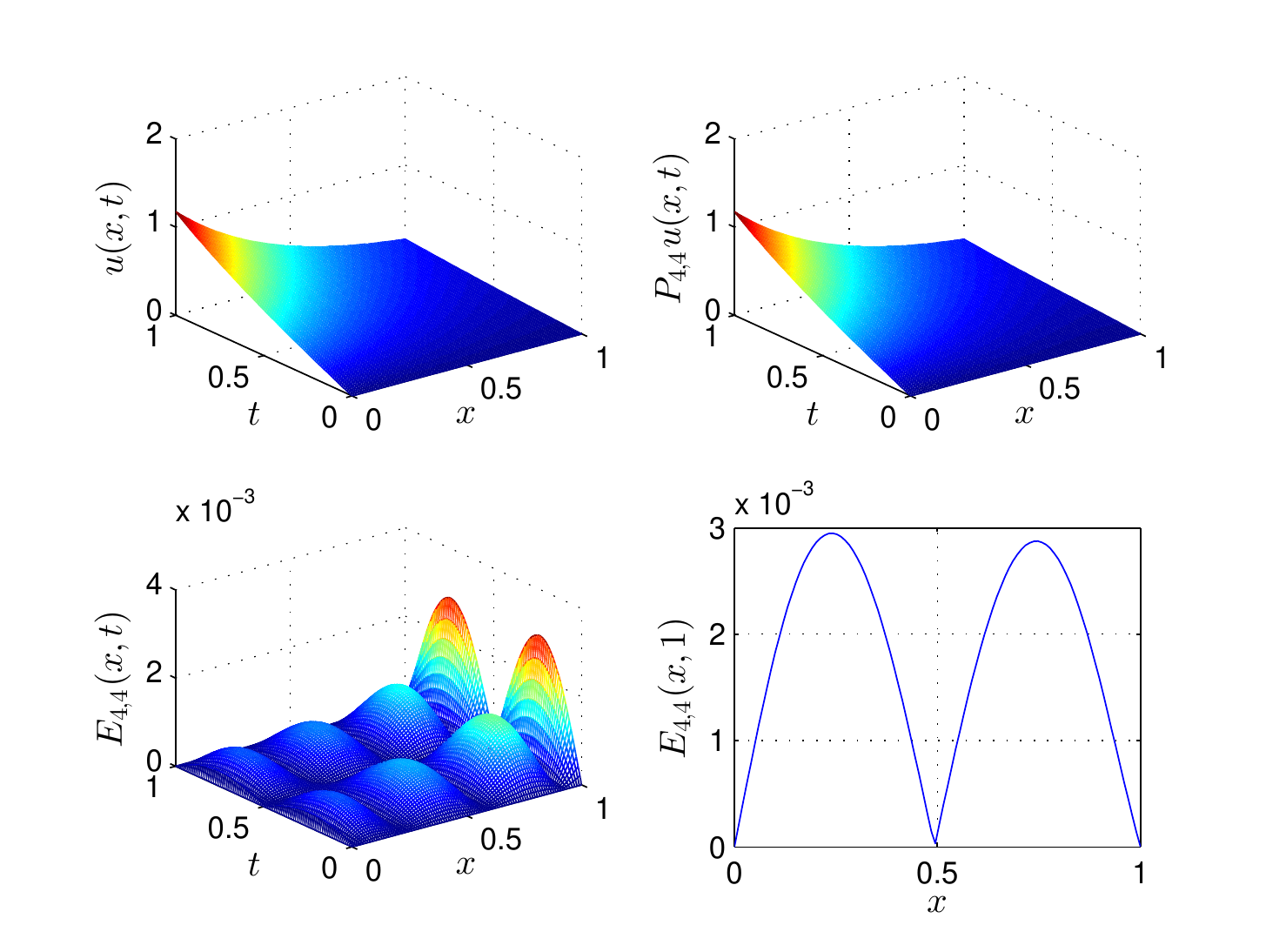}
\caption{The numerical simulation of the SGPM on Example 4. The figure shows the plots of the exact solution $u(x,t)$  on $D_{1,1}^2$ (upper left), its approximation $P_{4,4}u(x,t)$ (upper right), the absolute error function ${E}_{4,4}(x,t)$ (lower left), and its values at the final time, ${E}_{4,4}(x,1)$ (lower right). The optimal S-matrix was constructed using $M_t = 4$, and the plots were generated using $100$ linearly spaced nodes in the $x$- and $t$-directions from $0$ to $1$.}
\label{sec:numerical:fig:ex4exapprox4}
\end{figure}
\begin{table}[ht]
\begin{center} 
\scalebox{0.6}{
\resizebox{\textwidth}{!}{ %
\begin{tabular}{ccccc}
\toprule
\multicolumn{5}{c}{\textbf{Example 4}} \\
\cmidrule(r){1-5}
$N, M_t$ & $4, 4$ & $4,5$ & $4,6$ & $6,6$\\
\cmidrule(r){1-5}
${\left\| {\mathcal{E}} \right\|_1}$ & $ 4.027 \times 10^{-3}$ & $3.382 \times 10^{-3}$ & $2.807 \times 10^{-3}$ & $1.132 \times 10^{-5}$\\
${\left\| {\mathcal{E}} \right\|_2}$ & $ 3.178 \times 10^{-3}$ & $2.688 \times 10^{-3}$ & $2.218 \times 10^{-3}$ & $8.012 \times 10^{-6}$\\
${\left\| {\mathcal{E}} \right\|_\infty}$ & $ 4.053 \times 10^{-3}$ & $3.578 \times 10^{-3}$ & $2.933 \times 10^{-3}$ & $1.251 \times 10^{-5}$\\
$E_{N,N}^\infty$ & $ 1.834 \times 10^{-4}$ & $2.571 \times 10^{-3}$ & $2.087 \times 10^{-3}$ & $4.855 \times 10^{-6}$\\
AECPUT & $0.012$s & $0.013$s & $0.014$s & $0.029$s \\
\bottomrule
\end{tabular}
}}
\caption{The $l^1-, l^2-, l^{\infty}-$norms of the absolute error matrix ${\mathcal{E}}$, the $L^{\infty}$-norm of the absolute error, and the AECPUT of Example 4.}
\label{sec:numerical:tab:norm4}
\end{center}
\end{table}
\begin{table}[ht]
\begin{center} 
\scalebox{0.7}{
\resizebox{\textwidth}{!}{ %
\begin{tabular}{ccc}
\toprule
\multicolumn{3}{c}{\textbf{Example 4}} \\
\cmidrule(r){1-3}
\textbf{\citeauthor{Mohanty2005}'s method \cite{Mohanty2005}} & \textbf{The \citeauthor{Pandit2015} method \cite{Pandit2015}} & \textbf{Present method}\\
 $(2 M)/(\Delta\,t)/\text{RMS}$ & $(2 M)/(\Delta\,t)/\text{RMS}$ & $(N)/(M_t)/\text{RMS}$\\
\cmidrule(r){1-3}
$(16)/(1/16)/0.3603 \times 10^{-2}$ & $(16)/(1/16)/4.4971 \times 10^{-3}$ & $(4)/(4)/6.382 \times 10^{-4}$\\
$(32)/(1/32)/0.6834 \times 10^{-3}$ & $(32)/(1/32)/5.8771 \times 10^{-3}$ & $(4)/(5)/5.375 \times 10^{-4}$\\
$(64)/(1/64)/0.2420 \times 10^{-3}$ & $(64)/(1/64)/6.3473 \times 10^{-3}$ & $(4)/(6)/4.436 \times 10^{-4}$\\
\_ & \_ & $(6)/(6)/1.184 \times 10^{-6}$\\
\bottomrule
\end{tabular}
}}
\caption{A comparison of Example 4 between \citeauthor{Mohanty2005}'s finite difference method \cite{Mohanty2005} and the \citeauthor{Pandit2015} combined Crank-Nicolson finite difference and Haar wavelets numerical scheme \cite{Pandit2015}, and the SGPM. The first two columns of the table lists the results in the form $(2 M)/(\Delta\,t)/\text{RMS}$, and are exactly as quoted from Ref. \cite{Pandit2015}. The last column of the table lists the RMS errors of the SGPM associated with various values of $N$ and $M_t$.}
\label{sec:numerical:tab:norm44}
\end{center}
\end{table}
\vspace{-6pt}
\section{Future Work}
\label{sec:fw}
\vspace{-2pt}
The present SGPM assumes sufficient global smoothness of the solution, and generally uses single grids for discretization on the spatial and temporal domains. An interesting direction for future work could involve a study of composite shifted Gegenbauer grids and adaptivity to improve the convergence behavior of the numerical solver when dealing with nonsmooth problems. On the other hand, the numerical experiments conducted in Section \ref{sec:numerical} demonstrate the rapid convergence and stability of the SGPM; nonetheless, further stability analysis may be required to theoretically prove the stability of the SGPM on a wide variety of problems.

\vspace{-6pt}
\section{Conclusion}
\label{sec:conc}
\vspace{-2pt}
In this work, we developed a novel SGPM for the solution of the telegraph equation provided with some initial and boundary conditions. The method recasts the problem into its integral formulation to take advantage of the stability and well-conditioning of numerical integral operators. The discretization is carried out using some novel shifted Gegenbauer integration matrices and optimal shifted Gegenbauer integration matrices in the sense of solving the one-dimensional optimization problems \eqref{subsec:opt:reducedprob1}. We established Algorithm \ref{sec:theshi2:alg1matrix} for the efficient construction of the global collocation matrix and the right hand side of the resulting linear system, which together with a standard direct solver can produce very accurate approximations. The TCC of the developed algorithm scales like $O\left( {{N_x}{N_t}({N_x}{N_t} + {M_t})} \right),\text{ as } {N_x},{N_t},{M_t} \to \infty$. A numerical study on the time complexity required for the calculation of the approximate solution at the collocation points using a direct solver implementing Algorithm \ref{sec:theshi2:alg1matrix} shows that it scales like $O\left((L + 1)^{2}\right)$, as $L \to \infty$, for relatively small values of $M_t$, where $(L+1)$ is the total number of unknowns in the resulting linear system. Theorem \ref{sec:conerr1:thm:coeffconv1} and its corollaries demonstrate that the coefficients of the bivariate shifted Gegenbauer expansions decay faster for negative $\alpha$-values than for non-negative $\alpha$-values. In fact, we proved that the coefficients of the bivariate shifted Gegenbauer expansions are bounded for a non-positive Gegenbauer parameter, $\alpha \le 0$, as $n, m \to \infty$. Corollary \ref{cor:optcheb1} shows also that the asymptotic truncation error is minimized in the Chebyshev norm exactly when applying the shifted Chebyshev basis polynomials. Corollary \ref{cor:infiniteorder1} proves the exponential convergence exhibited by the SGPM. The extensive numerical results and comparisons demonstrate the fast execution, the exponential convergence, and the computational cost effectiveness of the proposed method. Moreover, the results show the ability of the numerical scheme to achieve higher-order approximations while preserving the same number of the solution expansion terms; thus the dimension of the resulting linear system of algebraic equations \eqref{sec:theshi2:eq:phisol1}. The method is memory minimizing, easily programmed, and can be efficiently applied and extended for the solution of various problems in many areas of science.

\vspace{-6pt}
\section{Acknowledgments}
\vspace{-2pt}
I would like to express my deepest gratitude to the editor for carefully handling the article, and the anonymous reviewers for their constructive comments and useful suggestions, which shaped the article into its final form.
\clearpage
\appendix
\section{A Computational Algorithm for the Constructions of the Global Collocation Matrix and the Right Hand Side of the Resulting Collocation Equations}
\label{appendix:comalg11}
\vspace{-2pt}
\begin{algorithm}[H]
\renewcommand{\thealgorithm}{A.1}
\caption{Matrix and Right Hand Side Constructions for Solving the Collocation Equations \eqref{sec:theshi2:eq:phisol1}}
\label{sec:theshi2:alg1matrix}
\begin{algorithmic}
\REQUIRE Positive integers $N_x, N_t$, and $M_t$; Positive real numbers $l$ and $\tau$; constants $\beta_1$ and $\beta_2$; SGG nodes $x_{l,N_x,i}^{(\alpha)}, i = 0, \ldots, N_x$; the S-matrices ${\mathbf{\hat P}}_l^{(2)}, {\mathbf{\hat P}}_{\tau}^{(1)}$, and ${\mathbf{\hat P}}_{\tau}^{(2)}$; the optimal S-matrices ${\mathbf{P}}_\tau^{(1)}$ and ${\mathbf{P}}_\tau^{(2)}$.\\
\vspace{2mm}
\STATE $L \leftarrow N_x + N_t + N_x N_t$
\FOR{$m=0$ \TO $L$}
\FOR{$n=0$ \TO $L$}
\STATE $A_{n,m} \leftarrow 0$
\ENDFOR
\ENDFOR
\FOR{$i=0$ \TO $N_x$}
  \STATE ${\theta _{l,x_{l,N_x,i}^{(\alpha)}}} \leftarrow x_{l,N_x,i}^{(\alpha)}/l$
    \FOR{$j=0$ \TO $N_t$}
      \STATE $n \leftarrow \text{index}(i,j)$
			\STATE $\mathbf{A}_{n,n} \leftarrow \left( {\hat p_{l,i,i}^{(2)} - {\theta _{l,x_{l,{N_x},i}^{(\alpha )}}}\,\hat p_{l,{N_x} + 1,i}^{(2)}} \right)\,\left( {{\beta _1}\,\hat p_{\tau ,j,j}^{(1)} + {\beta _2}\,\hat p_{\tau ,j,j}^{(2)}} + 1 \right) - \hat p_{\tau ,j,j}^{(2)}$
				\FOR{$k=0$ \TO $N_x$}
				\IF{$k \ne i$}
				  \STATE ${\mathbf{A}_{n,{\text{index}}\left( {k,j} \right)}} \leftarrow \left( {\hat p_{l,i,k}^{(2)} - {\theta _{l,x_{l,{N_x},i}^{(\alpha )}}}\,\hat p_{l,{N_x} + 1,k}^{(2)}} \right)\,\left( {{\beta _1}\,\hat p_{\tau ,j,j}^{(1)} + {\beta _2}\,\hat p_{\tau ,j,j}^{(2)} + 1} \right)$
				\ENDIF
				\ENDFOR			
				\FOR{$k=0$ \TO $N_t$}
				  \IF{$k \ne j$}
            \STATE ${\mathbf{A}_{n,{\text{index}}\left( {i,k} \right)}} \leftarrow \left( {\hat p_{l,i,i}^{(2)} - {\theta _{l,x_{l,{N_x},i}^{(\alpha )}}}\,\hat p_{l,{N_x} + 1,i}^{(2)}} \right)\,\left( {{\beta _1}\,\hat p_{\tau ,j,k}^{(1)} + {\beta _2}\,\hat p_{\tau ,j,k}^{(2)}} \right) - \hat p_{\tau ,j,k}^{(2)}$
					\FOR{$s=0$ \TO $N_x$}
					  \IF{$s \ne i$}
						 \STATE ${\mathbf{A}_{n,{\text{index}}\left( {s,k} \right)}} \leftarrow \left( {\hat p_{l,i,s}^{(2)} - {\theta _{l,x_{l,{N_x},i}^{(\alpha )}}}\hat p_{l,{N_x} + 1,s}^{(2)}} \right)\,\left( {{\beta _1}\hat p_{\tau ,j,k}^{(1)} + {\beta _2}\hat p_{\tau ,j,k}^{(2)}} \right)$
						\ENDIF
					\ENDFOR
					\ENDIF
        \ENDFOR
			\STATE ${\left( {{\text{RHS}}} \right)_n} \leftarrow {{\hat \psi }_{i,j}} - \left( {{\beta _1}{\mkern 1mu} \sum\limits_{k = 0}^{{M_t}} {p_{\tau ,j,k}^{(1)}\,{\psi _{i,(j,k)      }}}  + {\beta _2}{\mkern 1mu} \sum\limits_{k = 0}^{{M_t}} {p_{\tau ,j,k}^{(2)}\,{\psi _{i,(j,k)}}} } \right){\mkern 1mu}  + \sum\limits_{k = 0}^{{M_t}} {p_{\tau ,j,k}^      {(2)}\,{f_{i,(j,k)}}}$
    \ENDFOR
\ENDFOR
\RETURN{$\mathbf{A}$; $\{(\text{RHS})_n\}_{n=0}^L$}
\end{algorithmic}
\end{algorithm}

\bibliographystyle{model1-num-names}

\end{document}